\documentclass[a4paper,10pt]{amsart}
\usepackage{color}
\definecolor{Chocolat}{rgb}{0.173, 0.216, 0.230}
\definecolor{BleuTresFonce}{rgb}{0.215, 0.215, 0.36}

\usepackage[colorlinks,final,hyperindex]{hyperref}
\hypersetup{citecolor=BleuTresFonce, linkcolor=Chocolat}

\usepackage{graphicx}
\usepackage{epstopdf}
\usepackage[all]{xy}
\usepackage{amsmath,amsfonts,amssymb,MnSymbol}
\usepackage[centering,vscale=0.68,hscale=0.68]{geometry}
\usepackage{latexsym,amscd}
\usepackage{tikz}
\usetikzlibrary{arrows,decorations.markings,shapes.geometric}
\tikzset{>=latex}
\usepackage{url}
\usepackage{dsfont}
\usepackage{microtype}
\usepackage{todonotes}
\makeatletter
\providecommand\@dotsep{5}
\renewcommand{\listoftodos}[1][\@todonotes@todolistname]{%
  \@starttoc{tdo}{#1}}
\makeatother

\newtheorem*{theorem*}{Theorem}
\newtheorem{theorem}{Theorem}[section]
\newtheorem{lemma}[theorem]{Lemma}
\newtheorem{proposition}[theorem]{Proposition}
 
\theoremstyle{definition}  
\newtheorem{definition}[theorem]{Definition}
\newtheorem{example}[theorem]{Example}

\newtheorem{remark}[theorem]{Remark}

\newcommand{\Ad}{\operatorname{Ad}}

\newcommand{\on}{\operatorname}

\newcommand{\kk}{\mathbf{k}}
\newcommand{\Z}{{\mathbb Z}}
\newcommand{\Q}{{\mathbb Q}}
\newcommand{\C}{{\mathbb C}}

\newcommand{\KK}{\mathbf{k}}

\newcommand{\Pa}{\mathbf{Pa}}
\newcommand{\PaB}{\mathbf{PaB}}
\newcommand{\PaCD}{\mathbf{PaCD}}
\renewcommand{\CD}{\mathbf{CD}}

\renewcommand{\t}{{\mathfrak{t}}}

\newcommand{\n}{{\mathfrak{n}}}

\newcommand{\f}{{\mathfrak{f}}}

\renewcommand{\i}{\on{i}}

\newcommand{\Assoc}{\mathbf{Assoc}}
\newcommand{\Ass}{\mathbf{Assoc}}

\newcommand{\GT}{\mathbf{GT}}
\newcommand{\GRT}{\mathbf{GRT}}

\newcommand{\cM}{{\mathcal M}}

\tikzstyle cross=[preaction={draw=white, -, line width=4pt}, thick]
\tikzstyle normal=[thick]
\tikzstyle chord=[densely dotted, thick]
\tikzstyle zero=[ultra thick, gray]
\tikzstyle zell=[ultra thick, white]
\tikzstyle zerocross=[preaction={draw=white, -, line width=4pt}, ultra thick, gray]
\tikzstyle point=[draw,circle,inner sep=1,fill=black]
\tikzstyle grospoint=[draw,circle,inner sep=5,fill=black]
\tikzstyle diam=[draw,diamond,inner sep=1,fill=black]
\tikzstyle petitpoint=[draw,circle,inner sep=0.3,fill=black]

\newcommand{\straight}[3][-]{\draw[normal,#1] (#2,-#3) -- (#2,-#3-1);}

\newcommand{\fixch}[1]{\draw[zero] (0,-#1) -- (0,-#1-1);}

\newcommand{\hori}[4][-]{\draw[normal,#1] (#2,-#3)--(#2,-#3-1);\draw[normal,#1] (#2+#4,-#3)--(#2+#4,-#3-1);\draw[chord] (#2,-#3-0.5)--(#2+#4,-#3-0.5);}

\newcommand{\tzero}[3][-]{\draw[zero] (0,-#2)--(0,-#2-1);\draw[normal,#1] (#3,-#2)--(#3,-#2-1);\draw[chord] (0,-#2-0.5)--(#3,-#2-0.5);}

\newcommand{\tik}[1]{\begin{tikzpicture}[baseline=(current bounding box.center)] #1 \end{tikzpicture} }
\newcommand{\tmop}[1]{\ensuremath{\operatorname{#1}}}
\newcommand{\assign}{:=}
\newcommand{\tmtextit}[1]{{\itshape{#1}}}

\usepackage[toc,section=subsection]{glossaries}

\newglossary*{notation}{List of Notation}
\newglossary*{operads}{Operads}
\newglossary*{groups}{Groups}
\newglossary*{spaces}{Spaces}
\newglossary*{torsors}{Torsor sets}
\newglossary*{associators}{Series}

\makenoidxglossaries

\newglossaryentry{PaB}{type=operads, name=$\PaB$, 
  description={Operad of parenthesized braids} , sort={S}}

\newglossaryentry{PaB1}{type=operads, name=$\PaB^{1}$, 
  description={$\PaB$-moperad of parenthesized braids with a frozen strand}, sort={S}}  
  
\newglossaryentry{PaBg}{type=operads, name=$\PaB^{\Gamma}$, 
  description={$\PaB$-moperad of parenthesized cyclotomic braids}, sort={S}}

\newglossaryentry{PaCD}{type=operads, name=$\PaCD(\KK)$, 
  description={Operad of parenthesized chord diagrams}, sort={S}}

\newglossaryentry{PaCDg}{type=operads, name=$\PaCD^{\Gamma}(\KK)$, 
  description={$\PaCD(\KK)$-moperad of parenthesized $N$-chord diagrams}, sort={S}}
  
\newglossaryentry{GT}{type=groups, name=$\GT$, 
  description={Operadic Grothendieck-Teichm\"uller group}, sort={S}}

\newglossaryentry{kGTg}{type=groups, name=$\widehat{\GT}^{\Gamma}(\kk)$, 
  description={Operadic $\kk$-pro-unipotent cyclotomic Grothendieck-Teichm\"uller group}, sort={S}}
  
\newglossaryentry{GRT}{type=groups, name=$\GRT(\KK)$, 
  description={Operadic graded Grothendieck-Teichm\"uller group}, sort={S}}
  
\newglossaryentry{GRTg}{type=groups, name=$\GRT^{\Gamma}(\KK)$, 
  description={Operadic graded cyclotomic Grothendieck-Teichm\"uller group}, sort={S}}
  
\newglossaryentry{bGT}{type=groups, name=$\widehat{\on{GT}}(\kk)$, 
  description={$\kk$-pro-unipotent Grothendieck-Teichm\"uller group}, sort={S}}

\newglossaryentry{bGRT}{type=groups, name=$\on{GRT}(\KK)$, 
  description={Graded Grothendieck-Teichm\"uler group}, sort={S}}
  
\newglossaryentry{bGRTg}{type=groups, name=$\on{GRT}_{\bar 1}^{\Gamma}(\KK)$, 
  description={Graded $(\bar{1},N)$-cyclotomic Grothendieck-Teichm\"uller group}, sort={S}}  
  
\newglossaryentry{bkGTg}{type=groups, name=\ensuremath{\on{GTM}_{\bar 1}(N,\kk)}, 
  description={$\kk$-pro-unipotent $(\bar 1,N)$-cyclotomic Grothendieck-Teichm\"uller group}, sort={S}}
  
\newglossaryentry{Ass}{type=torsors, name=$\Ass(\KK)$, 
  description={Operadic $\KK$-associators}, sort={S}}

\newglossaryentry{Assg}{type=torsors, name=$\Ass^{\Gamma}(\KK)$, 
  description={Operadic cyclotomic $\KK$-associators}, sort={S}}
  
\newglossaryentry{bAss}{type=torsors, name=$\on{Ass}(\KK)$, 
  description={$\KK$-associators}, sort={S}}
  
\newglossaryentry{bAssg}{type=torsors, name=$\on{Ass}^{\Gamma}(\KK)$, 
  description={Cyclotomic $(\bar{1}, \KK)$-associators}, sort={S}}
 
\newglossaryentry{KZAss}{type=associators, name=$\Phi_{\on{KZ}}$, 
  description={KZ associator}, sort={S}}
  
\newglossaryentry{KZAssg}{type=associators, name=$\Psi_{\on{KZ}}$, 
  description={Cyclotomic KZ associator}, sort={S}}

\newglossaryentry{CCn}{type=spaces, name=\ensuremath{\on{C}(\C,I)}, 
  description={Reduced configuration space of $I$-indexed points in $\C$}, sort={S}}
  
\newglossaryentry{FMCn}{type=spaces, name=\ensuremath{\bar{\on{C}}(\C,I)}, 
  description={Fulton-MacPherson compactification of \ensuremath{\on{C}(\C,I)}}, sort={S}}
  
\newglossaryentry{CCI1}{type=spaces, name=\ensuremath{\on{C}(\C^{\times},I)}, 
  description={Reduced configuration space of $I$-indexed points in $\C^{\times}$}, sort={S}}
  
\newglossaryentry{CCIM}{type=spaces, name=\ensuremath{\on{C}(\C^{\times},I,\Gamma)}, 
  description={Reduced $\Gamma$-decorated configuration space of $I$-indexed points in $\C^{\times}$}, sort={S}}

\title{A moperadic approach to cyclotomic associators}
\author{Damien Calaque and Martin Gonzalez}
\address{Damien CALAQUE \newline \indent IMAG, Univ Montpellier, 
CNRS, Montpellier, France}
\email{damien.calaque@umontpellier.fr}
\address{Martin GONZALEZ \newline \indent Institut de Recherche Technologique SystemX, Palaiseau, France.}
\email{martin.gonzalez@irt-systemx.fr}

\begin{document}

\begin{abstract}
%This is a companion paper to \textit{Ellipsitomic associators} \cite{CaGo2}. 
We provide a (m)operadic description of Enriquez's torsor of cyclotomic associators, 
as well as of its associated cyclotomic Grothendieck--Teichm\"uller groups. 
\end{abstract}

\maketitle

\setcounter{tocdepth}{1}

\tableofcontents

\section*{Introduction}

Since the introduction of associators and Grothendieck--Teichm\"uller groups by Drinfeld \cite{DrGal}, 
several variations of these have been considered; for instance 
\begin{itemize}
\item following Grothendieck's esquisse \cite{Esquisse}, Lochak--Nakamura--Schneps defined a new version of the 
Grothendieck--Teichm\"uller group \cite{LNS-new}, which acts on more general surface mapping class groups than 
Drinfled's original one;
\item cyclotomic \cite{En} and elliptic \cite{En2} variants of associators and Grothendieck--Teichm\"uller groups 
have been defined by Enriquez; 
\item ellipsitomic associators, which share both the features of cyclotomic and elliptic associators, have recently 
been introduced by the authors \cite{CaGo2}. 
\end{itemize}

\medskip

It is known, after the original insight of Drinfeld \cite{DrGal} and Bar-Natan \cite{BN}, and thanks to the detailed 
proof of Fresse \cite{Fresse}, that the torsor of associators can be understood as the torsor of isomorphisms between two 
operads in groupoids. A similar result holds for Enriquez's torsor of elliptic associators as well, as was recently proven 
in \cite{CaGo2}, where one has to consider operadic modules instead of just operads. This need comes from the fact that, 
while compactified configuration spaces of points in the plane form an operad, compactified configuration spaces of points 
in a torus form an operadic module on the latter. In \cite{CaGo2}, ellipsitomic associators are defined as 
operadic module isomorphisms, and the description \textit{\`a la Drinfeld} is derived from it afterwards. 

\medskip

In this companion paper to \cite{CaGo2}, we prove that Enriquez's cyclotomic associators torsor (resp.~Grothendieck--Teichm\"uller groups) can also be identified with isomorphisms (resp.~automorphisms) of operadic gadgets. The appropriate notion here is the one 
of a moperad; it was introduced by Willwacher \cite{Will16}, and it typically encodes the structure of compactified configuration 
spaces of points in the punctured plane (or, equivalently, the annulus). 

\medskip

After two reminders, on moperads (Section \ref{A short reminder on operads}) on the one hand, and associators 
(Section \ref{Operads associated with configuration spaces (associators)}) on the other hand, we introduce in Section 
\ref{Frozen} the moperad $\PaB^1$ of parenthesized braids with a frozen strand 
(obtained as the fundamental groupoid of the configuration moperad of points in the punctured plane) and provide 
a generators and relations presentation for $\PaB^1$:
\begin{theorem*}[Theorem \ref{PaB1}]
The moperad in groupoid $\PaB^1$ is generated by an arity $1$ arrow $E$ and an arity $2$ arrow $\Psi$, with relations 
\eqref{eqn:cU}, \eqref{eqn:MP}, \eqref{eqn:RP}, and \eqref{eqn:O}. 
\end{theorem*}
Unsurprisingly, these relations are completely analogous to the axioms for braided module categories from \cite{Adrien}; 
indeed, one can verify that a braided module category is nothing but a representation of $\PaB^1$ in categories. 
In Section \ref{Cyclo}, we decorate the unfrozen strands of our parenthesized braids with elements from a finite quotient 
$\Gamma=\mathbb{Z}/N\mathbb{Z}$ of the fundamental group of the punctured plane, giving rise to a moperad in groupoids 
$\PaB^\Gamma$. We show that $\PaB^\Gamma$ admits a presentation by generators and relations similar to the one of $\PaB^1$ 
(Theorem \ref{thm-gamma-pres}), and thus identify the group of $\Gamma$-equivariant automorphisms of $\PaB^\Gamma$ that 
are the identity on objects with Enriquez's cyclotomic Grothendieck--Teichm\"uller group (Proposition \ref{Cyc:GT}). 
Finally, in Section \ref{Assocyclo}, we put a moperad structure on the (parenthesized) horizontal $N$-chord diagrams of 
\cite{Adrien}, and prove the following
\begin{theorem*}[Theorem \ref{Ass:cyc:iso}]
The set of $\Gamma$-equivariant moperad isomorphisms that are the identity on objects between $\PaB^\Gamma$ 
and parenthesized $N$-chord diagrams is in bijection with Enriquez's cyclotomic associators.  
\end{theorem*}
We moreover show that this identification respects the (bi)torsor structures (Theorem \ref{thm-last-torsor}). 

\subsection*{Acknowledgements}
We are deeply grateful to Adrien Brochier for numerous conversations and suggestions. Discussions with Benjamin Enriquez 
have also been very helpful. 
The first author has received funding from the Institut Universitaire de France, and from the European Research Council 
(ERC) under the European Union's Horizon 2020 research and innovation programme (Grant Agreement No. 768679). 
This paper is extracted from the second author's PhD thesis \cite{G1} at Sorbonne Universit\'e, and part of this 
work has been done while the second author was visiting the Institut Montpelli\'erain Alexander Grothendieck,
thanks to the financial support of the Institut Universitaire de France. The second author warmly thanks the 
Max-Planck Institute for Mathematics in Bonn and Université d'Aix-Marseille, for their hospitality and excellent working conditions.

\subsection*{Convention}
All along the paper, $\kk$ is a field of characteristic zero. 

We also warn the reader that we use all along the paper the following rather unusual convention for arrows in a groupoid, and more 
generally in a category: we often concatenate arrows rather than composing them. In other words, $f_1f_2 = f_2 \circ f_1$. 

\section{Moperads}
\label{A short reminder on operads}

In this section we fix a symmetric monoidal category $(\mathcal C,\otimes,\mathbf{1})$ 
having small colimits and such that $\otimes$ commutes with these. 
We borrow the notation and convention for $\mathfrak{S}$-modules and 
operads from \cite{CaGo2}.

% 1.1

\subsection{Moperads over an operad}\label{sec-mop}

Let ${\mathcal O}$ be an operad. A \textit{moperad} over $\mathcal O$ is an $\mathfrak{S}$-module ${\mathcal P}$ carrying 
\begin{itemize}
\item a unital monoid structure for the monoidal product $\otimes$, 
\item and a right $\mathcal O$-module structure for the monoidal product $\circ$, 
that are compatible in the following sense: 
\begin{itemize}
\item One first observes that, for every $\mathfrak{S}$-module $\mathcal Q$, there is a 
natural morphism $\mathcal Q\otimes(\mathcal P\circ\mathcal O)\to(\mathcal Q\otimes\mathcal P)\circ\mathcal O$. 
\item Then the compatibility means that the following diagram commutes: 
\[
\xymatrix{
\mathcal P\otimes(\mathcal P\circ\mathcal O) \ar[r] \ar[d] & \mathcal P \otimes\mathcal P \ar[rd] & \\
(\mathcal P\otimes\mathcal P)\circ\mathcal O \ar[r] & \mathcal O\circ\mathcal P \ar[r] & \mathcal P
}
\]
\end{itemize}
\end{itemize}
The map $\mathcal P\otimes(\mathcal P\circ\mathcal O)\to \mathcal P$ 
one obtains decomposes into maps 
$$
\mathcal P(k)\otimes \mathcal P(m_0)\otimes\mathcal O(m_1)\otimes\cdots
\otimes\mathcal O(m_k)\to\mathcal P(m_0+\cdots+m_k)
$$
satisfying certain associativity, unit and $\mathfrak{S}$-equivariance relations. 
We let the reader spell out these conditions explicitely.

Within the symmetric monoidal category of differential graded vector spaces, this 
definition coincides with Willwacher's one from \cite{Will16} (from which we borrowed 
the name ``moperad''). Note that the monoid structure for the monoidal product 
$\otimes$ encodes precisely the partial composition with respect to the second colour. 
We will denote this partial composition by $\circ_0$. 

% 1.2

\subsection{Example of a moperad over an operad: coloured Stasheff polytopes}

To any finite set $I$ we associate the configuration space 
\[
\textrm{Conf}(\mathbb{R}_{>0},I)=\{\mathbf{x}=(x_i)_{i\in I}\in 
(\mathbb{R}_{>0})^I|x_i\neq x_j\textrm{ if }i\neq j\}
\]
and its reduced version 
\[
\textrm{C}(\mathbb{R}_{>0},I):=\textrm{Conf}(\mathbb{R}_{>0},I)/\mathbb{R}_{>0}\,.
\]
The Axelrod--Singer--Fulton--MacPherson compactification 
$\overline{\textrm{C}}(\mathbb{R}_{>0},I)$ 
of $\textrm{C}(\mathbb{R}_{>0},I)$ is a disjoint union of $|I|$-th Stasheff 
polytopes with two kinds 
of colours, indexed by $\mathfrak S_I$. The boundary 
\[
\partial\overline{\textrm{C}}(\mathbb{R}_{>0},I):=\overline{\textrm{C}}
(\mathbb{R}_{>0},I)-\textrm{C}(\mathbb{R}_{>0},I)
\]
is the union, over all partitions $I=J_0\coprod J_1\coprod\cdots\coprod J_k$, of 
$$
\partial_{J_0,\cdots,J_k}\overline{\textrm{C}}(\mathbb{R}_{>0},I):=\overline{\rm C}(\mathbb{R}_{>0},k) 
\times \overline{\rm C}(\mathbb{R}_{>0},J_0)\times\prod_{i=1}^k\overline{\rm C}(\mathbb{R},J_i)\,.
$$
The inclusion of boundary components provides $\overline{\rm C}(\mathbb{R}_{>0},-)$ with the structure 
of a $\overline{\rm C}(\mathbb{R},-)$-moperad in topological spaces. 

One can see that $\overline{\textrm{C}}(\mathbb{R}_{>0},I)$ is a manifold with corners, and that considering 
only zero-dimensional strata of our configuration spaces we get a sub-moperad 
$\mathbf{Pa_0}\subset\overline{\textrm{C}}(\mathbb{R}_{>0},-)$ that can be briefly described as follows: 
\begin{itemize}
\item $\mathbf{Pa_0}(I)$ is the set of pairs $(\sigma,p)$, where $\sigma$ is a linear order on $I$ and $p$ 
a maximal parenthesization of $\circ\!\!\!\underbrace{\bullet\cdots\bullet}_{|I|~{\rm times}}$. 
\item The $\overline{\rm C}(\mathbb{R},-)$-moperad structure is given by substitution as above. 
\end{itemize}
Forgetting the $\overline{\rm C}(\mathbb{R},-)$-moperad structure on $\overline{\textrm{C}}(\mathbb{R}_{>0},-)$ 
and considering a $\overline{\rm C}(\mathbb{R},-)$-module structure on it amounts to forbidding points to 
be close to the origin (i.e. the 0-strand cannot be inside a parenthesis in this case). 

% 1.3

\subsection{Pointing}\label{sec-pointings}

Recall the operad $Unit$ defined by 
$$
Unit(n):=\begin{cases}
\mathbf{1} &\mathrm{if}~n=0,1 \\
\emptyset &\mathrm{else}
\end{cases}
$$
By convention, all our operads $\mathcal O$ will be pointed in the sense that they will come equipped with 
a specific operad morphism $Unit\to\mathcal O$. 
Morphisms of operads are required to be compatible with this pointing. Actually, all operads appearing in 
this paper are such that $\mathcal O(n)\simeq\mathbf{1}$ if $n=0,1$. 

\medskip

Similarly, we intoduce the moperad $MUnit$ over $Unit$, which is such that $MUnit(n)=\mathbf{1}$ for all 
$n\geq0$. By convention, all our moperads will be pointed, in the sense that they will come equipped with 
a specific $Unit$-moperad morphism $MUnit\to \mathcal Q$. Morphisms of moperads are required to be 
compatible with the pointing. 
\begin{remark}
In the category of sets, $MUnit$ is the sub-$Unit$-moperad of $\mathbf{Pa_0}$ that consists only of the 
left-most maximal parenthesization. 
\end{remark}

\medskip

The main reason for this rather strange convention is that we need the following features, that we 
have in the case of compactified configuration spaces: 
\begin{itemize}
\item For operads and moperads, we want to have ``deleting operations'' $\mathcal O(n)\to\mathcal O(n-1)$ 
that decrease arity. 
\item For a moperad, we want to be able to ``see the operad inside'' it, i.e.\ we want to have a 
distinguished morphism $\mathcal O\to\mathcal P$ of $\mathfrak{S}$-modules. 
\end{itemize}

\begin{example}
For instance, being a $\mathbf{Pa}$-moperad, $\mathbf{Pa_0}$ comes equipped with a morphism of 
$\mathfrak{S}$-modules $\mathbf{Pa}\to\mathbf{Pa_0}$. 
It sends a parenthesized permutation $\mathbf{p}$ to $\circ(\mathbf{p})$. 
\end{example}

% 1.4

\subsection{Group actions}\label{sec-gpaction}

Let $G$ be a group and $\mathcal O$ be an operad. We say that an $\mathcal O$-module $\mathcal P$ 
carries a $G$-action if 
\begin{itemize}
\item for every $n\geq0$, $G^n$ acts $\mathfrak{S}_n$-equivariantly on $\mathcal P(n)$, from the left. 
\item for every $m\geq0$, $n\geq0$, and $1\leq i\leq n$, the partial composition 
$$
\circ_i:\mathcal P(n)\otimes\mathcal O(m)\longrightarrow \mathcal P(n+m-1)
$$
is equivariant along the group morphism 
\begin{eqnarray*}
G^n & \longrightarrow & G^{n+m-1} \\
(g_1,\dots, g_n) & \longmapsto & (g_1,\dots, g_{i-1},
\underbrace{g_i,\dots,g_i}_{m~\mathrm{times}},g_{i+1},\dots,g_n)
\end{eqnarray*}
\end{itemize}

If $\mathcal P$ is a moperad, we additionally require that, for every $m\geq0$, and $n\geq0$, the partial composition 
$$
\circ_0:\mathcal P(n)\otimes\mathcal P(m)\longrightarrow \mathcal P(n+m)
$$
is $G^{n+m}$-equivariant. 

\medskip

A morphism $\mathcal P\to\mathcal Q$ of $\mathcal O$-moperads with $G$-action is said to be 
\textit{$G$-equivariant} if, for every $n\geq0$, the map $\mathcal P(n)\to\mathcal Q(n)$ 
is $G^n$-equivariant. 

\section{Reminders on associators and G(R)T}
\label{Operads associated with configuration spaces (associators)}

In this section we recollect some results from \cite{DrGal,BN,Fresse}, following essentially 
the presentation of \cite[Section 2]{CaGo2}. 

% 2.1

\subsection{Compactified configuration space of the plane}

To any finite set $I$ we associate the (reduced) configuration space 
\[
\text{\gls{CCn}}:=
\{\mathbf{z}=(z_i)_{i\in I}\in \mathbb{C}^I|z_i\neq z_j\textrm{ if }i\neq j\}/\mathbb{C}\rtimes\mathbb{R}_{>0}\,
\]
of points in the plane. 
We then consider its Axelrod--Singer--Fulton--MacPherson compactification \gls{FMCn}, whose boundary 
\[
\partial\overline{\textrm{C}}(\mathbb{C},I)=\overline{\textrm{C}}(\mathbb{C},I)-\textrm{C}(\mathbb{C},I)
\]
is made of irreducible components $\partial_{J_1,\cdots,J_k}\overline{\textrm{C}}(\mathbb{C},I)$ indexed 
by partitions $I=J_1\coprod\cdots\coprod J_k$ of $I$: 
\[
\partial_{J_1,\cdots,J_k}\overline{\textrm{C}}(\mathbb{C},I)\cong \overline{\rm C}(\mathbb{C},k)\times
\prod_{i=1}^k\overline{\rm C}(\mathbb{C},J_i)\,.
\]
The inclusion of boundary components provides $\overline{\rm C}(\mathbb{C},-)$ with the structure of an 
operad in topological spaces. 

% 2.2

\subsection{The operad of parenthesized braids}\label{sec-pab}

The inclusions of topological operads 
$$
\mathbf{Pa}\,\subset\,\overline{\textrm{C}}(\mathbb{R},-)\,\subset\,\overline{\textrm{C}}(\mathbb{C},-)\,
$$
allow us to define 
$$
\text{\gls{PaB}}:=\pi_1\left(\overline{\textrm{C}}(\mathbb{C},-),\mathbf{Pa}\right)\,,
$$
which is an operad in groupoids. 
\begin{example}[of arrows in small arity]\label{ex2.1}
Recall from \cite[Examples 2.1]{CaGo2} that, in arity two, there is an arrow $R^{1,2}$ going from $(12)$ to $(21)$, that can be depicted 
in the following ways: 
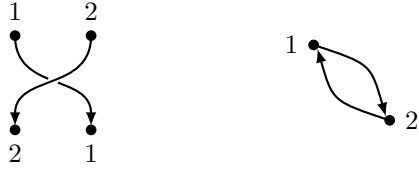
\begin{center}
\begin{tikzpicture}[baseline=(current bounding box.center)]
\tikzstyle point=[circle, fill=black, inner sep=0.05cm]
 \node[point, label=above:$1$] at (1,1) {};
 \node[point, label=below:$2$] at (1,-0.25) {};
 \node[point, label=above:$2$] at (2,1) {};
 \node[point, label=below:$1$] at (2,-0.25) {};
  \draw[->,thick] (1,1) .. controls (1,0.25) and (2,0.5).. (2,-0.20);
 \node[point, ,white] at (1.5,0.4) {};
 \draw[->,thick] (2,1) .. controls (2,0.25) and (1,0.5).. (1,-0.20); 
\end{tikzpicture}
\qquad\qquad\qquad
\begin{tikzpicture}[baseline=(current bounding box.center)]
\tikzstyle point=[circle, fill=black, inner sep=0.05cm]
 \node[point, label=right:$2$] at (0.5,-0.5) {};
 \node[point, label=left:$1$] at (-0.5,0.5) {};
 \draw[->,thick] (0.5,-0.5) .. controls (-0.25,-0.25) .. (-0.45,0.45); 
 \draw[->,thick] (-0.5,0.5) .. controls (0.25,0.25) .. (0.45,-0.45); 
\end{tikzpicture}
\end{center}
There is another arrow $\tilde R^{1,2}:=(R^{2,1})^{-1}$, having the same source and target, that can 
be depicted as an undercrossing braid. 

In arity three, there is an arrow $\Phi^{1,2,3}$, going from $(12)3$ to $1(23)$, that can be depicted 
in the following ways: 
\begin{center}
\begin{tikzpicture}[baseline=(current bounding box.center)]
\tikzstyle point=[circle, fill=black, inner sep=0.05cm]
 \node[point, label=above:$(1$] at (1,1) {};
 \node[point, label=below:$1$] at (1,-0.25) {};
 \node[point, label=above:$2)$] at (1.5,1) {};
 \node[point, label=below:$(2$] at (3.5,-0.25) {};
 \node[point, label=above:$3$] at (4,1) {};
 \node[point, label=below:$3)$] at (4,-0.25) {};
 \draw[->,thick] (1,1) .. controls (1,0) and (1,0).. (1,-0.20); 
 \draw[->,thick] (1.5,1) .. controls (1.5,0.25) and (3.5,0.5).. (3.5,-0.20);
 \draw[->,thick] (4,1) .. controls (4,0) and (4,0).. (4,-0.20);
\end{tikzpicture}
\qquad\qquad\qquad
\begin{tikzpicture}[baseline=(current bounding box.center)] 
\tikzstyle point=[circle, fill=black, inner sep=0.05cm]
 \node[point, label=below:$1$] at (0.5,0.5) {};
 \node[point, label=below:$2$] at (1,0.5) {};
 \node[point, label=below:$3$] at (2.5,0.5) {};
 \draw[->,thick] (1,0.5) .. controls (1,0.5) .. (2,0.5); 
\end{tikzpicture}
\end{center}
\end{example}
Grothendieck's insight \cite{Esquisse} about the (genus $0$) (Grothendieck--)Teichm\"uller tower, 
later proven by Drinfeld \cite{DrGal}, can be understood as a generators and relations presentation for $\PaB$. 
This was made more explicit by Bar-Natan \cite{BN} in a different language, and written in term of operads by Tamarkin \cite{Tam} and Fresse 
\cite[Theorem 6.2.4]{Fresse}. Following the convention and notation from \cite[Section 2]{CaGo2}, it reads as follows
\begin{theorem}\label{thm2.3}
As an operad in groupoids having $\mathbf{Pa}$ as operad of objects, $\mathbf{PaB}$ is generated by 
$R:=R^{1,2}$ and $\Phi:=\Phi^{1,2,3}$ together with the following relations: 
\begin{flalign}
& \Phi^{1,\emptyset,2}=\on{Id}_{1,2}									\quad 
\Big(\mathrm{in}~\mathrm{Hom}_{\mathbf{PaB}(2)}\big(12,12\big)\Big)\,, 				\tag{U}\label{eqn:U} \\
& R^{1,2}\Phi^{2,1,3}R^{1,3}=\Phi^{1,2,3}R^{1,23}\Phi^{2,3,1} 											\quad 
\Big(\mathrm{in}~\mathrm{Hom}_{\mathbf{PaB}(3)}\big((12)3,2(31)\big)\Big)\,, 				\tag{H1}\label{eqn:H1} \\
& \tilde R^{1,2}\Phi^{2,1,3}\tilde R^{1,3}=\Phi^{1,2,3}\tilde R^{1,23}\Phi^{2,3,1} 											\quad 
\Big(\mathrm{in}~\mathrm{Hom}_{\mathbf{PaB}(3)}\big((12)3,2(31)\big)\Big)\,, 					\tag{H2}\label{eqn:H2} \\
& \Phi^{12,3,4}\Phi^{1,2,34}=\Phi^{1,2,3}\Phi^{1,23,4}\Phi^{2,3,4} 									\quad
\Big(\mathrm{in}~\mathrm{Hom}_{\mathbf{PaB}(4)}\big(((12)3)4,1(2(34))\big)\Big)\,. &	\tag{P}\label{eqn:P}
\end{flalign}
\end{theorem}

In order to fix our braid group convention we recall the following. The automorphism group $\on{Aut}_{\PaB(n)}(\mathbf{p})$ of any parenthesized permutation $\mathbf{p}$ of lenght $n$ is the pure braid group $\on{PB}_n$ on $n$ strands, which is generated by elementary pure braids $x_{ij}$, $1\leq i<j\leq n$, 
which satisfy a certain list of relations (see \cite{CaGo2} for more details).
In this article we will depict the generator $x_{ij}$ in the following two equivalent ways: 
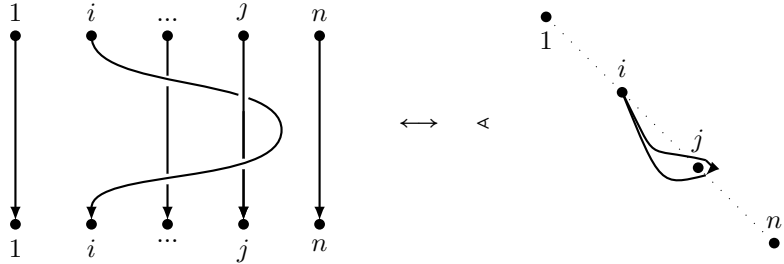
\begin{center}
\begin{tikzpicture}[baseline=(current bounding box.center)]
\tikzstyle point=[circle, fill=black, inner sep=0.05cm]
	\node[point, label=above:$1$] at (0,1) {};
	\node[point, label=below:$1$] at (0,-1.5) {};
	\draw[->,thick, postaction={decorate}] (0,1) -- (0,-1.45);
	\node[point, label=above:$i$] at (1,1) {};
	\node[point, label=below:$i$] at (1,-1.5) {};
	\node[point, label=above:$...$] at (2,1) {};
	\node[point, label=below:$...$] at (2,-1.5) {};
	\draw[->,thick, postaction={decorate}] (2,1) -- (2,-1.45);
	\node[point, ,white] at (2,0.4) {};
	\node[point, ,white] at (2,-0.9) {};
	\node[point, label=above:$j$] at (3,1) {};
	\node[point, label=below:$j$] at (3,-1.5) {};
	\draw[thick] (1,1) .. controls (1,0.25) and (3.5,0.5).. (3.5,-0.25);
	\node[point, ,white] at (3,0.2) {};
	\draw[->,thick, postaction={decorate}] (3,1) -- (3,-1.45);
	\draw[->,thick, postaction={decorate}] (3,0) -- (3,-1.45);
	\node[point, ,white] at (3,-0.7) {};
	\draw[->,thick, postaction={decorate}] (3.5,-0.25) .. controls (3.5,-1) and (1,-0.75).. (1,-1.45);
	\node[point, label=above:$n$] at (4,1) {};
	\node[point, label=below:$n$] at (4,-1.5) {};
	\draw[->,thick, postaction={decorate}] (4,1) -- (4,-1.45);
\end{tikzpicture}
$\qquad\longleftrightarrow\quad\sphericalangle\quad$
\begin{tikzpicture}[baseline=(current bounding box.center)] 
\tikzstyle point=[circle, fill=black, inner sep=0.05cm]
	\draw[loosely dotted] (-1.5,1.5) -- (1.5,-1.5); 
	\node[point, label=below:1] at (-1.5,1.5) {}; 
	\node[point, label=above:$i$] at (-0.5,0.5) {};
	\node[point, label=above:$j$] at (0.5,-0.5) {};
	\node[point, label=above:$n$] at (1.5,-1.5) {};
	\draw[thick, postaction={decorate}] (-0.5,0.5) .. controls (0,-0.5) and (-0.25,-0.2).. (0.6,-0.4); 
	\draw[->,thick] (0.6,-0.4) .. controls (0.7,-0.5)  .. (0.6,-0.6) ;
	\draw[thick, postaction={decorate}] (0.6,-0.6) .. controls (0,-0.75) .. (-0.5,0.5);
\end{tikzpicture}
\end{center}
The group $\on{PB}_n$ is the kernel of the map $\on{B}_n \to \mathfrak{S}_n$ sending, for all $1\leq i \leq n-1$, the generators $\sigma_i$ of $\on{B}_n$ to the permutation $(i,i+1)$. The elements $\sigma_i$ are depicted in the same way as the R's:
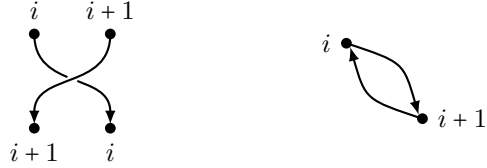
\begin{center}
\begin{tikzpicture}[baseline=(current bounding box.center)]
\tikzstyle point=[circle, fill=black, inner sep=0.05cm]
 \node[point, label=above:$i$] at (1,1) {};
 \node[point, label=below:$i+1$] at (1,-0.25) {};
 \node[point, label=above:$i+1$] at (2,1) {};
 \node[point, label=below:$i$] at (2,-0.25) {};
  \draw[->,thick] (1,1) .. controls (1,0.25) and (2,0.5).. (2,-0.20);
 \node[point, ,white] at (1.5,0.4) {};
 \draw[->,thick] (2,1) .. controls (2,0.25) and (1,0.5).. (1,-0.20); 
\end{tikzpicture}
\qquad\qquad\qquad
\begin{tikzpicture}[baseline=(current bounding box.center)]
\tikzstyle point=[circle, fill=black, inner sep=0.05cm]
 \node[point, label=right:$i+1$] at (0.5,-0.5) {};
 \node[point, label=left:$i$] at (-0.5,0.5) {};
 \draw[->,thick] (0.5,-0.5) .. controls (-0.25,-0.25) .. (-0.45,0.45); 
 \draw[->,thick] (-0.5,0.5) .. controls (0.25,0.25) .. (0.45,-0.45); 
\end{tikzpicture}
\end{center}

% 2.3

\subsection{The operad of (parenthesized) chord diagrams}\label{sec-pacd}

Recall \cite{Tam,Fresse} that the collection of Kohno--Drinfeld Lie $\kk$-algebras $\t_n(\KK)$ 
is provided with the structure of an operad in the category $grLie_\KK$ of positively graded Lie algebras over $\kk$, with symmetric monoidal structure given by the direct sum $\oplus$. 
This is equivalent to Bar-Natan's cabling operation \cite{BN} on chord diagrams. 

\medskip

Taking the degree completion of the universal enveloping algebra functor, we get an operad 
$\CD(\KK) := \hat{\mathcal{U}}(\t({\kk}))$ in complete filtered cocommutative Hopf algebras 
which we view as categories (with only one object) enriched in complete filtered cocommutative coalgebras, 
that we call the \textit{operad of chord diagrams}. 

\medskip

By definition, the operad $\on{Ob}(\CD(\kk))$ of object of $\CD(\kk)$ is the terminal operad in sets. 
One can thus define the operad $\text{\gls{PaCD}}$, that is called the \textit{operad of parenthesized chord diagrams}, as 
the \textit{fake pull-back} of $\CD(\KK)$ along the terminal morphism $\Pa\to *=\on{Ob}(\CD(\kk))$. We refer to 
\cite{CaGo2} for the definition of fake pull-back; it is enough to know that $\PaCD(\kk)$ has $\Pa$ as operad 
of objects, and that in arity $n$ the complete filtered cocommutative coalgebra of morphisms between any pair 
of objects is always $\hat{\mathcal{U}}(\t_n({\kk}))$. 

% 2.4

\subsection{Drinfeld associators}\label{sec:2.8assoc}

Recall that
\begin{itemize}
\item there is a symmetric monoidal functor $G:\bf{Cat(CoAlg_{\kk})}\to \mathbf{Grpd}_\kk$ from 
categories enriched in complete filtered cocommutative $\kk$-coalgebras to $\kk$-prounipotent groupoids. 
\item for $\mathcal C$ being $\mathbf{Grpd}$, $\mathbf{Grpd}_\kk$, or $\bf{Cat(CoAlg_{\kk})}$ (see \cite{CaGo2}), the notation 
$$
\on{Aut}_{\on{Op}\mathcal C}^+\quad(\mathrm{resp.}~\on{Iso}_{\tmop{Op}\mathcal C}^+)
$$ 
refers to those automorphisms (resp.~isomorphisms) which are the identity on objects within 
the category $\on{Op}\mathcal C$ of operads in $\mathcal C$. 
\end{itemize}

\medskip

A \textit{Drinfeld $\KK$-associator} is an isomorphism between the operads $\widehat{\PaB}(\KK)$ and 
$G \PaCD(\KK)$ in $\mathbf{Grpd}_{\kk}$, which is the identity on objects. We denote by
$$
\text{\gls{Ass}}:=\on{Iso}^+_{\on{Op}\mathbf{{Grpd}_{\kk}}}(\widehat{\PaB}(\KK),G \PaCD(\KK))
$$
the set of $\KK$-associators.
Drinfeld showed in \cite{DrGal} (see also \cite{BN,Fresse}) that there is a one-to-one correspondence between the 
set of Drinfeld $\kk$-associators and the set $\text{\gls{bAss}}$ of pairs 
$(\mu,\varphi)\in\kk^\times \times \on{exp}(\hat\f_2(\kk))$, 
such that 
\begin{itemize}
\item $\varphi^{3,2,1}=(\varphi^{1,2,3})^{-1} \quad $ in $\on{exp}(\hat\t_{3}(\kk))$,
\item $\varphi^{1,2,3}e^{\mu t_{23}/2}\varphi^{2,3,1}e^{\mu t_{31}/2}\varphi^{3,1,2}e^{\mu t_{12}/2}
=e^{\mu(t_{12}+t_{13}+t_{23})/2} \quad $ in $\on{exp}(\hat\t_{3}(\kk))$,
\item $\varphi^{1,2,3}\varphi^{1,23,4}\varphi^{2,3,4}=
\varphi^{12,3,4}\varphi^{1,2,34} \quad$ in $\on{exp}(\hat\t_{4}(\kk))$,
\end{itemize}
where $\varphi^{1,2,3}=\varphi(t_{12},t_{23})$ is viewed as an element of $\on{exp}(\hat\t_{3}(\kk))$ via 
the inclusion $\hat\f_2(\kk)\subset \hat\t_3(\kk)$ sending $x$ to $t_{12}$ and $y$ to $t_{23}$. 
The proof of this result relies on the universal property of $\PaB$ from Theorem \ref{thm2.3}. 
In particular, a morphism $F:\widehat{\PaB}(\KK) \longrightarrow G \PaCD(\KK)$ is uniquely determined 
by the pair $(\mu,\varphi)\in\kk\times\on{exp}(\hat\f_2(\kk))$ given by $F(R^{1,2})=e^{\mu t_{12}/2} X^{1,2}$ and
$F(\Phi^{1,2,3})=\varphi(t_{12},t_{23}) a^{1,2,3}$. 

An example of such an associator is the KZ associator $\Phi_{\on{KZ}}$.
It is defined as the the renormalized holonomy
from $0$ to $1$ of $G'(z) = (\frac{t_{12}}{z} + \frac{t_{23}}{z-1})G(z)$, 
i.e., $\text{\gls{KZAss}} := G_{0^{+}}G_{1^{-}}^{-1}\in\on{exp}(\hat\t_{3}(\C))$, where 
$G_{0^{+}},G_{1^{-}}$ are the solutions such that $G_{0^{+}}(z)\sim z^{t_{12}}$ 
when $z\to 0^{+}$ and $G_{1^{-}}(z)\sim (1-z)^{t_{23}}$ when $z\to 1^{-}$. We have that 
$(2\pi\i,\Phi_{\on{KZ}})$ is an element of $\on{Ass}(\C)$.

% 2.5

\subsection{The Grothendieck--Teichm\"uller group}\label{subsec-GT}

The \textit{Grothendieck--Teichm\"uller group} is defined as the group 
$$
\text{\gls{GT}}:= \on{Aut}_{\on{Op} \mathbf{Grpd}}^{+}(\PaB)
$$
of automorphisms of the operad in groupoids $\PaB$ which are the identity of objects and 
its $\kk$-pro-unipotent version is
$$
\widehat{\GT}(\kk):= \on{Aut}_{\on{Op}\mathbf{Grpd}_\kk}^{+}\big(\widehat{\PaB}(\kk)\big).
$$
In this article we will focus on the $\kk$-pro-unipotent version of this group in the cyclotomic situation.
The group $\widehat{\GT}(\KK)$ is isomorphic to Drinfeld's Grothendieck--Teichm\"uller group $\text{\gls{bGT}}$ 
consisting of pairs 
$$
(\lambda,f) \in \KK^{\times} \times\widehat{\on{F}}_2(\KK)
$$
which satisfy the following equations:
\begin{itemize}
\item $f(x,y)=f(y,x)^{-1}$ in $\widehat{\on{F}}_2(\KK)$, 
\item $x_1^{\nu}f(x_1,x_2)x_2^{\nu}f(x_2,x_3)x_3^{\nu}f(x_3,x_1)=1$ in $\widehat{\on{F}}_2(\KK)$, 
\item $f(x_{13}x_{23}, x_{34})f(x_{12}, x_{23}x_{24}) =f(x_{12}, x_{23}) 
f(x_{12}x_{13}, x_{24}x_{34})f(x_{23}, x_{34})$ in $\widehat{\on{PB}}_4(\KK)$,
\end{itemize}
where $x_1,x_2,x_3$ are 3 variables subject only to $x_1x_2x_3=1$, $\nu=\frac{\lambda-1}{2}$,
 and $x_{ij}$ is the elementary pure braid from Subsection \ref{sec-pab}.
The multiplication law is given by
\begin{equation*}\label{GT:LCI}
(\lambda_1, f_1)(\lambda_2, f_2)=
(\lambda_1\lambda_2, 
 f_1(x^{\lambda_2},f_2(x, y)y^{\lambda_2} f_2(x, y)^{-1}) f_2(x, y)) .
\end{equation*}

One obtains the pair $(\lambda,f)$ from an automorphism $F\in \widehat{\GT}(\KK)$, 
for $\lambda=2\nu+1$ by the following assignement: $F(R^{1,2})=(R^{1,2}R^{2,1})^\nu R^{1,2}$ 
and $F(\Phi^{1,2,3})=  f(x_{12},x_{23}) \Phi^{1,2,3}$.

% 2.6

\subsection{The graded Grothendieck--Teichm\"uller group}\label{subsec-GRT}

The graded Grothendieck--Teichm\"uller group is the group 
$$
\text{\gls{GRT}}:=\on{Aut}_{\on{Op}\mathbf{Grpd}_{\kk}}^+(G\PaCD(\KK))
$$ 
of automorphisms of $G\PaCD(\KK)$ that are the identity on objects.

Again, the operadic definition of $\GRT(\kk)$ coincides with the one originaly given by Drinfeld. 
Denote by $\on{GRT}_1$ the set of elements in 
$g\in\on{exp}(\hat\f_2(\kk))\subset \on{exp}(\hat\t_3(\kk))$ such that 
\begin{itemize}
\item $g^{3,2,1}=g^{-1}$ and $g^{1,2,3}g^{2,3,1}g^{3,1,2}=1$, in $\on{exp}(\hat\t_{3}(\kk))$,
%\item $t_{12}+\on{Ad}(g^{1,2,3})(t_{23})+ \on{Ad}(g^{2,1,3})(t_{13}) = t_{12}+t_{13}+t_{23}$, in $\hat\t_{3}(\kk)$, 
\item $g^{1,2,3}g^{1,23,4}g^{2,3,4}=g^{12,3,4}g^{1,2,34}$, in $\on{exp}(\hat\t_{4}(\kk))$,
\end{itemize}
One has the following multiplication law on $\on{GRT}_1$: 
$$
(g_1*g_2)(t_{12},t_{23}):= g_1(t_{12},\on{Ad}(g_2(t_{12},t_{23}))(t_{23}))g_2(t_{12},t_{23})\,.
$$
Drinfeld showed in \cite{DrGal} that the above $\on{GRT}_1$ is stable under $*$, that it defines 
a group structure on it, and that rescaling transformations $g(x,y)\mapsto \lambda\cdot g(x,y)=
g(\lambda x,\lambda y)$ define an action of $\kk^\times$ of $\on{GRT}_1$ by automorphisms. 
We denote by $\on{GRT}(\kk)$ the corresponding semi-direct product. 

Then, as was shown in \cite{Fresse}, the group $\GRT(\kk)$ is isomorphic to $\on{GRT}(\kk)$. 
One obtains the pair $(\lambda,g)$ from an automorphism $G\in\GRT(\kk)$ in the following way: 
$G(X^{1,2})=X^{1,2}$, $G(H^{1,2})=e^{\lambda t_{12}} H^{1,2}$, and $G(a^{1,2,3})= g(t_{12},t_{23}) a^{1,2,3}$.

% 2.7

\subsection{Bitorsor structure}

Recall first that there is a free and transitive left action of $\widehat{\on{GT}}(\kk)$ on $\on{Ass}(\kk)$, 
defined, for $(\lambda,f)\in\widehat{\on{GT}}(\kk)$ and $(\mu,\varphi)\in \on{Ass}(\kk)$, by
$$
((\lambda,f)*(\mu,\varphi))(t_{12},t_{23}):= (\lambda\mu,f(e^{\mu t_{12}},\on{Ad}(\varphi(t_{12},t_{23}))
(e^{\mu t_{23}}))\varphi(t_{12},t_{23}))\,,
$$
where $\on{Ad}(f)(g):=fgf^{-1}$, for any symbols $f,g$.

Recall also that there is a free and transitive right action of $\on{GRT}(\kk)$ on $\on{Ass}(\kk)$ defined 
as follows: for $(\lambda,g)\in \on{GRT}(\kk)$ and $(\mu,\varphi)
\in \on{Ass}(\kk)$, given by
$$
((\mu,\varphi)*(\lambda,g))(t_{12},t_{23}):= (\lambda\mu,\varphi(\lambda t_{12},\on{Ad}(g)(\lambda t_{23}))
g(t_{12},t_{23}) )\,.
$$

These two action commute with each other, and turn 
$\big(\widehat{\on{GT}}(\kk),\on{Ass}(\kk),\on{GRT}(\kk)\big)$ into a bitorsor. 
By its very definition, the triple $\big(\widehat{\GT}(\kk),\Assoc(\kk),\GRT(\kk)\big)$ is also a bitorsor, 
and it is proven 
in \cite{Fresse} that the above identifications from subsections \ref{sec:2.8assoc}, \ref{subsec-GT}, 
and \ref{subsec-GRT} provides a bitorsor isomorphism
\begin{equation}\label{bitorsor:cl}
\big(\widehat{\GT}(\kk),\Assoc(\kk),\GRT(\kk)\big) \longrightarrow \big(\widehat{\on{GT}}(\kk),
\on{Ass}(\kk),\on{GRT}(\kk)\big)\,.
\end{equation}

\section{Parenthesized braids with a frozen strand}
\label{Frozen}

% 3.1

\subsection{Compactified configuration space of the annulus}

For each finite set $I$, consider the (reduced) configuration space of
$\mathbb{C}^{\times}$:
\[
\text{\gls{CCI1}} \assign
\left\{ \mathbf{z} = (z_i)_{i \in I}\in (\mathbb{\C^{\times}})^I |z_i \neq z_j, \forall i \neq j \right\}
/ \mathbb{R}_{> 0}\,. 
\]
We clearly have an isomorphism between $\mathrm{C} (\mathbb{C^{\times}},n)$ and 
$\mathrm{C}  (\mathbb{C}, n + 1)$. We then consider the Axelrod--Singer--Fulton--MacPherson 
compactification $\overline{\mathrm{C}}(\mathbb{C^{\times}}, n)$ of $\mathrm{C} (\mathbb{C^{\times}}, n)$. 
The boundary 
\[
\partial \overline{\mathrm{C}} (\mathbb{\mathbb{C^{\times}}}, n)
= \overline{\mathrm{C}} (\mathbb{C^{\times}}, n) - \mathrm{C}(\mathbb{\mathbb{C^{\times}}}, n)
\]
is made of the following irreducible components: for any partition 
$[\![ 0, n ]\!] = J_0 \coprod \cdots\coprod J_k$ such that $0 \in J_m$ for some $0 \leq m \leq k$, 
there is a component
\[
 \partial_{J_0, \cdots, J_k}  \overline{\mathrm{C}} (\mathbb{C^{\times}},
   n) \cong  \overline{\mathrm{C}} (\mathbb{C^{\times}}, k) \times  \overline{\mathrm{C}} 
   (\mathbb{C^{\times}}, J_m) \times  \prod_{i = 0 ; i \neq m}^k \overline{\mathrm{C}} (\mathbb{C}, J_i)\,. 
\]

\begin{remark}\label{Lien avec Marcy and co}
One should be aware that the identification of the boundary with a union of such components depends on a choice of framing (up to isotopy). Framings on $\mathbb{C}^\times$ or classified (up to isotopy) by their rotation number. This is similar to what happens in \cite[Remark 3.6]{BZBJ}, where there is a choice of rotation number to make in the definition of a braided module category. There is no suprise here: braided module categories are algebras in categories over the moperad of parenthezised braids with a frozen strand. 

In this paper, we choose the framing with rotation number $1$, also know as the cylinder framing. The main reason for this choice is that it is the only one compatible with the action by multiplication of $S^1$ on $\mathbb{C}^\times$; we will use this action of $\mu_N\subset S^1$ in the next section to define the moperad of cyclotomic parenthezised braids. 

Note that the recent work \cite{DHLARS} uses the framing with rotation number $0$, also known as the blackboard framing. This leads to a \textbf{different} moperad of 
parenthezised braids with a frozen strand. 
\end{remark}

% 3.2

\subsection{The $\mathbf{PaB}$-moperad of parenthesized braids with a frozen strand} 

We have inclusions of topological moperads
$$
\mathbf{Pa}_{0}\,\subset\,\overline{\textrm{C}}(\mathbb{R}_{>0},-)\,\subset\,
\overline{\textrm{C}}(\mathbb{C}^{\times},-)\,
$$
over 
$$
\mathbf{Pa}\,\subset\,\overline{\textrm{C}}(\mathbb{R},-)\,\subset\,\overline{\textrm{C}}
(\mathbb{C},-)\,.
$$
We then define 
$$
\text{\gls{PaB1}}:=\pi_1\left(\overline{\textrm{C}}(\mathbb{C^{\times}},-),\mathbf{Pa}_{0}\right)\,,
$$
which is a moperad over the operad in groupoids $\mathbf{PaB}$.
\begin{example}[Description of $\PaB^1(1)$]\label{ex3.1}
First, observe that $\overline{\textrm{C}}(\mathbb{C}^{\times},1)\simeq \overline{\textrm{C}}
(\mathbb{C},2)\simeq S^1$. 
Moreover, $\mathbf{Pa}_{0}=\{(01)\}$. Hence $\PaB^1(1)\simeq\mathbb{Z}$: 
it has only one object $(01)$ and is freely generated by an automorphism $E^{0,1}$ of $(01)$, 
which can be depicted as an elementary pure braid: 
\begin{center}
\begin{tikzpicture}[baseline=(current bounding box.center)]
\tikzstyle point=[circle, fill=black, inner sep=0.05cm]% style des points
\tikzstyle point2=[circle, fill=black, inner sep=0.08cm]% style des points
\node[point, label=above:$0$] at (1,1) {};
\node[point, label=above:$1$] at (1.5,1) {};
\node[point, label=below:$0$] at (1,0) {};
\node[point, label=below:$1$] at (1.5,0) {};
\draw[->,thick] (0.75,0.5) .. controls (0.75,0.35) and (1.5,0.35) .. (1.5,0.05);
\node[point, ,white] at (1,0.37) {};
\draw[->,thick] (1,1) .. controls (1,0.5) .. (1,0.05);
\node[point, ,white] at (1,0.7) {};
\draw[-,thick] (1.5,1) .. controls (1.5,0.75) and (0.75,0.75) .. (0.75,0.5); 
\end{tikzpicture}
\qquad\qquad\qquad
\begin{tikzpicture}[baseline=(current bounding box.center)]
\tikzstyle point=[circle, fill=black, inner sep=0.05cm]
 \node[point, label=left:$0$] at (0,0) {};
 \node[point, label=right:$1$] at (0.5,0) {};
 \draw[->,thick] (0.5,0) .. controls (-1,-1) and (-1,1) .. (0.5,0); 
\end{tikzpicture}
\\ \text{$E^{0,1}$ from two different angles}
\end{center}
\end{example}
\begin{example}[Notable arrow in $\PaB^1(2)$]\label{ex3.2}
Let us first recall that $\mathbf{Pa_0}(2)=\mathfrak{S}_2\times\{(\bullet\bullet)
\bullet,\bullet(\bullet\bullet)\}$ and that 
$\overline{\textrm{C}}(\mathbb{R}_{>0},2)\cong\mathfrak{S}_2\times[0,1]$. 
Hence we have an arrow $\Psi^{0,1,2}$ (the identity path in $[0,1]$) from 
$(01)2$ to $0(12)$ in $\PaB^1(2)$, which can be depicted as follows: 
\begin{center}
\begin{tikzpicture}[baseline=(current bounding box.center)]
\tikzstyle point=[circle, fill=black, inner sep=0.05cm]
 \node[point, label=above:$(0$] at (1,1) {};
 \node[point, label=below:$0$] at (1,-0.25) {};
 \node[point, label=above:$1)$] at (1.5,1) {};
 \node[point, label=below:$(1$] at (3.5,-0.25) {};
 \node[point, label=above:$2$] at (4,1) {};
 \node[point, label=below:$2)$] at (4,-0.25) {};
 \draw[->,thick] (1,1) .. controls (1,0) and (1,0).. (1,-0.20); 
 \draw[->,thick] (1.5,1) .. controls (1.5,0.25) and (3.5,0.5).. (3.5,-0.20);
 \draw[->,thick] (4,1) .. controls (4,0) and (4,0).. (4,-0.20);
\end{tikzpicture}
\qquad\qquad\qquad
\begin{tikzpicture}[baseline=(current bounding box.center)] 
\tikzstyle point=[circle, fill=black, inner sep=0.05cm]
 \node[point, label=below:$0$] at (0.5,0.5) {};
 \node[point, label=below:$1$] at (1,0.5) {};
 \node[point, label=below:$2$] at (2.5,0.5) {};
 \draw[->,thick] (1,0.5) .. controls (1,0.5) .. (2,0.5); 
\end{tikzpicture}
\\ \text{$\Psi^{0,1,2}$ from two different angles}
\end{center}
\end{example}

\begin{remark}\label{rem4.3-poiting}
Recall from \S\ref{sec-pointings} that, being a $\PaB$-moperad, $\PaB^1$ comes 
equipped with a morphism of $\mathfrak{S}$-modules $\PaB\to\PaB^1$. 
In pictorial terms, this morphism sends a parentesized braid with $n$ strands to a 
parenthesized braid with $n+1$ strands by adding a frozen strand labelled by $0$ on the left. 
For instance, the images of $R^{1,2}$ (a morphism in $\PaB(2)$) and of 
$\Phi^{1,2,3}$ (a morphism in $\PaB(3)$) can be respectively depicted as follows: 
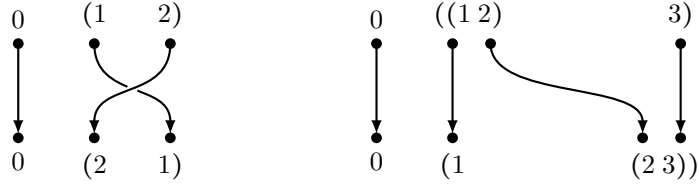
\begin{center}
\begin{tikzpicture}[baseline=(current bounding box.center)]
\tikzstyle point=[circle, fill=black, inner sep=0.05cm]
 \node[point, label=above:$0$] at (0,1) {};
 \node[point, label=below:$0$] at (0,-0.25) {};
 \node[point, label=above:$(1$] at (1,1) {};
 \node[point, label=below:$(2$] at (1,-0.25) {};
 \node[point, label=above:$2)$] at (2,1) {};
 \node[point, label=below:$1)$] at (2,-0.25) {};
 \draw[->,thick] (0,1) to (0,-0.20) ;
  \draw[->,thick] (1,1) .. controls (1,0.25) and (2,0.5).. (2,-0.20);
 \node[point, ,white] at (1.5,0.4) {};
 \draw[->,thick] (2,1) .. controls (2,0.25) and (1,0.5).. (1,-0.20); 
\end{tikzpicture}
\qquad\qquad\qquad
\begin{tikzpicture}[baseline=(current bounding box.center)]
\tikzstyle point=[circle, fill=black, inner sep=0.05cm]
 \node[point, label=above:$0$] at (0,1) {};
 \node[point, label=below:$0$] at (0,-0.25) {};
 \node[point, label=above:$((1$] at (1,1) {};
 \node[point, label=below:$(1$] at (1,-0.25) {};
 \node[point, label=above:$2)$] at (1.5,1) {};
 \node[point, label=below:$(2$] at (3.5,-0.25) {};
 \node[point, label=above:$3)$] at (4,1) {};
 \node[point, label=below:$3))$] at (4,-0.25) {};
 \draw[->,thick] (0,1) to (0,-0.20) ;
 \draw[->,thick] (1,1) .. controls (1,0) and (1,0).. (1,-0.20); 
 \draw[->,thick] (1.5,1) .. controls (1.5,0.25) and (3.5,0.5).. (3.5,-0.20);
 \draw[->,thick] (4,1) .. controls (4,0) and (4,0).. (4,-0.20);
\end{tikzpicture}
\end{center}
We will still denote these images by $R^{1,2}$ and $\Phi^{1,2,3}$.
\end{remark}

% 3.3

\subsection{The universal property of $\PaB^1$}

Our main goal in this section is to prove the following generators and relations presentation of $\PaB^1$. 
\begin{theorem}\label{PaB1}
As a $\PaB$-moperad having $\mathbf{Pa_0}$ as $\mathbf{Pa}$-moperad of objects, $\PaB^1$ is 
generated by $E:=E^{0,1}\in \PaB^1(1)$ and $\Psi:=\Psi^{0,1,2} \in \PaB^1(2)$ together with the 
following relations: 
\begin{flalign}
 & \Psi^{0,\emptyset,1}=\on{Id}_{01}
 \quad\Big(\mathrm{in}~\mathrm{Hom}_{\PaB^1(1)}\big(01,01\big)\Big)\,,  \tag{cU}\label{eqn:cU} \\
 & \Psi^{01,2,3} \Psi^{0,1,23} = \Psi^{0,1,2} \Psi^{0,12,3} \Phi^{1,2,3}
 \quad\Big(\mathrm{in}~\mathrm{Hom}_{\PaB^1(3)}\big(((01)2)3,0(1(23))\big)\Big)\,, \tag{M}\label{eqn:MP} \\
 & \Psi^{0,1,2}E^{0,12}( \Psi^{0,1,2})^{-1}=E^{0,1} E^{01,2}
 \quad\Big(\mathrm{in}~\mathrm{Hom}_{\PaB^1(2)}\big((01)2,(01)2)\big)\Big)\,,  \tag{R}\label{eqn:RP} \\ 
 & E^{01,2} = \Psi^{0,1,2} R^{1,2} (\Psi^{0,2,1})^{-1} E^{0,2} \Psi^{0,2,1} R^{2,1} (\Psi^{0,1,2})^{-1}
 \quad\Big(\mathrm{in}~\mathrm{Hom}_{\PaB^1(2)}\big((01)2,(01)2\big)\Big)\,. \tag{O}\label{eqn:O} 
\end{flalign}
\end{theorem}

\begin{remark}\label{Lien avec Marcy and co 2}
Observe that our relation \eqref{eqn:RP} is slightly different from the relation (RP) from \cite[Theorem 1.24]{DHLARS}. 
This is explained by a different choice of framing on $\mathbb{C}^\times$ (see Remark \ref{Lien avec Marcy and co}). 
\end{remark}

\begin{proof}
Let $\mathcal Q^1$ be the $\PaB$-moperad with the above presentation. 
From Examples \ref{ex3.1} and \ref{ex3.2} we deduce that, as a $\PaB$-moperad in groupoid, 
$\mathbf{PaB}^{1}$ contains two morphisms $E^{0,1}$ (in $\PaB^1(1)$) and $\Psi^{0,1,2}$ (in $\PaB^1(2)$). 
One easily shows, using the following pictures, that they satisfy the \textit{mixed pentagon} and 
\textit{octogon} relations, \eqref{eqn:MP} and \eqref{eqn:O}, and relation \eqref{eqn:RP}: 
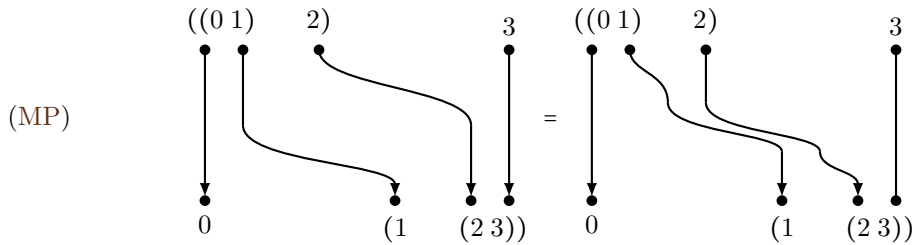
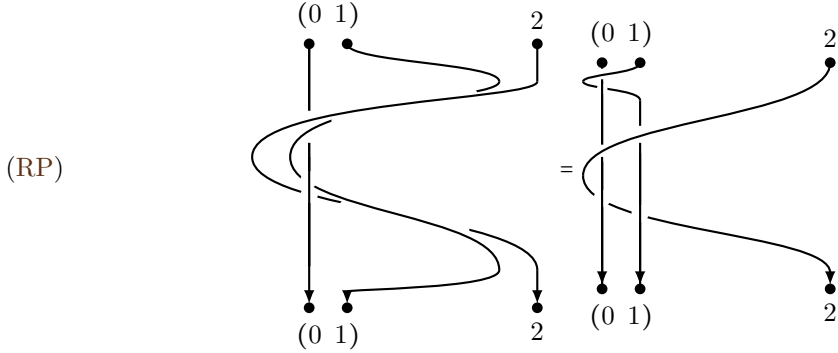
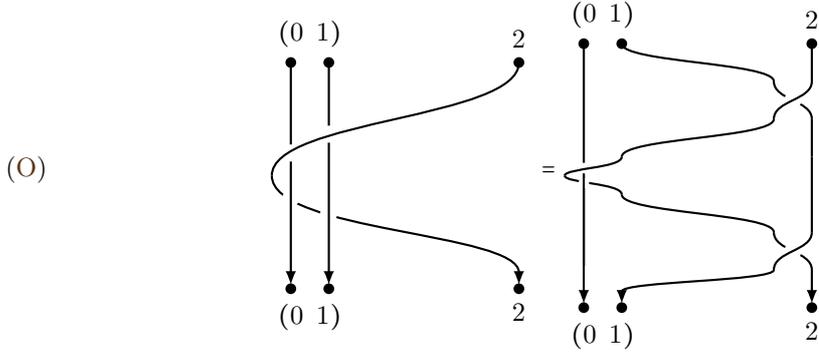
\begin{center}
\begin{align}\tag{\ref{eqn:MP}}
\begin{tikzpicture}[baseline=(current bounding box.center)] 
\tikzstyle point=[circle, fill=black, inner sep=0.05cm]
\tikzstyle point2=[circle, fill=black, inner sep=0.08cm]
\node[point, label=above:$((0$] at (1,2) {};
\node[point, label=above:$1)$] at (1.5,2) {};
\node[point, label=above:$2)$] at (2.5,2) {};
\node[point, label=above:$3$] at (5,2) {};
\node[point, label=below:$0$] at (1,0) {};
\node[point, label=below:$(1$] at (3.5,0) {};
\node[point, label=below:$(2$] at (4.5,0) {};
\node[point, label=below:$3))$] at (5,0) {};
 \draw[->,thick] (1,2) .. controls (1,0.5) .. (1,0.05);
 \draw[-,thick] (1.5,2) .. controls (1.5,1.5) .. (1.5,1);
 \draw[-,thick] (2.5,2) .. controls (2.5,1.5) and (4.5,1.5).. (4.5,1);
 \draw[->,thick] (1.5,1) .. controls (1.5,0.5) and (3.5,0.5).. (3.5,0.05);
 \draw[->,thick] (4.5,1) .. controls (4.5,0.5) .. (4.5,0.05);
 \draw[->,thick] (5,2) .. controls (5,0.5) .. (5,0.05);
\end{tikzpicture}
=
\begin{tikzpicture}[baseline=(current bounding box.center)] 
\tikzstyle point=[circle, fill=black, inner sep=0.05cm]
\tikzstyle point2=[circle, fill=black, inner sep=0.08cm]
\node[point, label=above:$((0$] at (1,2) {};
\node[point, label=above:$1)$] at (1.5,2) {};
\node[point, label=above:$2)$] at (2.5,2) {};
\node[point, label=above:$3$] at (5,2) {};
\node[point, label=below:$0$] at (1,0) {};
\node[point, label=below:$(1$] at (3.5,0) {};
\node[point, label=below:$(2$] at (4.5,0) {};
\node[point, label=below:$3))$] at (5,0) {};
 \draw[->,thick] (1,2) .. controls (1,0.5) .. (1,0.05);
 \draw[-,thick] (2.5,2) .. controls (2.5,1.5) .. (2.5,1.3);
 \draw[-,thick] (1.5,2) .. controls (1.5,1.7) and (2,1.7).. (2,1.3);
 \draw[-,thick] (2,1.3) .. controls (2,0.9) and (3.5,0.9).. (3.5,0.65);
 \draw[-,thick] (2.5,1.3) .. controls (2.5,0.9) and (4,0.9).. (4,0.65);
 \draw[->,thick] (4,0.65) .. controls (4,0.4) and (4.5,0.4).. (4.5,0.05);
 \draw[->,thick] (3.5,0.65) .. controls (3.5,0.5) .. (3.5,0.05);
 \draw[-,thick] (5,2) .. controls (5,0.5) .. (5,0.05);
\end{tikzpicture}
\end{align}
\end{center}
\begin{center}
\begin{align}\tag{\ref{eqn:RP}}
\begin{tikzpicture}[baseline=(current bounding box.center)]
\tikzstyle point=[circle, fill=black, inner sep=0.05cm]
\tikzstyle point2=[circle, fill=black, inner sep=0.08cm]
\tikzstyle point3=[circle, fill=black, inner sep=0.5cm]
 \node[point, label=above:$(0$] at (1,2) {};
 \node[point, label=below:$(0$] at (1,-1.5) {};
 \node[point, label=above:$1)$] at (1.5,2) {};
 \node[point, label=below:$1)$] at (1.5,-1.5) {};
 \node[point, label=above:$2$] at (4,2) {};
 \node[point, label=below:$2$] at (4,-1.5) {};
\draw[->,thick] (1,2) .. controls (1,1) and (1,1).. (1,-1.45); 
\draw[-,thick] (1.5,2) .. controls (1.5,1.75) and (3.5,1.75).. (3.5,1.5);
 \draw[-,thick] (1,2) .. controls (1,0.5) .. (1,-0.95);
 \node[point2, ,white] at (1,1) {};
 \node[point2, ,white] at (1,0.8) {};
   \draw[-,thick] (0.25,0.5) .. controls (0.25,-0.25) and (4,-0.25) .. (4,-1); 
  \node[point3, ,white] at (2.4,-0.5) {}; 
  \node[point3, ,white] at (2.2,-0.75) {};
  \node[point3, ,white] at (2,-0.5) {};
  \draw[-,thick] (0.75,0.5) .. controls (0.75,-0.25) and (3.5,-0.25) .. (3.5,-1); 
     \node[point2, ,white] at (1,0) {};
 \node[point2, ,white] at (1,0.2) {};
 \draw[-,thick] (3.5,1.5) .. controls (3.5,1.25) and (0.75,1.25) .. (0.75,0.5); 
  \node[point3, ,white] at (2.2,1) {};
  \node[point3, ,white] at (2.6,1) {};
  \node[point3, ,white] at (2,1) {};
\draw[-,thick] (4,1.5) .. controls (4,1.25) and (0.25,1.25) .. (0.25,0.5); 
\draw[->,thick] (3.5,-1) .. controls (3.5,-1.25) and (1.5,-1.25).. (1.5,-1.45);
\draw[-,thick] (4,2) .. controls (4,1.5) and (4,1.5).. (4,1.5); 
\draw[->,thick] (4,-1) .. controls (4,-1) and (4,-1).. (4,-1.45); 
 \draw[-,thick] (1,0.5) .. controls (1,0.5) .. (1,-0.95);
\end{tikzpicture}
=
\begin{tikzpicture}[baseline=(current bounding box.center)]
\tikzstyle point=[circle, fill=black, inner sep=0.05cm]
\tikzstyle point2=[circle, fill=black, inner sep=0.08cm]% style des points
 \node[point, label=above:$(0$] at (1,2) {};
 \node[point, label=below:$(0$] at (1,-1) {};
 \node[point, label=above:$1)$] at (1.5,2) {};
 \node[point, label=below:$1)$] at (1.5,-1) {};
 \node[point, label=above:$2$] at (4,2) {};
 \node[point, label=below:$2$] at (4,-1) {};
  \draw[-,thick] (1,2) .. controls (1,0.5) .. (1,-0.95);
 \node[point, ,white] at (1,0.35+1.5) {};
 \draw[-,thick] (1.5,0.5+1.5) .. controls (1.5,0.35+1.5) and (0.75,0.35+1.5) .. (0.75,0.25+1.5); 
 \draw[-,thick] (0.75,0.25+1.5) .. controls (0.75,0.15+1.5) and (1.5,0.15+1.5) .. (1.5,0+1.5); 
 \node[point, ,white] at (1,0.15+1.5) {};
 \draw[->,thick] (0.75,0.5) .. controls (0.75,-0.25) and (4,-0.25) .. (4,-0.95); 
  \draw[->,thick] (1,1.5) .. controls (1,0.5) .. (1,-0.95);
 \draw[->,thick] (1.5,1.5) .. controls (1.5,0.5) .. (1.5,-0.95);
 \node[point2, ,white] at (1,0.8) {};
 \node[point2, ,white] at (1.5,1.05) {};
  \draw[-,thick] (4,2) .. controls (4,1.25) and (0.75,1.25) .. (0.75,0.5); 
 \node[point2, ,white] at (1.5,0) {};
 \node[point2, ,white] at (1,0.2) {};
   \draw[-,thick] (1,0.75) .. controls (1,0.5) .. (1,-0.95);
      \draw[-,thick] (1,1.75) .. controls (1,1.5) .. (1,1.5);
     \draw[-,thick] (1.5,0.8) .. controls (1.5,0.5) .. (1.5,-0.95); 
\end{tikzpicture} 
\end{align}
\end{center}
and
\begin{center}
\begin{align}\tag{\ref{eqn:O}}
\begin{tikzpicture}[baseline=(current bounding box.center)]
\tikzstyle point=[circle, fill=black, inner sep=0.05cm]
\tikzstyle point2=[circle, fill=black, inner sep=0.08cm]
 \node[point, label=above:$(0$] at (1,2) {};
 \node[point, label=below:$(0$] at (1,-1) {};
 \node[point, label=above:$1)$] at (1.5,2) {};
 \node[point, label=below:$1)$] at (1.5,-1) {};
 \node[point, label=above:$2$] at (4,2) {};
 \node[point, label=below:$2$] at (4,-1) {};
\draw[->,thick] (0.75,0.5) .. controls (0.75,-0.25) and (4,-0.25) .. (4,-0.95); 
 \node[point2, ,white] at (1.5,0) {};
 \node[point2, ,white] at (1,0.2) {};
 \draw[->,thick] (1,2) .. controls (1,0.5) .. (1,-0.95);
 \draw[->,thick] (1.5,2) .. controls (1.5,0.5) .. (1.5,-0.95);
 \node[point2, ,white] at (1.5,1.05) {};
 \node[point2, ,white] at (1,0.8) {};
 \draw[-,thick] (4,2) .. controls (4,1.25) and (0.75,1.25) .. (0.75,0.5); 
\end{tikzpicture} 
= 
\begin{tikzpicture}[baseline=(current bounding box.center)] 
\tikzstyle point=[circle, fill=black, inner sep=0.05cm]
\tikzstyle point2=[circle, fill=black, inner sep=0.08cm]
 \node[point, label=above:$(0$] at (1,2) {};
 \node[point, label=below:$(0$] at (1,-1.5) {};
 \node[point, label=above:$1)$] at (1.5,2) {};
 \node[point, label=below:$1)$] at (1.5,-1.5) {};
 \node[point, label=above:$2$] at (4,2) {};
 \node[point, label=below:$2$] at (4,-1.5) {};
\draw[->,thick] (1,2) .. controls (1,1) and (1,1).. (1,-1.45); 
\draw[-,thick] (1.5,2) .. controls (1.5,1.75) and (3.5,1.75).. (3.5,1.5);
\draw[-,thick] (4,2) .. controls (4,1.5) and (4,1.5).. (4,1.5);
\draw[-,thick] (3.5,1.5) .. controls (3.5,1.25) and (4,1.25).. (4,1); 
\node[point2, ,white] at (3.75,1.25) {};
\draw[-,thick] (4,1.5) .. controls (4,1.25) and (3.5,1.25).. (3.5,1); 
\draw[-,thick] (3.5,1) .. controls (3.5,0.75) and (1.5,0.75).. (1.5,0.5);
  \draw[-,thick] (4,1) .. controls (4,0.5) and (4,0.5).. (4,0.5);
 \draw[-,thick] (1,2) .. controls (1,0.5) .. (1,-0.95);
 \node[point, ,white] at (1,0.35) {};
 \draw[-,thick] (1.5,0.5) .. controls (1.5,0.35) and (0.75,0.35) .. (0.75,0.25); 
 \draw[-,thick] (0.75,0.25) .. controls (0.75,0.15) and (1.5,0.15) .. (1.5,0); 
 \node[point, ,white] at (1,0.15) {};
 \draw[-,thick] (1,0.25) .. controls (1,0.25) .. (1,-0.95);
 \draw[-,thick] (4,0.5) .. controls (4,0.25) and (4,0.25).. (4,-0.5);
 \draw[-,thick] (1.5,0) .. controls (1.5,-0.25) and (3.5,-0.25).. (3.5,-0.5);
\draw[-,thick] (3.5,-0.5) .. controls (3.5,-0.75) and (4,-0.75).. (4,-1); 
\node[point2, ,white] at (3.75,-0.75) {};
\draw[-,thick] (4,-0.5) .. controls (4,-0.75) and (3.5,-0.75).. (3.5,-1); 
\draw[->,thick] (3.5,-1) .. controls (3.5,-1.15) and (1.5,-1.15).. (1.5,-1.45);
 \draw[->,thick] (4,-1) .. controls (4,-1.25) and (4,-1.25).. (4,-1.45);
\end{tikzpicture}
\end{align}
\end{center}

Therefore, by the universal property of $\mathcal Q^1$, there is a morphism of 
$\PaB$-moperads $\mathcal Q^1 \to \mathbf{PaB}^{1}$, which is the identity on objects.
In order to show that this is an isomorphism, it suffices to show that it is an 
isomorphism at the level of automorphism groups of an object arity-wise because 
all groupoids involved are connected. 
Let $n\geq 0$, and let $p$ be the object $(\cdots(01)2\cdots\cdots)n$ of 
$\mathcal Q^1(n)$ and $\PaB^1(n)$. We want to show that the induced group morphism 
$$
\on{Aut}_{\mathcal Q^1(n)}(p)\longrightarrow \on{Aut}_{\PaB^1(n)}(p)=
\pi_1\big(\bar{\textrm{C}}(\mathbb{C}^\times,n),p\big)
$$
is an isomorphism. 

On the one hand, we can replace the base-point $p$ with $p_{reg}=(1,2,\dots,n)\in 
\textrm{C}(\mathbb{C}^\times,n)$, as they are in the same path-connected component. 
Moreover, since the Axelrod--Singer--Fulton--MacPherson compactification does 
not change the homotopy type of our configuration spaces, we get an isomorphism 
$$
\pi_1(\bar{\textrm{C}}(\mathbb{C}^\times,n),p) \simeq \pi_1(\textrm{C}(\mathbb{C}^\times,n),p_{reg})\,.
$$

On the other hand, in \cite[\S4.4]{En}, Enriquez proves several useful facts: 
\begin{itemize}
\item Given a braided module category $\mathcal M$ over a braided monoidal 
category $\mathcal C$, an object $X$ of $\mathcal C$, and an 
object $M$ of $\mathcal M$, there is a group morphism 
$$
\on{B}_{n}^{1}\to \on{Aut}_{\cM}(M\otimes X^{\otimes n})\,,
$$
where, by convention, $M\otimes X^{\otimes n}$ comes equipped with the 
left-most parenthesization $((M\otimes X)\otimes ... )\otimes X$, 
and $\on{B}_{n}^{1}=\on{B}_{n+1}\times_{\mathfrak{S}_{n+1}}\mathfrak{S}_n$ is generated by elements $\sigma_i$, for $1\leq i \leq n-1$ and $\tau$. Seen in $\on{B}_{n+1}$ with generators $\sigma_0,\ldots \sigma_{n-1}$, we have $\tau=\sigma_0^2=x_{01}$.
\item There is a universal braided module category $\PaB^{1,Enr}$ generated 
by a single object $0$, over the universal braided monoidal 
category $\PaB^{Enr}$ generated by a single object $\bullet$. 
Hence objects of $\PaB^{1,Enr}$ are parenthesizations of $0\bullet\cdots\bullet$, 
and thus $p$ determines an object (which we abusively still denote $p$). 
\item the morphism $\on{B}_{n}^{1}\to \on{Aut}_{\PaB^{1,Enr}}(p)$ is an isomorphism. 
\end{itemize}
One can moreover see that, by construction, $\on{Aut}_{\mathcal Q^1(n)}(p)$ is 
exactly the kernel subgroup 
$$
\ker\left(\on{Aut}_{\PaB^{1,Enr}(n)}(p) \to \mathfrak{S}_n\right)\simeq \on{PB}_{n+1}\,.
$$
Hence we have a commuting diagram 
$$
\xymatrix{
 \on{PB}^{1}_{n} \ar[r]^-{\simeq} \ar[d] &  \on{Aut}_{\mathcal Q^1(n)}(p) \ar[r] \ar[d] & 
\pi_1\left(\overline{\textrm{C}}(\mathbb{C}^\times,n),p\right)  \ar[d] & \ar[l]_-{\simeq}
 \pi_1\left(\textrm{C}(\mathbb{C}^\times,n),p_{reg}\right)\ar[d] \\ 
 \on{B}^{1}_{n} \ar[r]^-{\simeq} \ar[d] &  \on{Aut}_{\PaB^{1,Enr}}(p) \ar[r] \ar[d] & 
\pi_1\left(\overline{\textrm{C}}(\mathbb{C}^\times,n)/\mathfrak S_n,[p] \right) \ar[d]  & \ar[l]_-{\simeq}
\pi_1\left(\textrm{C}(\mathbb{C}^\times,n)/\mathfrak S_n,[p_{reg}]\right) \ar[d] \\
\mathfrak{S}_n \ar@{=}[r]& \mathfrak{S}_n \ar@{=}[r] & \mathfrak{S}_n \ar@{=}[r] & \mathfrak{S}_n
}
$$
where all vertical sequences are short exact sequences. 
Thus, in order to get that the map 
$\on{Aut}_{\mathcal Q^1(n)}(p) \to \pi_1\left(\overline{\textrm{C}}(\mathbb{C}^\times,n),p\right)$
 is an isomorphism, 
we are left to prove that the composite map 
$\on{B}_{n}^1 \longrightarrow \pi_1(\textrm{C}(\mathbb{C}^\times,n)/\mathfrak S_n,[p_{reg}])$ is 
indeed an isomorphism. But this map is, by its very construction, 
the isomorphism (from \cite{Sos,Ver}) exhibiting a presentation by 
generators and relations of the braid group of the annulus. 
\end{proof}

\section{The moperad of twisted parenthesized braids, and cyclotomic GT}
\label{Cyclo}

% 4.1

\subsection{Compactified twisted configuration space of the annulus}

Consider, for $N \geq 1$, the additive group $\Gamma = \mathbb{Z} / N\mathbb{Z}$. 
To every finite set $I$ let us associate the so-called $\Gamma$-\tmtextit{twisted configuration space}
\[
\mathrm{Conf} (\mathbb{C^{\times}}, I, \Gamma) = \{ \mathbf{z} = (z_i)_{i \in I} \in 
(\mathbb{C^{\times}})^I |z_i \neq \zeta z_j, \forall i \neq j, \forall \zeta \in {\mu}_N\}
\]
(${\mu}_N$ is the set of complex $N$th roots of unity) and its reduced version
\[
\text{\gls{CCIM}} \assign \mathrm{Conf} (\mathbb{C^{\times}}, I, \Gamma) / \mathbb{R}_{> 0}\,.
\]

There are inclusions 
\begin{equation}\label{eqn-inclusions}
\mathrm{Conf} (\mathbb{C^{\times}}, I, \Gamma)\hookrightarrow \mathrm{Conf} 
(\mathbb{C^{\times}}, I\times\mu_N)
\qquad\textrm{and}\qquad
\mathrm{C} (\mathbb{C^{\times}}, I, \Gamma)\hookrightarrow \mathrm{C} 
(\mathbb{C^{\times}}, I\times\mu_N)
\end{equation}
given by $(z_i)_{i\in I}\mapsto (\zeta z_i)_{(i,\zeta)\in I\times\mu_N}$. 
This allows us to define the compactification $\overline{\mathrm{C}} 
(\mathbb{C}^{\times}, I, \Gamma)$ of $\mathrm{C} (\mathbb{C}^{\times},I, \Gamma)$, 
as the closure of $\mathrm{C} (\mathbb{C^{\times}}, I, \Gamma)$ inside 
$\overline{\mathrm{C}}(\mathbb{C^{\times}}, I\times\mu_N)$. 
The irreducible components of its boundary 
$\partial\overline{\textrm{C}}(\mathbb{C^{\times}},I,\Gamma)=\overline{\textrm{C}} 
(\mathbb{C^{\times}},I,\Gamma)-
\textrm{C}(\mathbb{C^{\times}},I,\Gamma)$ can be described as follows. 
For an arbitrary partition $J_0\coprod\cdots\coprod J_k$ of $\{0\}\sqcup I$ 
such that $0\in J_0$, there is a component 
\[
\partial_{J_0,\cdots,J_k}\overline{\textrm{C}}(\mathbb{C^{\times}},I,\Gamma)
\cong \overline{\rm C}(\mathbb{C^{\times}},J_0,\Gamma) \times 
\overline{\rm C}(\mathbb{C^{\times}},k,\Gamma) 
\times \prod_{i=1}^k\overline{\rm C}(\mathbb{C},J_i)\,.
\]
The inclusion of these boundary components provides $\overline{\textrm{C}}(\mathbb{C^{\times}},-,\Gamma)$ 
with the structure of a $\overline{\textrm{C}}(\mathbb{C},-)$-moperad in topological spaces. 

\medskip

Observe that $\mathrm{Conf} (\mathbb{C^{\times}}, I, \Gamma)$, resp.~$\mathrm{C} 
(\mathbb{C^{\times}}, I, \Gamma)$, is a $\Gamma^I$-covering space of 
$\mathrm{Conf} (\mathbb{C^{\times}}, I)$, resp.~$\mathrm{C} (\mathbb{C^{\times}}, I)$, 
the covering map being given by $(z_i)_{i\in I}\mapsto (z_i^N)_{i\in I}$. 

We let the reader check that this covering map extends to a continuous map 
$\phi_n:\overline{\mathrm{C}}(\mathbb{C^{\times}}, I, \Gamma)\to \overline{\mathrm{C}}(\mathbb{C^{\times}}, I)$ 
between their compactifications, and thus leads to a morphism of $\overline{\textrm{C}}(\mathbb{C},-)$-moperads. 

Finally, there is a natural action of $\Gamma^I$ on each 
$\mathrm{C} (\mathbb{C^{\times}}, I\times\mu_N)$, given by 
\[
(\mathbf{\alpha}\cdot\mathbf{z})_{(j,\zeta)}:=
\mathbf{z}_{\left(j,\zeta e^{\frac{2\mathrm{i}\pi\alpha_j}{N}}\right)}\,,
\]
which is such that the inclusions \eqref{eqn-inclusions} are $\Gamma^I$ equivariant, and thus 
induces an action of $\Gamma$ on the moperad $\overline{\textrm{C}}(\mathbb{C^{\times}},-,\Gamma)$, 
in the sense of \S\ref{sec-gpaction}. 

% 4.2

\subsection{The $\mathbf{Pa}$-moperad of labelled parenthesized permutations}
\label{ssec:labelperm}

Borrowing the notation from the previous subsection, we define 
$\mathbf{Pa_0^\Gamma}(n):=\phi_n^{-1}\big(\mathbf{Pa_0}(n)\big)$. 
Explicitly, $\mathbf{Pa_0^\Gamma}(n)$ is the set of parenthesized 
permutations of $\{0,1,\dots,n\}$ that fix $0$ and that are equipped 
with a label $\{1,\dots,n\}\to\Gamma$. In terms of configuration spaces, the label $\alpha\in\Gamma^n$ 
expresses the fact that we have a configuration $(z_1,\dots,z_n)$ of points
approaching $0$ with the condition that $e^{\frac{2\mathrm{i}\pi\alpha_j}{N}}z_j\in\mathbb{R}_{>0}\subset\mathbb{C}^\times$ 
for every $j\in\{1,\dots,n\}$.  
\\
\textbf{Notation.} As a matter of notation, we will write the label as an 
index attached to each $1,\dots,n$. 
For instance, $(02_\beta)1_\alpha$ belongs to $\mathbf{Pa_0^\Gamma}(2)$ 
for every $\alpha,\beta\in\Gamma$. \\
Observe that the $\mathfrak{S}$-module (in sets) $\mathbf{Pa_0^\Gamma}$ 
carries the structure of a $\mathbf{Pa}$-moperad. Indeed, it is a fiber product 
$$
\mathbf{Pa_0^\Gamma}=\mathbf{Pa_0}\underset{\overline{\mathrm{C}}
(\mathbb{C^{\times}},-)}{\times}\overline{\mathrm{C}}(\mathbb{C^{\times}},-, \Gamma)
$$
in the category of $\mathbf{Pa}$-moperads (in topological spaces). 
Here are two self-explanatory examples of partial compositions: 
$$
(02_\alpha)1_\beta\circ_2(12)3=(0((2_\alpha3_\alpha)4_\alpha))1_\beta
\qquad\textrm{and}\qquad
(02_\alpha)1_\beta\circ_0(02_\gamma)1_\delta=(((02_\gamma)1_\delta)4_\alpha)3_\beta\,.
$$
\begin{remark}
As we have seen in Subsection \ref{sec-pointings}, 
our convention is such that the $\mathbf{Pa}$-moperad structure on
 $\mathbf{Pa_0^\Gamma}$ gives in particular a morphism of $\mathbf{Pa}$-modules 
$\mathbf{Pa}\to \mathbf{Pa_0^\Gamma}$. 
One can see that it is the map that sends a parenthezised permutation 
$\mathbf{p}$ to $0(\mathbf{p})$ together with the trivial label function $i\mapsto 0$. 
\end{remark}
Finally, $\mathbf{Pa_0^\Gamma}$ is acted on 
by $\Gamma$ in the following way: for $n\geq0$, $\Gamma^n$ 
only acts on the labellings, \textit{via} the group law of $\Gamma$. 
For instance, if $f:\{1,\dots,n\}\to\Gamma$ and $\alpha\in\Gamma^n$, 
then $(\alpha\cdot f)(i)=f(i)+\alpha_i$. 

% 4.3

\subsection{The $\mathbf{PaB}$-moperad of parenthesized cyclotomic braids}\label{sec-4.5}

We define 
$$
\text{\gls{PaBg}}:=\pi_1\left(\overline{\textrm{C}}(\mathbb{C^{\times}},-,
\Gamma),\mathbf{Pa_0^\Gamma}\right)\,.
$$
It is a $\PaB$-moperad (in groupoids), that carries an action of the group $\Gamma$. 
The maps $\phi_n:\overline{\textrm{C}}(\mathbb{C}^{\times},n,\Gamma) 
\to \overline{\textrm{C}}(\mathbb{C}^{\times},n)$ 
induce a $\PaB$-moperad morphism $\PaB^\Gamma \to \PaB^1$. 

\begin{example}[Description of $\mathbf{PaB}^{\Gamma}(1)$]
First observe that $\mathbf{Pa_0^\Gamma}(1)\to \mathbf{Pa_0}(1)$ 
is the terminal map $\mu_N\simeq\{01_\alpha|\alpha\in\Gamma\}\to \{01\}=*$. 
Then observe that the map $\overline{\mathrm{C}}(\mathbb{C}^\times,1,
\Gamma)\to \overline{\mathrm{C}}(\mathbb{C}^\times,1)$ 
is nothing but the path-connected $\Gamma$-cover $S^1\to S^1$. 
Hence we in particular have morphisms $E^{0,1_\alpha}$, 
$\alpha\in \Gamma$ from $01_\alpha$ to 
$01_{\alpha+\bar1}$ in $\PaB^\Gamma(1)$, being the unique lift of 
$E^{0,1}$ that starts at $01_\alpha\in \mathbf{Pa_0^\Gamma}(1)$.  
Pictorially: 
\begin{center}
\begin{tikzpicture}[baseline=(current bounding box.center)]
\tikzstyle point=[circle, fill=black, inner sep=0.05cm]
\tikzstyle point2=[circle, fill=black, inner sep=0.08cm]
\node[point, label=above:$0$] at (1,1) {};
\node[point, label=above:$1_{\bar0}$] at (1.5,1) {};
\node[point, label=below:$0$] at (1,0) {};
\node[point, label=below:$1_{\bar1}$] at (1.5,0) {};
\draw[->,thick] (0.75,0.5) .. controls (0.75,0.35) and (1.5,0.35) .. (1.5,0.05);
\node[point, ,white] at (1,0.37) {};
\draw[->,thick] (1,1) .. controls (1,0.5) .. (1,0.05);
\node[point, ,white] at (1,0.7) {};
\draw[-,thick] (1.5,1) .. controls (1.5,0.75) and (0.75,0.75) .. (0.75,0.5); 

\end{tikzpicture}
\qquad\qquad\qquad
\begin{tikzpicture}[baseline=(current bounding box.center)]
\tikzstyle point=[circle, fill=black, inner sep=0.05cm]
 \node[point, label=above:$0$] at (-3,0) {};
 \node[point, label=right:$z_1$] at (-2.5,0) {};
 \node[point, label=left:$e^{-2\mathrm{i}\pi/N}z_1$] at (-3.3,-0.3) {};
 \draw[->,thick] (-2.5,0) .. controls (-2.9,-0.3) .. (-3.3,-0.3); 
 \node[point, label=left:$0$] at (0,0) {};
 \node[point, label=right:$z_1^N$] at (0.5,0) {};
 \draw[->,thick] (0.5,0) .. controls (-1,-1.5) and (-1,1.5) .. (0.5,0); 
 \draw[->,thick] (-1.7,0) to (-1,0);
 \node at (-1.9,-1) {$z$};
 \node at (-0.65,-1) {$z^N$};
 \draw[|->,thick] (-1.7,-1) to (-1,-1);
\end{tikzpicture}
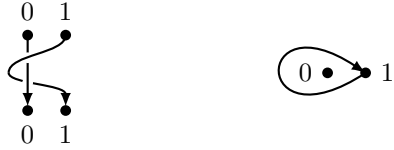
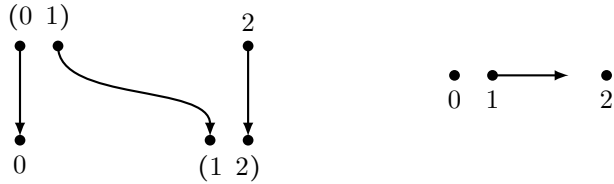
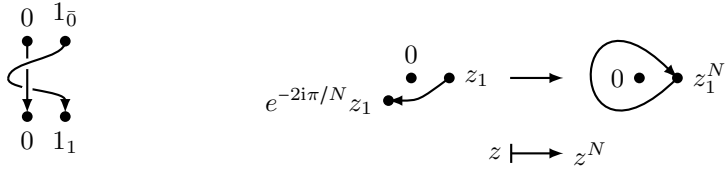
\\ \text{Two incarnations of $E^{0,1_{\bar0}}$}
\end{center}
\end{example}
In the above picture, on the right we have pictured a path in the twisted configuration space, together with its image under the covering map, which is a loop. 
Diagrammatically (see the left of the above picture), we depict it as a pure braid (a loop in the base configuration space) 
together with appropriate base points (which uniquely determines the lift in the covering twisted configuration space). 

\begin{example}[Notable arrow in $\mathbf{PaB}^{\Gamma}(2)$]
Let $\Psi^{0,1_{\bar0},2_{\bar0}}$ be the unique lift of $\Psi^{0,1,2}$ (a morphism in $\PaB^1(2)$) 
starting at $(01_{\bar0})2_{\bar0}$. It can be depicted as follows: 
\begin{center}
\begin{tikzpicture}[baseline=(current bounding box.center)]
\tikzstyle point=[circle, fill=black, inner sep=0.05cm]
 \node[point, label=above:$(0$] at (1,1) {};
 \node[point, label=below:$0$] at (1,-0.25) {};
 \node[point, label=above:$1_{\bar0})$] at (1.5,1) {};
 \node[point, label=below:$(1_{\bar0}$] at (3.5,-0.25) {};
 \node[point, label=above:$2_{\bar0}$] at (4,1) {};
 \node[point, label=below:$2_{\bar0})$] at (4,-0.25) {};
 \draw[->,thick] (1,1) .. controls (1,0) and (1,0).. (1,-0.20); 
 \draw[->,thick] (1.5,1) .. controls (1.5,0.25) and (3.5,0.5).. (3.5,-0.20);
 \draw[->,thick] (4,1) .. controls (4,0) and (4,0).. (4,-0.20);
\end{tikzpicture}
%\qquad\qquad\qquad
%\begin{tikzpicture}[baseline=(current bounding box.center)] % pour avoir des fleches au milieu des courbes
%\tikzstyle point=[circle, fill=black, inner sep=0.05cm]% style des points
% \node[point, label=below:$0$] at (0.5,0.5) {};
% \node[point, label=below:$z_1$] at (1,0.5) {};
% \node[point, label=below:$z_2$] at (2.5,0.5) {};
% \draw[->,thick] (1,0.5) .. controls (1,0.5) .. (2,0.5); 
%\end{tikzpicture}
%\\ \text{Two incarnations of $\Psi_0^{0,1,2}$}
\end{center}
\end{example}

\begin{remark}
As in Remark \ref{rem4.3-poiting}, one can see from 
\S\ref{sec-pointings} that there is a morphism of $\mathfrak{S}$-modules $\PaB\to\PaB^\Gamma$. 
In pictorial terms, it sends a parenthesized braid with $n$ 
strands to a labelled parenthesized braid with $n+1$ strands 
by adding a frozen strand labelled by $0$ on the left and choosing the trivial label. 
For instance, the images $R^{1_{\bar0},2_{\bar0}}$ of $R^{1,2}$ and 
$\Phi^{1_{\bar0},2_{\bar0},3_{\bar0}}$ of $\Phi^{1,2,3}$ can be respectively depicted as follows: 
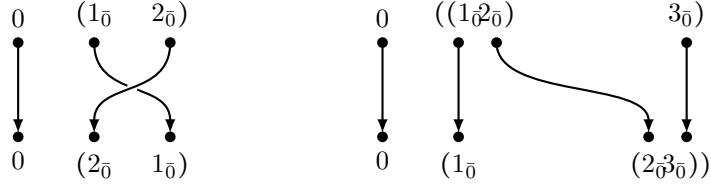
\begin{center}
\begin{tikzpicture}[baseline=(current bounding box.center)]
\tikzstyle point=[circle, fill=black, inner sep=0.05cm]
 \node[point, label=above:$0$] at (0,1) {};
 \node[point, label=below:$0$] at (0,-0.25) {};
 \node[point, label=above:$(1_{\bar0}$] at (1,1) {};
 \node[point, label=below:$(2_{\bar0}$] at (1,-0.25) {};
 \node[point, label=above:$2_{\bar0})$] at (2,1) {};
 \node[point, label=below:$1_{\bar0})$] at (2,-0.25) {};
 \draw[->,thick] (0,1) to (0,-0.20) ;
  \draw[->,thick,postaction={decorate}] (1,1) .. controls (1,0.25) and (2,0.5).. (2,-0.20);
 \node[point, ,white] at (1.5,0.4) {};
 \draw[->,thick] (2,1) .. controls (2,0.25) and (1,0.5).. (1,-0.20); 
\end{tikzpicture}
\qquad\qquad\qquad
\begin{tikzpicture}[baseline=(current bounding box.center)]
\tikzstyle point=[circle, fill=black, inner sep=0.05cm]
 \node[point, label=above:$0$] at (0,1) {};
 \node[point, label=below:$0$] at (0,-0.25) {};
 \node[point, label=above:$((1_{\bar0}$] at (1,1) {};
 \node[point, label=below:$(1_{\bar0}$] at (1,-0.25) {};
 \node[point, label=above:$~2_{\bar0})$] at (1.5,1) {};
 \node[point, label=below:$(2_{\bar0}$] at (3.5,-0.25) {};
 \node[point, label=above:$3_{\bar0})$] at (4,1) {};
 \node[point, label=below:$~3_{\bar0}))$] at (4,-0.25) {};
 \draw[->,thick] (0,1) to (0,-0.20) ;
 \draw[->,thick] (1,1) .. controls (1,0) and (1,0).. (1,-0.20); 
 \draw[->,thick] (1.5,1) .. controls (1.5,0.25) and (3.5,0.5).. (3.5,-0.20);
 \draw[->,thick] (4,1) .. controls (4,0) and (4,0).. (4,-0.20);
\end{tikzpicture}
\end{center}
\end{remark}

\noindent\textbf{Notation.} (i) First of all, for any arrow 
$X=X^{0,1_{\bar0},\dots,n_{\bar0}}$ in 
$\mathbf{PaB}^\Gamma(n)$ starting at a parenthesized permutation 
$x$ equipped with the constant labelling equal to $\bar0$, and for any 
$\underline{\alpha}=(\alpha_1,\dots,\alpha_n)\in\Gamma^n$, 
we write $X^{0,1_{\alpha_1},\dots,n_{\alpha_n}}:=\underline{\alpha}\cdot X$, 
which starts now at the same parenthesized permutation $x$ 
equipped with the labelling $\underline{\alpha}$. 

(ii) Second of all, for $p\geq0$, if $X$ ends at the same 
parenthesized permutation $x$, but equipped with a possibly 
non-trivial labelling $\underline{\alpha}$, then we write 
$$
X^{(p)}:=\prod_{k=0,...,p-1}^{\rightarrow}(k\underline{\alpha})\cdot X
=X^{0,1_{\bar0},\dots,n_{\bar0}}X^{0,1_{\alpha_1},\dots,n_{\alpha_n}}\cdots X^{0,1_{(p-1)\alpha_1},\dots,n_{(p-1)\alpha_n}}\,,
$$
which starts at $(x,\underline{\bar0})$ and ends at $(x,p\underline{\alpha})$.  

(iii) Finally, if $\gamma\in\Gamma$ and $1\leq i\leq n$, then 
we write $\gamma_i:=(\bar0,\dots,\bar0,\underset{i}{\gamma},\bar0,\dots,\bar0)$. 
In particular,  
$$
(E^{0,1_{\bar0}})^{(p)}:=\prod_{k=0,...,p-1}^{\rightarrow}E^{0,1_{\bar{k}}}
=E^{0,1_{\bar0}} E^{0,1_{\bar1}}\cdots E^{0,1_{\overline{p-1}}}\,,
$$
which is an element in $\on{Hom}_{\PaB^\Gamma(1)}((0,1_{\bar0}),(0,1_{\bar{p}}))$.

% 4.4

\subsection{The universal property of $\PaB^\Gamma$}

We are now ready to provide an explicit presentation for the $\PaB$-moperad $\PaB^\Gamma$:
\begin{theorem}\label{thm-gamma-pres}
As a $\mathbf{PaB}$-moperad in groupoids with a $\Gamma$-action having $\mathbf{Pa}_0^{\Gamma}$ as 
$\mathbf{Pa_0^\Gamma}$-moperad of objects, $\mathbf{PaB}^{\Gamma}$ is generated by $E^{0,1_{\bar0}}$ and $\Psi^{0,1_{\bar0},2_{\bar0}}$ together with the following relations: 
\begin{flalign}
& \Psi^{0,\emptyset,1_{\bar0}}=\on{Id}_{0,1_{\bar0}}
\quad\Big(\mathrm{in}~\mathrm{Hom}_{\PaB^\Gamma(1)}\big(01_{\bar0},01_{\bar0}\big)\Big)\,,  \tag{tU}\label{eqn:tU} \\
& \Psi^{01_{\bar0},2_{\bar0},3_{\bar0}} \Psi^{0,1_{\bar0},2_{\bar0}3_{\bar0}} 
= \Psi^{0,1_{\bar0},2_{\bar0}} \Psi^{0,1_{\bar0}2_{\bar0},3_{\bar0}} \Phi^{1_{\bar0},2_{\bar0},3_{\bar0}} \tag{tM}\label{eqn:tMP} \\
\nonumber & \Big(\mathrm{in}~\mathrm{Hom}_{\PaB^\Gamma(3)}\big(((01_{\bar0})2_{\bar0})3_{\bar0},0(1_{\bar0}(2_{\bar0}3_{\bar0}))\big)\Big)\,, \\
&  \Psi^{0,1_{\bar0},2_{\bar0}}E^{0,1_{\bar0}2_{\bar0}}( \Psi^{0,1_{\bar1},2_{\bar1}})^{-1}
=E^{0,1_{\bar0}}E^{01_{\bar1},2_{\bar0}} 
\quad\Big(\mathrm{in}~\mathrm{Hom}_{\PaB^\Gamma(2)}\big((01_{\bar0})2_{\bar0},(01_{\bar1})2_{\bar1}\big)\Big)\,,  \tag{tR}\label{eqn:tRP} \\
& E^{01_{\bar0},2_{\bar0}} = \Psi^{0,1_{\bar0},2_{\bar0}} R^{1_{\bar0},2_{\bar0}} (\Psi^{0,2_{\bar0},1_{\bar0}})^{-1} E^{0,2_{\bar0}} 
\Psi^{0,2_{\bar1},1_{\bar0}} R^{2_{\bar1},1_{\bar0}}  (\Psi^{0,1_{\bar0},2_{\bar1}})^{-1}),
 \tag{tO}\label{eqn:tO}\\
\nonumber &  \Big(\mathrm{in}~\mathrm{Hom}_{\PaB^\Gamma(2)}\big((01_{\bar0})2_{\bar0},(01_{\bar0})2_{\bar1}\big)\,. 
\end{flalign}
\end{theorem}

\begin{proof}
Let $\mathcal Q^{\Gamma}$ be the $\PaB$-moperad with the above presentation, 
and recall that $\mathcal Q^{1}$ 
is the $\PaB$-moperad with the presentation of Theorem \ref{PaB1}. 
Our first goal is to show that there is a morphism 
$\mathcal Q^{\Gamma}\to\PaB^\Gamma$ of $\PaB$-moperads, 
commuting with the $\Gamma$-action.  
We have already seen in the Examples above that there are 
morphisms $E^{0,1_{\bar0}}$ and $\Psi^{0,1_{\bar0},2_{\bar0}}$, in $\PaB^\Gamma(1)$ and 
$\PaB^\Gamma(2)$, respectively. 
We have to prove that they satisfy the \textit{mixed pentagon} and 
\textit{twisted octogon} relation, \eqref{eqn:tMP} and \eqref{eqn:tO} and \eqref{eqn:tRP}.

\medskip

These relations are the unique lifts of the similar relations 
\eqref{eqn:MP}, \eqref{eqn:RP} and \eqref{eqn:O} in $\PaB^1$ from 
Theorem \ref{PaB1}, starting at $((01_{\bar0})2_{\bar0})3_{\bar0}$ and $(01_{\bar0})2_{\bar0}$, 
respectively. They can be depicted as follows: 
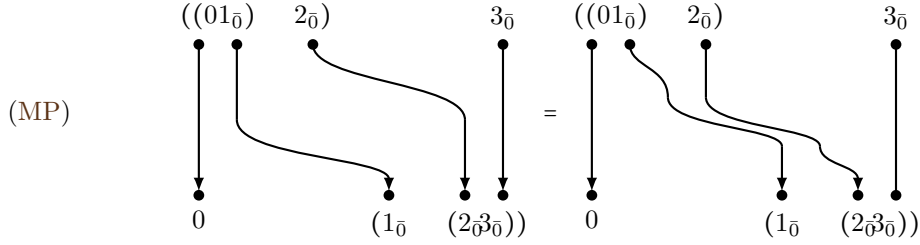
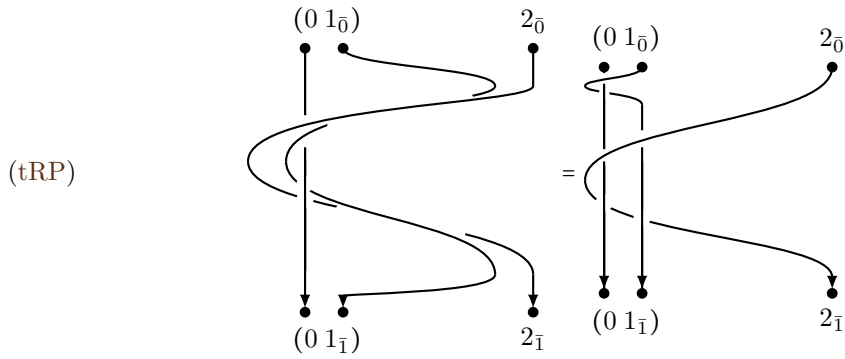
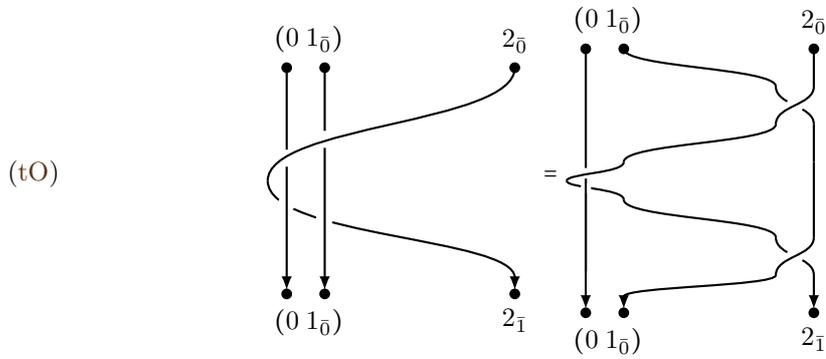
\begin{center}
\begin{align}\tag{\ref{eqn:tMP}}
\begin{tikzpicture}[baseline=(current bounding box.center)] 
\tikzstyle point=[circle, fill=black, inner sep=0.05cm]
\tikzstyle point2=[circle, fill=black, inner sep=0.08cm]
\node[point, label=above:$((0$] at (1,2) {};
\node[point, label=above:$1_{\bar0})$] at (1.5,2) {};
\node[point, label=above:$2_{\bar0})$] at (2.5,2) {};
\node[point, label=above:$3_{\bar0}$] at (5,2) {};
\node[point, label=below:$0$] at (1,0) {};
\node[point, label=below:$(1_{\bar0}$] at (3.5,0) {};
\node[point, label=below:$(2_{\bar0}$] at (4.5,0) {};
\node[point, label=below:$3_{\bar0}))$] at (5,0) {};
 \draw[->,thick] (1,2) .. controls (1,0.5) .. (1,0.05);
 \draw[-,thick] (1.5,2) .. controls (1.5,1.5) .. (1.5,1);
 \draw[-,thick] (2.5,2) .. controls (2.5,1.5) and (4.5,1.5).. (4.5,1);
 \draw[->,thick] (1.5,1) .. controls (1.5,0.5) and (3.5,0.5).. (3.5,0.05);
 \draw[->,thick] (4.5,1) .. controls (4.5,0.5) .. (4.5,0.05);
 \draw[->,thick] (5,2) .. controls (5,0.5) .. (5,0.05);
\end{tikzpicture}
=
\begin{tikzpicture}[baseline=(current bounding box.center)] 
\tikzstyle point=[circle, fill=black, inner sep=0.05cm]
\tikzstyle point2=[circle, fill=black, inner sep=0.08cm]
\node[point, label=above:$((0$] at (1,2) {};
\node[point, label=above:$1_{\bar0})$] at (1.5,2) {};
\node[point, label=above:$2_{\bar0})$] at (2.5,2) {};
\node[point, label=above:$3_{\bar0}$] at (5,2) {};
\node[point, label=below:$0$] at (1,0) {};
\node[point, label=below:$(1_{\bar0}$] at (3.5,0) {};
\node[point, label=below:$(2_{\bar0}$] at (4.5,0) {};
\node[point, label=below:$3_{\bar0}))$] at (5,0) {};
 \draw[->,thick] (1,2) .. controls (1,0.5) .. (1,0.05);
 \draw[-,thick] (2.5,2) .. controls (2.5,1.5) .. (2.5,1.3);
 \draw[-,thick] (1.5,2) .. controls (1.5,1.7) and (2,1.7).. (2,1.3);
 \draw[-,thick] (2,1.3) .. controls (2,0.9) and (3.5,0.9).. (3.5,0.65);
 \draw[-,thick] (2.5,1.3) .. controls (2.5,0.9) and (4,0.9).. (4,0.65);
 \draw[->,thick] (4,0.65) .. controls (4,0.4) and (4.5,0.4).. (4.5,0.05);
 \draw[->,thick] (3.5,0.65) .. controls (3.5,0.5) .. (3.5,0.05);
 \draw[-,thick] (5,2) .. controls (5,0.5) .. (5,0.05);
\end{tikzpicture}
\end{align}
\end{center}
\begin{center}
\begin{align}\tag{\ref{eqn:tRP}}
\begin{tikzpicture}[baseline=(current bounding box.center)]
\tikzstyle point=[circle, fill=black, inner sep=0.05cm]
\tikzstyle point2=[circle, fill=black, inner sep=0.08cm]
\tikzstyle point3=[circle, fill=black, inner sep=0.5cm]
 \node[point, label=above:$(0$] at (1,2) {};
 \node[point, label=below:$(0$] at (1,-1.5) {};
 \node[point, label=above:$1_{\bar0})$] at (1.5,2) {};
 \node[point, label=below:$1_{\bar1})$] at (1.5,-1.5) {};
 \node[point, label=above:$2_{\bar0}$] at (4,2) {};
 \node[point, label=below:$2_{\bar1}$] at (4,-1.5) {};
\draw[->,thick] (1,2) .. controls (1,1) and (1,1).. (1,-1.45); 
\draw[-,thick] (1.5,2) .. controls (1.5,1.75) and (3.5,1.75).. (3.5,1.5);
 \draw[-,thick] (1,2) .. controls (1,0.5) .. (1,-0.95);
 \node[point2, ,white] at (1,1) {};
 \node[point2, ,white] at (1,0.8) {};
   \draw[-,thick] (0.25,0.5) .. controls (0.25,-0.25) and (4,-0.25) .. (4,-1); 
  \node[point3, ,white] at (2.4,-0.5) {}; 
  \node[point3, ,white] at (2.2,-0.75) {};
  \node[point3, ,white] at (2,-0.5) {};
  \draw[-,thick] (0.75,0.5) .. controls (0.75,-0.25) and (3.5,-0.25) .. (3.5,-1); 
     \node[point2, ,white] at (1,0) {};
 \node[point2, ,white] at (1,0.2) {};
 \draw[-,thick] (3.5,1.5) .. controls (3.5,1.25) and (0.75,1.25) .. (0.75,0.5); 
  \node[point3, ,white] at (2.2,1) {};
  \node[point3, ,white] at (2.6,1) {};
  \node[point3, ,white] at (2,1) {};
\draw[-,thick] (4,1.5) .. controls (4,1.25) and (0.25,1.25) .. (0.25,0.5); 
\draw[->,thick] (3.5,-1) .. controls (3.5,-1.25) and (1.5,-1.25).. (1.5,-1.45);
\draw[-,thick] (4,2) .. controls (4,1.5) and (4,1.5).. (4,1.5); 
\draw[->,thick] (4,-1) .. controls (4,-1) and (4,-1).. (4,-1.45); 
 \draw[-,thick] (1,0.5) .. controls (1,0.5) .. (1,-0.95);
\end{tikzpicture}
=
\begin{tikzpicture}[baseline=(current bounding box.center)]
\tikzstyle point=[circle, fill=black, inner sep=0.05cm]
\tikzstyle point2=[circle, fill=black, inner sep=0.08cm]
 \node[point, label=above:$(0$] at (1,2) {};
 \node[point, label=below:$(0$] at (1,-1) {};
 \node[point, label=above:$1_{\bar0})$] at (1.5,2) {};
 \node[point, label=below:$1_{\bar1})$] at (1.5,-1) {};
 \node[point, label=above:$2_{\bar0}$] at (4,2) {};
 \node[point, label=below:$2_{\bar1}$] at (4,-1) {};
  \draw[-,thick] (1,2) .. controls (1,0.5) .. (1,-0.95);
 \node[point, ,white] at (1,0.35+1.5) {};
 \draw[-,thick] (1.5,0.5+1.5) .. controls (1.5,0.35+1.5) and (0.75,0.35+1.5) .. (0.75,0.25+1.5); 
 \draw[-,thick] (0.75,0.25+1.5) .. controls (0.75,0.15+1.5) and (1.5,0.15+1.5) .. (1.5,0+1.5); 
 \node[point, ,white] at (1,0.15+1.5) {};
 \draw[->,thick] (0.75,0.5) .. controls (0.75,-0.25) and (4,-0.25) .. (4,-0.95); 
  \draw[->,thick] (1,1.5) .. controls (1,0.5) .. (1,-0.95);
 \draw[->,thick] (1.5,1.5) .. controls (1.5,0.5) .. (1.5,-0.95);
 \node[point2, ,white] at (1.5,1.05) {};
 \node[point2, ,white] at (1,0.8) {};
 \draw[-,thick] (4,2) .. controls (4,1.25) and (0.75,1.25) .. (0.75,0.5); 
 \node[point2, ,white] at (1.5,0) {};
 \node[point2, ,white] at (1,0.2) {};
   \draw[-,thick] (1,0.75) .. controls (1,0.5) .. (1,-0.95);
      \draw[-,thick] (1,1.75) .. controls (1,1.5) .. (1,1.5);
     \draw[-,thick] (1.5,0.75) .. controls (1.5,0.5) .. (1.5,-0.95); 
\end{tikzpicture} 
\end{align}
\end{center}
and
\begin{center}
\begin{align}\tag{\ref{eqn:tO}}
\begin{tikzpicture}[baseline=(current bounding box.center)]
\tikzstyle point=[circle, fill=black, inner sep=0.05cm]
\tikzstyle point2=[circle, fill=black, inner sep=0.08cm]% style des points
 \node[point, label=above:$(0$] at (1,2) {};
 \node[point, label=below:$(0$] at (1,-1) {};
 \node[point, label=above:$1_{\bar0})$] at (1.5,2) {};
 \node[point, label=below:$1_{\bar0})$] at (1.5,-1) {};
 \node[point, label=above:$2_{\bar0}$] at (4,2) {};
 \node[point, label=below:$2_{\bar1}$] at (4,-1) {};
\draw[->,thick] (0.75,0.5) .. controls (0.75,-0.25) and (4,-0.25) .. (4,-0.95); 
 \node[point2, ,white] at (1.5,0) {};
 \node[point2, ,white] at (1,0.2) {};
 \draw[->,thick] (1,2) .. controls (1,0.5) .. (1,-0.95);
 \draw[->,thick] (1.5,2) .. controls (1.5,0.5) .. (1.5,-0.95);
 \node[point2, ,white] at (1.5,1.05) {};
 \node[point2, ,white] at (1,0.8) {};
 \draw[-,thick] (4,2) .. controls (4,1.25) and (0.75,1.25) .. (0.75,0.5); 
\end{tikzpicture} 
= 
\begin{tikzpicture}[baseline=(current bounding box.center)] 
\tikzstyle point=[circle, fill=black, inner sep=0.05cm]
\tikzstyle point2=[circle, fill=black, inner sep=0.08cm]
 \node[point, label=above:$(0$] at (1,2) {};
 \node[point, label=below:$(0$] at (1,-1.5) {};
 \node[point, label=above:$1_{\bar0})$] at (1.5,2) {};
 \node[point, label=below:$1_{\bar0})$] at (1.5,-1.5) {};
 \node[point, label=above:$2_{\bar0}$] at (4,2) {};
 \node[point, label=below:$2_{\bar1}$] at (4,-1.5) {};
\draw[->,thick] (1,2) .. controls (1,1) and (1,1).. (1,-1.45); 
\draw[-,thick] (1.5,2) .. controls (1.5,1.75) and (3.5,1.75).. (3.5,1.5);
\draw[-,thick] (4,2) .. controls (4,1.5) and (4,1.5).. (4,1.5);
\draw[-,thick] (3.5,1.5) .. controls (3.5,1.25) and (4,1.25).. (4,1); 
\node[point2, ,white] at (3.75,1.25) {};
\draw[-,thick] (4,1.5) .. controls (4,1.25) and (3.5,1.25).. (3.5,1); 
\draw[-,thick] (3.5,1) .. controls (3.5,0.75) and (1.5,0.75).. (1.5,0.5);
  \draw[-,thick] (4,1) .. controls (4,0.5) and (4,0.5).. (4,0.5);
 \draw[-,thick] (1,2) .. controls (1,0.5) .. (1,-0.95);
 \node[point, ,white] at (1,0.35) {};
 \draw[-,thick] (1.5,0.5) .. controls (1.5,0.35) and (0.75,0.35) .. (0.75,0.25); 
 \draw[-,thick] (0.75,0.25) .. controls (0.75,0.15) and (1.5,0.15) .. (1.5,0); 
 \node[point, ,white] at (1,0.15) {};
 \draw[-,thick] (1,0.25) .. controls (1,0.25) .. (1,-0.95);
 \draw[-,thick] (4,0.5) .. controls (4,0.25) and (4,0.25).. (4,-0.5);
 \draw[-,thick] (1.5,0) .. controls (1.5,-0.25) and (3.5,-0.25).. (3.5,-0.5);
\draw[-,thick] (3.5,-0.5) .. controls (3.5,-0.75) and (4,-0.75).. (4,-1); 
\node[point2, ,white] at (3.75,-0.75) {};
\draw[-,thick] (4,-0.5) .. controls (4,-0.75) and (3.5,-0.75).. (3.5,-1); 
\draw[->,thick] (3.5,-1) .. controls (3.5,-1.15) and (1.5,-1.15).. (1.5,-1.45);
 \draw[->,thick] (4,-1) .. controls (4,-1.25) and (4,-1.25).. (4,-1.45);
\end{tikzpicture}
\end{align}
\end{center}

\medskip

By the universal property of $\mathcal Q^{\Gamma}$ there is a 
$\Gamma$-equivariant morphism of $\PaB$-moperads 
$Q^{\Gamma} \longrightarrow \PaB^\Gamma$, which is the 
identity on objects. 
As before, in order to show that this is an isomorphism, it 
suffices to show that it is an isomorphism at the level 
of automorphism groups of an object arity-wise (because all 
groupoids involved are connected). 
Let $n\geq 0$, and let $\tilde p$ be the object 
$(\cdots(01_{\bar0})2_{\bar0}\cdots\cdots)n_{\bar0}$ of $\mathcal Q^\Gamma(n)$ and 
$\PaB^\Gamma(n)$, which lifts the object 
$p=(\cdots(01)2\cdots\cdots)n$ of $\mathcal Q^1(n)\simeq\PaB^1(n)$. 
We want to show that the induced group morphism 
$$
\on{Aut}_{\mathcal Q^\Gamma(n)}(\tilde p)\longrightarrow 
\on{Aut}_{\PaB^\Gamma(n)}(\tilde p)=
\pi_1\big(\bar{\textrm{C}}(\mathbb{C}^\times,n,\Gamma),\tilde p\big)
$$
is an isomorphism. 

We claim that it fits into a commuting diagram
$$
\xymatrix{
 \on{Aut}_{\mathcal Q^\Gamma(n)}(\tilde p) \ar[r] \ar[d] & 
 \pi_1\left(\overline{\textrm{C}}(\mathbb{C}^\times,n,\Gamma),\tilde p\right) \ar[d] 
 & \ar[l]_-{\simeq}
 \pi_1\left(\textrm{C}(\mathbb{C}^\times,n,\Gamma),\tilde p_{reg}\right)\ar[d] \\ 
 \on{Aut}_{\mathcal Q^1(n)}(p) \ar[r]^-{\simeq} \ar[d] & 
 \pi_1\left(\overline{\textrm{C}}(\mathbb{C}^\times,n),p \right)  \ar[d] & 
 \ar[l]_-{\simeq}
 \pi_1\left(\textrm{C}(\mathbb{C}^\times,n)),p_{reg}\right) \ar[d] \\
 \Gamma^{n} \ar@{=}[r] &\Gamma^{n} \ar@{=}[r]&  \Gamma^{n}
}
$$
where only the left-most vertical arrows remain to be described. \\
\underline{The morphism $\on{Aut}_{\mathcal Q^1(n)}(p)\to\Gamma^n$.} 
Let $*$ be the terminal operad in groupoids. 
We have a $*$-moperad structure on the following $\mathfrak{S}$-module 
in groupoids: $\underline{\Gamma}=\{\Gamma^n\}_{\n\geq0}$, 
where we view a group as a groupoid with only one object, and 
where the action of the symmetric group is by permutation. 
The moperad structure is described as follows: 
\begin{itemize}
\item $\circ_0:\Gamma^n\times\Gamma^m\to\Gamma^{n+m}$ 
is the concatenation of sequences, 
\item for every $i\neq 0$, $\circ_i:\Gamma^n\to\Gamma^{n+m-1}$ 
is the partial diagonal 
$$
(\alpha_1,\dots,\alpha_n)\longmapsto (\alpha_1,\dots,\alpha_{i-1},
\underbrace{\alpha_i,\dots,\alpha_i}_{m~\textrm{times}},\alpha_{i+1},\dots,\alpha_n)\,.
$$
\end{itemize}
We let the reader check that sending $E$ to $\bar1\in\Gamma$ and 
$\Psi$ to $(\bar0,\bar0)\in\Gamma^2$ defines a moperad 
morphism $\PaB^1\to\underline{\Gamma}$ along the terminal 
operad morphism $\PaB\to*$. 
This in particular induces a group morphism 
$$
\on{Aut}_{\mathcal Q^1(n)}(p)\longrightarrow\Gamma^n,
$$
for every $n\geq0$. Heuristically, this morphism counts, for 
every $i$, and modulo $N$, the number of times that 
$E^{0,i}$ appears in an element of $\on{Aut}_{\mathcal Q^1(n)}(p)$. 
It is obviously surjective, and we let the reader check that the 
following triangle commutes: 
$$
\xymatrix{
 \on{Aut}_{\mathcal Q^1(n)}(p) \ar[r]^-{\simeq} \ar[dr] & 
 \pi_1\left(\overline{\textrm{C}}(\mathbb{C}^\times,n),p \right)  \ar[d] \\
 & \Gamma^{n} 
}
$$
\underline{The morphism 
$\on{Aut}_{\mathcal Q^\Gamma(n)}(\tilde p)\to \on{Aut}_{\mathcal Q^1(n)}(p)$.} We have a 
$\Gamma$-equivariant morphism of \linebreak
$\PaB$-moperads $\mathcal Q^\Gamma\to\mathcal Q^1$, 
where $\Gamma$ acts trivially on $\mathcal Q^1$, which forgets the label on objects, and sends the generators 
$E^{0,1_{\bar0}}$ and $\Psi^{0,1_{\bar0},2_{\bar0}}$ to $E$ and $\Psi$, respectively. It obviously fits into 
a commuting square 
$$
\xymatrix{
 \mathcal Q^\Gamma \ar[r]\ar[d] & \PaB^\Gamma \ar[d] \\
 \mathcal Q^1 \ar[r] & \PaB^1
}
$$
of $\PaB$-moperads. This induces in particular a group morphism 
$$
\on{Aut}_{\mathcal Q^\Gamma(n)}(\tilde p)\longrightarrow \on{Aut}_{\mathcal Q^1(n)}(p),
$$
for every $n\geq0$, that fits into a commuting square
$$
\xymatrix{
 \on{Aut}_{\mathcal Q^\Gamma(n)}(\tilde p) \ar[r] \ar[d] & 
 \pi_1\left(\overline{\textrm{C}}(\mathbb{C}^\times,n,\Gamma),\tilde p\right) \ar[d] \\ 
 \on{Aut}_{\mathcal Q^1(n)}(p) \ar[r]^-{\simeq} & 
 \pi_1\left(\overline{\textrm{C}}(\mathbb{C}^\times,n),p \right)
}
$$

We now turn to the proof of the fact that the left-most vertical 
sequence is a short exact sequence, which shows that 
$$
\on{Aut}_{\mathcal Q^\Gamma(n)}(\tilde p)\longrightarrow 
\on{Aut}_{\PaB^\Gamma(n)}(\tilde p)=\pi_1\big(\bar{\textrm{C}}
(\mathbb{C}^\times,n,\Gamma),\tilde p\big)
$$
is an isomorphism. We already know that the morphism 
$\on{Aut}_{\mathcal Q^1(n)}(p)\to\Gamma^n$ is surjective. \\
\underline{The morphism $\on{Aut}_{\mathcal Q^\Gamma(n)}(\tilde p)\to 
\on{Aut}_{\mathcal Q^1(n)}(p)$ is injective.} Indeed, an automorphism 
of $\tilde{p}$ in \linebreak
$\mathcal Q^\Gamma(n)$ can be 
represented by a finite sequence $\tilde{S}$ of $R$'s, 
$\Phi$'s, $E$'s, $\Psi$'s, 
and their images under the action of $\Gamma^n$. 
The image of such an automorphism under 
$\mathcal Q^\Gamma\to\mathcal Q^1$ 
is represented by the corresponding finite sequence 
$S$ of $R$'s, $\Phi$'s, $E$'s and $\Psi$'s. 
Every modification of $S$ using the relations 
\eqref{eqn:MP}, \eqref{eqn:RP} and \eqref{eqn:O} can be lifted (uniquely) to a 
modification of $\tilde{S}$ using \eqref{eqn:tMP}, \eqref{eqn:tRP} and \eqref{eqn:tO} or 
their images under the action of $\Gamma^n$. 
Hence, if an automorphism has trivial image, then it must be trivial. \\
\underline{The sequence is exact.} We already know 
from the commuting diagram that the image of 
$\on{Aut}_{\mathcal Q^\Gamma(n)}(\tilde p)$ in 
$\on{Aut}_{\mathcal Q^1(n)}(p)$ lies 
in the kernel of $\on{Aut}_{\mathcal Q^1(n)}(p)\to\Gamma^n$. 
We finally can show that the image is exactly the kernel. Indeed:
\begin{itemize}
\item Using \eqref{eqn:O}, every element $g$ in 
$\on{Aut}_{\mathcal Q^1(n)}(p)$ can be written as a 
product of $\Phi$'s, $R$'s, $\Psi$'s and $E$'s, where the only 
$E$'s appearing are of the form $E^{0,i}$. 
\item Such an element admits a unique lift to a morphism $\tilde{g}$ 
in $\mathcal Q^\Gamma(n)$, with 
source being $\tilde p$ (one just replace $\Phi$'s, $R$'s, $\Psi$'s 
and $E$'s in the expression for $g$ 
by the same symbols, maybe acted on by 
$\Gamma^n$ in order to get the correct starting objects). 
\item An element $g$ as above lies in 
$$
\ker\big(\on{Aut}_{\mathcal Q^1(n)}(p)\longrightarrow\Gamma^n\big)
$$
if and only if for every $i$, the number of occurence of $E^{0,i}$ 
(counted in an algebraic way) is a multiple of $N$. 
This tells us in particular that the target of the lifted morphism 
shall be the same as its source, so that $\tilde{g}$ lies in the kernel. 
\end{itemize}
This ends the proof of the Proposition. 
\end{proof}

% 4.5

\subsection{Cyclotomic Grothendieck-Teichm\"uller groups}\label{sec:cycloGT}

We let $\on{Mop}\mathcal C$ be the category of pairs $(\mathcal O,\mathcal M)$, with $\mathcal O$ 
an operad and $\mathcal M$ a $\mathcal O$-moperad, in a symmetric monoidal category $\mathcal C$. 
A morphism $(\mathcal O,\mathcal M)\to (\mathcal P,\mathcal N)$ is a pair $(F,G)$, with 
$F:\mathcal O\to\mathcal P$ an operad morphism and $G:\mathcal M\to \mathcal N$ a 
$\mathcal O$-moperad morphism, where the $\mathcal O$-moperad structure on $\mathcal N$ 
is defined from its $\mathcal P$-moperad structure by applying $F$. 

In addition to the superscript ``$+$'', wich means, as in \S\ref{sec:2.8assoc}, that we are 
considering morphisms of groupoids/categories that are the identity on objects, we may also add, 
as usual, a superscript ``$\Gamma$'' for $\Gamma$-equivariant morphisms. 
 
\begin{definition}
The ($\kk$-pro-unipotent version of the) 
\textit{cyclotomic Grothendieck-Teichm\"uller group} is defined as the group 
$$
\text{\gls{kGTg}}:=\on{Aut}_{\on{Mop}\mathbf{Grpd}_{\kk}}^+\big(\widehat \PaB(\KK),\widehat \PaB^{\Gamma}(\KK)\big)^{\Gamma}
$$
of $\Gamma$-equivariant automorphisms of the pair $\big(\widehat \PaB(\KK),\widehat \PaB^{\Gamma}(\KK)\big)$ 
which are the identity on objects. 
\end{definition}

Our main goal in this subsection is to relate this cyclotomic Grothendieck-Teichm\"uller group 
with one of those introduced by Enriquez in \cite{En}. 

\medskip

Let us recall that $\on{PB}_2\simeq\on{PB}_1^1$ is identified 
with the free group $\on{F}_1$ generated by a single generators 
$x$ (being $x_{12}$ in $\on{PB}_2$, and $x_{01}$ in $\on{PB}_1^1$). 
Let us also recall that the quotient of $\on{PB}_3\simeq\on{PB}_2^1$ 
by its center (which is freely generated by a single 
element) is a free group $\on{F}_2$ generated by two elements $x,y$. 
As usual, we will consider the inclusion of 
$\on{F}_2$ in $\on{PB}_3$ (resp.~$\on{PB}_2^1$) sending $x$ to $x_{12}$ 
(resp.~$x_{01}$), and $y$ to $x_{23}$ (resp.~$x_{12}$). 
Recall finally that $\on{PB}_n^\Gamma$ is the kernel of the 
morphism $\on{PB}_n^1\to\Gamma^n$ sending $x_{0j}$ to $\bar1_j$, 
and the other generators to $(\bar0,\dots,\bar0)$. 

Whenever $n=1$, this is nothing but the morphism 
$\on{F}_1\to \Gamma$ sending $x$ to $\bar1$, having kernel freely 
generated by $X=x^N$. 
Finally notice that the morphism $\phi_{N}:\on{F}_2\to \Gamma$ 
sending $x$ to $\bar1$ and $y$ to $\bar0$ fits into the 
following commuting square: 
\[
\xymatrix{
 \on{F}_2 \ar[r]\ar[d] & \ar[d] \on{PB}_2^1 \\
 \Gamma \ar[r]^{(id,\bar0)} & \Gamma^2
 }
\]
It induces a morphism between the kernels 
$\on{F}_{N+1}\simeq\ker\phi_N\to\on{PB}_2^\Gamma$. 
The generators of $\ker\phi_N$ are $X=x^N$ and 
$y(a)=x^{-a}yx^a$, $1\leq a\leq N-1$. 

\medskip

An element of the cyclotomic Grothendieck-Teichm\"uller group 
$\widehat{\GT}^\Gamma(\KK)$ first depends 
on an automorphism $F$ of $\widehat{\PaB}(\KK)$, 
which is determined by a pair $(\lambda,f)$, 
where $\lambda\in\KK^\times$ and $f\in\hat{\on{F}}_2(\KK)$ 
satisfying the relations from \S \ref{sec:2.8assoc}:
\begin{itemize}
\item $F(R^{1,2})=x_{12}^{\frac{\lambda-1}{2}} R^{1,2}$, 
\item $F(\Phi^{1,2,3})=f(x_{12},x_{23})\Phi^{1,2,3}$. 
\end{itemize}
Then we have an automorphism $G$ of $\widehat{\PaB}^\Gamma(\KK)$, 
compatible with $F$, which is likewise determined 
by the images of $E^{0,1_{\bar0}} \in \on{Hom}_{\widehat{\PaB}^{\Gamma}(\KK)(1)}(01_{\bar0},01_{\bar1})$ and 
$\Psi^{0,1_{\bar0},2_{\bar0}} \in \on{Hom}_{\widehat{\PaB}^{\Gamma}(\KK)(2)}\big((01_{\bar0})2_{\bar0},0(1_{\bar0}2_{\bar0})\big)$: 
\begin{itemize}
\item $G(E^{0,1_{\bar0}})=u E^{0,1_{\bar0}}$, with 
$u=X^{\mu_1}=x^{N\mu_1}$ for some $\mu_1\in\KK^{\times}$, necessarily, 
\item $G(\Psi^{0,1_{\bar0},2_{\bar0}})=v \Psi^{0,1_{\bar0},2_{\bar0}}$,
\end{itemize}
where $v\in\widehat{\on{PB}}_2^\Gamma(\kk)
\subset\widehat{\on{PB}}_2^1(\kk)$ can be written as 
$C^{\mu_2}g(x_{01},x_{12})$, with $C$ a central generator of ${\ker\phi_N}$ and 
$g\in\widehat{\ker\phi}_N(\KK)\subset\hat{\on{F}}_2(\KK)$. \\

\noindent\textbf{Notation.} We will also write 
$g\big(X,y(0),\dots,y(N-1)\big)$ when we want to view $g$ in 
$\hat{\on{F}}_{N+1}(\KK)\simeq\widehat{\ker\phi_N}(\KK)$. \\

Relation \eqref{eqn:tU} tells us that $X^{\mu_2}=v^{0,\emptyset,1}=1$, and thus that $\mu_2=0$. 
Indeed, the morphism $(-)^{0,\emptyset,1}:\on{PB}_2^\Gamma\to \on{PB}_1^\Gamma\simeq\on{F}_1$ sends 
${\ker\phi_N}$ to $1$, and $x_{02}^N$ (as well as the central generator) to $X=x^N$. 
We conclude that $v=g(x_{01},x_{12})=g\big(X,y(0),\dots,y(N-1)\big)$. 

%\textcolor{red}{CHECK PROP AND PROOF}
\begin{proposition}\label{prop-MPO}
The elements $(\lambda,f,\mu,g)$ satisfy 
\begin{equation}\label{eqn:GT:(MP)}
g(x_{02}x_{12},x_{23})g(x_{01},x_{12}x_{13}) = 
g(x_{01},x_{12})g(x_{01}x_{02},x_{13}x_{23})f(x_{12}x_{23})\,,
\end{equation}
and, for $\alpha=\bar{1}\in\Gamma$,
\begin{equation}\label{eqn:GT:(tO)}
x^{\mu}g(x,y)y^{{{\lambda+1}\over 2}}g(z,y)^{-1}
z^{\mu}\alpha \cdot (g(z,y)y^{{{\lambda-1}\over 2}} g(x,y)^{-1})
=1\quad\big(\textrm{in }\hat{\on{F}}_2(\phi_N,\KK)\big)~~~~xyz=1.
\end{equation}

\end{proposition}
\begin{proof}
First of all, the fact that relation \eqref{eqn:GT:(MP)} is satisfied is straightforward. 
Second of all, suppose $N=1$ and consider the image of \eqref{eqn:O} by $G$:
\begin{eqnarray}\label{GtO}
G(E^{01,2}) & = & g(x_{01},x_{12}) \Psi^{0,1,2} 
x_{12}^{{\lambda-1}\over 2} R^{1,2} (\Psi^{0,2,1})^{-1}
g^{-1}(x_{02},x_{12})G(E^{0,2}) \\
\nonumber & & g(x_{02},x_{12}) \Psi^{0,2,1}x_{12}^{{\lambda-1}\over 2} 
R^{2,1} (\Psi^{0,1,2})^{-1}  g^{-1}(x_{01},x_{12}).
\end{eqnarray}
Now, by using $x_{21}=x_{12}$, $\sigma_1 x_{01} \sigma^{-1}_1=x_{02}$ and 
$\sigma_1 x_{12} \sigma^{-1}_1=x_{12}$, we get 
\begin{eqnarray*}
& g(x_{01},x_{12})\Psi^{0,1,2} x_{12}^{{\lambda-1}\over 2} 
R^{1,2} (\Psi^{0,2,1})^{-1}  g^{-1}(x_{02},x_{12}) \\
= & \sigma^{-1}_1 g(x_{02},x_{12})\sigma_1 x_{12}^{{\lambda-1}\over 2}  
\sigma_1 g^{-1}(x_{01},x_{12}) \sigma^{-1}_1 \Psi^{0,1,2} R^{1,2} (\Psi^{0,2,1})^{-1} \\
= & \sigma^{-1}_1 g(x_{02},x_{12}) x_{12}^{{\lambda+1}\over 2}   
g^{-1}(x_{01},x_{12}) \sigma^{-1}_1 \Psi^{0,1,2} R^{1,2} (\Psi^{0,2,1})^{-1}.
\end{eqnarray*}
Plugging this into equation \eqref{GtO} we obtain
\begin{eqnarray}\label{GtO2}
G(E^{01,2}) & = & \sigma^{-1}_1 g(x_{02},x_{12}) x_{12}^{{\lambda+1}\over 2}   
g^{-1}(x_{01},x_{12}) \sigma^{-1}_1 \Psi^{0,1,2} R^{1,2} 
(\Psi^{0,2,1})^{-1}G(E^{0,2}) \\
\nonumber & &\Psi^{0,2,1} R^{2,1}  (\Psi^{0,1,2})^{-1} \sigma^{-1}_1 g(x_{02},x_{12}) x_{12}^{{\lambda+1}\over 2}  
 g^{-1}(x_{01},x_{12}) \sigma^{-1}_1 .
\end{eqnarray}
Now, since $\Psi^{0,1,2} R^{1,2} (\Psi^{0,1,2})^{-1}$ is nothing but $\sigma_1$, we obtain 
\begin{eqnarray*}\label{GtO3}
G(E^{01,2}) & = & \sigma^{-1}_1 g(x_{02},x_{12}) x_{12}^{{\lambda+1}\over 2}   
g^{-1}(x_{01},x_{12}) \sigma^{-1}_1\cdot G(E^{0,2})   
g(x_{02},x_{12}) x_{12}^{{\lambda+1}\over 2}   g^{-1}(x_{01},x_{12}).
\end{eqnarray*}
Next, $G(E^{0,1})=x_{01}$ so, by using relation 
$\sigma_1^{-1} x_{02} = x_{01} \sigma_1^{-1}$, we obtain
\begin{eqnarray*}\label{GtO4}
G(E^{01,2}) & = & \sigma^{-1}_1 g(x_{02},x_{12}) 
x_{12}^{{\lambda+1}\over 2}   g^{-1}(x_{01},x_{12})  x_{01}^{\mu} 
 \sigma^{-1}_1 g(x_{02},x_{12}) x_{12}^{{\lambda+1}\over 2}   g^{-1}(x_{01},x_{12})\\
 & = & \sigma^{-1}_1 g(x_{02},x_{12}) 
 x_{12}^{{\lambda+1}\over 2}   g^{-1}(x_{01},x_{12}) x_{01}^{\mu}   
 g(x_{01},x_{12}) \sigma^{-1}_1 x_{12}^{{\lambda+1}\over 2}   
 \sigma^{-1}_1 g^{-1}(x_{02},x_{12})  \\
 & = & \sigma^{-1}_1 g(x_{02},x_{12}) x_{12}^{{\lambda+1}\over 2}   
 g^{-1}(x_{01},x_{12}) x_{01}^{\mu}   g(x_{01},x_{12}) 
 x_{12}^{{\lambda-1}\over 2} g^{-1}(x_{02},x_{12}) \sigma_1.
\end{eqnarray*}
The above equation is then equivalent to 
\begin{eqnarray*}\label{GtO5}
g(x_{02},x_{12}) x_{12}^{{\lambda+1}\over 2}   
g^{-1}(x_{01},x_{12}) x_{01}^{\mu}   g(x_{01},x_{12}) x_{12}^{{\lambda-1}\over 2} 
g^{-1}(x_{02},x_{12}) & = & \sigma_1 G(E^{01,2}) \sigma^{-1}_1  \\
 & = & \sigma_1 (zx_{01}^{-1})^{\mu} \sigma^{-1}_1 \\
& = & (zx_{02}^{-1})^{\mu}.
\end{eqnarray*}
Finally, by writing $x_{01}=zx_{02}^{-1}x_{12}^{-1}$,
we obtain, by absorbing the central element $z$ and simplifying 
it from the equation, the following result:
\begin{eqnarray*}\label{GtO6}
1& =&x_{02}^{\mu}g(x_{02},x_{12}) x_{12}^{{\lambda+1}\over 2}   
g^{-1}(x_{01},x_{12}) x_{02}^{-\mu}x_{12}^{-\mu} g(x_{01},
x_{12}) x_{12}^{{\lambda-1}\over 2} g^{-1}(x_{02},x_{12}).
\end{eqnarray*}
By denoting $x=x_{02}$, $y=x_{12}$ and $\tilde{z}=y^{-1}x^{-1} $  we obtain
$$
x^{\mu}g(x,y)y^{{{\lambda+1}\over 2}}g(\tilde{z},y)^{-1}\tilde{z}^{\mu}g(\tilde{z},y)y^{{{\lambda-1}\over 2}} g(x,y)^{-1}
=1.
$$
Finally, when $N\geq 1$, one takes the chosen lifts of each 
term of the above equation to obtain equation \eqref{eqn:GT:(tO)}.
\end{proof}

\begin{lemma}
We have $\lambda=1+\mu_1N$. 
\end{lemma}
\begin{proof}
It is proven in \cite{En} that, if we have a quadruple 
$(\lambda,\mu,f,g)$, with $(\lambda,f)\in\widehat{\on{GT}}(\KK)$, 
$\mu=(a,\mu_1)\in\Gamma\times\kk$, and $g\in \widehat{\ker\phi_N}(\KK)$, 
satisfying the above two equations 
\eqref{eqn:GT:(MP)} and \eqref{eqn:GT:(tO)}, then 
$\lambda=[\mu]:=\tilde{a}+\mu_1N$, where $0\leq\tilde{a}\leq N-1$ 
is a representative of $a\in\Gamma$. In our case, we 
are in the situation where $\mu=(\bar1,\mu_1)$. 
\end{proof}
As a consequence, we can identify the underlying set of our 
operadicly defined cyclotomic 
Grothendieck-Teichm\"uller group $\widehat{\GT}^{\Gamma}(\KK)$ 
with the underlying set 
of the group $\on{GTM}_{\bar1}(N,\kk)$ introduced in \cite{En}. 

\medskip

Indeed, \gls{bkGTg} is defined as the set of triples $(\lambda,f,g)$ with 
$(\lambda,f)\in\widehat{\on{GT}}(\KK)$ and $g\in \widehat{\ker\phi_N}(\KK)$, 
and satisying equations 
\eqref{eqn:GT:(MP)} and \eqref{eqn:GT:(tO)} with $\mu=\frac{\lambda-1}{N}$. 
It carries a group structure, which is defined by  
\[
(\lambda_1,f_1,g_1)*(\lambda_2,f_2,g_2)=(\lambda,f,g)
\]
with 
\begin{itemize}
\item $\lambda=\lambda_1\lambda_2$,
\item $f(x,y) = f_1(f_2(x, y)^{-1}x^{\lambda_2}f_2(x, y),y^{\lambda_2} )
 \cdot f_2(x, y) ,$
\item $g(x,y) = g_1(g_2(x, y)^{-1}x^{\mu_2}g_2(x, y),y^{\lambda_2} )
 \cdot g_2(x, y)$.
\end{itemize}

Therefore it follows from the above discussion that we get an assignement  
\begin{equation}\label{GT-vers-GT}
\widehat{\GT}^{\Gamma}(\KK)\longrightarrow 
\on{GTM}_{\bar1}(N,\kk)\,;\,(F,G)\longmapsto (\lambda,f,g)\,.
\end{equation}
It is obviously injective. 
\begin{proposition}\label{Cyc:GT}
The injective map \eqref{GT-vers-GT} is a group morphism.
\end{proposition}
We will see later in Theorem \ref{thm-last-torsor} that \eqref{GT-vers-GT} is actually an isomorphism. 
\begin{proof}
We have to prove that the assignment $(F,G)\to (\lambda,f,g)$ constructed above is a group morphism. 
As we already know, the composition of automorphisms $F_1$ and $F_2$ in
$\on{Aut}_{\on{Op}\mathbf{Grpd}_{\kk}}^+(\widehat{\PaB}(\KK))$
corresponds to the composition law in $\widehat{\on{GT}}(\KK)$, that is, the associated 
couples $(\lambda_1,f_1)$ and $(\lambda_2,f_2)$ in $\kk^\times \times \hat{\on{F}}_2(\kk)$ satisfy 
$$
(F_2\circ F_1)(R^{1,2})=(R^{1,2})^{\lambda_1\lambda_2}\qquad\textrm{and}
$$
$$
(F_2\circ F_1)(\Phi^{1,2,3})=f_1(x^{\lambda_2},f_2(x, y)y^{\lambda_2} f_2(x, y)^{-1}) f_2(x, y) \Phi^{1,2,3}
$$
(here $\on{F}_2$ is generated by $x:=x_{12}$ and $y:=x_{23}$). We also already showed 
that any two automorphisms $G_1$ and $G_2$ in the group $\widehat{\GT}^\Gamma(\kk)$, 
depending on $F_1$ and $F_2$ respectively, are associated to couples 
$(\mu_1,g_1(X | y(0),\ldots, y(N-1)))$ and $(\mu_2,g_2(X | y(0),\ldots, y(N-1)))$ where 
$g_1$ and $g_2$ are elements of in $\hat{\on{F}}_{N+1}(\KK)$. 

In the group $\mathbf{A}:=\on{Aut}_{\widehat{\PaB}^\Gamma(\kk)(2)}((01_{\bar0})2_{\bar0})$, we have 
$$
X=x_{01}^N=(E^{0,1_{\bar0}})^{(N)}
$$
and $G_1(X)=X^{\lambda_1}$ for $\lambda_1\in\kk^\times$.
Next, we want to compute $G_1(y(0))$, where $y(0)=x_{12}$ decomposes in $\mathbf{A}$ as follows:
$$
\xymatrix{
(01_{\bar0})2_{\bar0} \ar[rr]^{\Psi^{0,1_{\bar0},2_{\bar0}}} && 0(1_{\bar0}2_{\bar0}) \ar[rr]^{R^{1_{\bar0},2_{\bar0}}
R^{2_{\bar0},1_{\bar0}}} && 0(1_{\bar0}2_{\bar0}) \ar[rr]^{(\Psi^{0,1_{\bar0},2_{\bar0}})^{-1}} && (01_{\bar0})2_{\bar0}.
}
$$
Then, as 
$$
G_1(\Psi^{0,1_{\bar0},2_{\bar0}})= g_1(X|y(0),\ldots,y(N-1))\Psi^{0,1_{\bar0},2_{\bar0}}
$$
and 
$$
G_1(R^{1_{\bar0}2_{\bar0}}R^{2_{\bar0}1_{\bar0}})= (x_{12})^{\lambda_1},
$$
we obtain
$$
G_1(y(0))=g_1\big(X|y(0),\ldots,y(N-1)\big) y(0)^{\lambda_1} g_1^{-1}\big(X|y(0),\ldots,y(N-1)\big)\,.
$$
More generally, $y(a)=x_{01}^{-a} x_{12}x^{a}_{01}$ decomposes as 
$$
\xymatrix{
(01_{\bar0})2_{\bar0} \ar[r]^{(E^{0,1_{\bar0}})^{(a)}} & 
(01_{\bar{a}})2_{\bar0} \ar[r]^{\Psi^{0,1_{\bar{a}},2}} & 
0(1_{\bar{a}} 2_{\bar0}) \ar[rr]^{R^{1_{\bar{a}},2_{\bar0}}R^{2_{\bar0},1_{\bar{a}}}} & &
0(1_{\bar{a}} 2_{\bar0}) \ar[r]^{(\Psi^{0,1_{\bar{a}},2_{\bar0}})^{-1}} &  
(01_{\bar{a}})2_{\bar0} \ar[r]^{(E^{0,1_{\bar0}})^{(-a)}} & 
(01_{\bar0})2_{\bar0}\,.
}
$$
Therefore, by $\Gamma$-equivariance we get 
$$
x^{a}_{01}g\big(X | y(0), \ldots, y( N-1))\big)=
g\big(X | y(a), \ldots, y(a + N-1))\big)x^{a}_{01}\,,
$$
and thus 
\begin{eqnarray*} 
G_1(y(a))&=& G_1\big((E^{0,1_{\bar0}})^{(a)}\Psi^{0,1_{\bar{a}},2_{\bar0}}
R^{1_{\bar{a}},2_{\bar0}}R^{2_{\bar0},1_{\bar{a}}}(\Psi^{0,1_{\bar{a}},2_{\bar0}})^{-1}(E^{0,1_{\bar0}})^{(-a)}\big) \\
& = & G_1\big((E^{0,1_{\bar0}})^{(a)}\big)G_1(\Psi^{0,1_{\bar{a}},2_{\bar0}})G_1(R^{1_{\bar{a}},2_{\bar0}}
R^{2_{\bar0},1_{\bar{a}}})G_1\big(\Psi^{0,1_{\bar{a}},2_{\bar0}})^{-1}\big)G_1\big((E^{0,1_{\bar0}})^{(-a)}\big) \\
& = & \on{Ad}\Big(\big(X^{N\mu_1}E^{0,1_{\bar0}}\big)^{(a)} g\big(X | y(0), \ldots, y( N-1)\big) \Big) (x_{12}^{\lambda_1}) \\
& = & \on{Ad}\Big(X^{aN\mu_1}(E^{0,1_{\bar0}})^{(a)} g\big(X | y(0), \ldots, y( N-1)\big) \Big) (x_{12}^{\lambda_1}) \\
& = & \on{Ad} \Big( X^{ak_1} g\big(X | y(a), \ldots, y(a + N-1)\big) \Big) (y(a)^{\lambda_1})\,.
\end{eqnarray*}

Finally, we obtain
\begin{eqnarray*} 
(G_2 \circ G_1)(\Psi^{0,1_{\bar0},2_{\bar0}})
& = & G_2(g_1(X|y(0),\ldots,y(N-1))\Psi^{0,1_{\bar0},2_{\bar0}})\\
& = & g_1(G_2(X)|G_2(y(0)),\ldots,G_2(y(N-1)))g_2(X|y(0),\ldots,y(N-1)\Psi^{0,1_{\bar0},2_{\bar0}}\\
& = & g_1\Big( X^{\lambda_2} | \Ad(g_2(X | y(0),\ldots,y(N-1)))(y(0)^{\lambda_2}) , \\
&   & \Ad \big( X^{\lambda_2} g_2(X | y(1), \ldots, y( N))\big) (y(1)^{\lambda_2}), \ldots, \\ 
&   & \Ad \big( X^{(N-1)k_2} g_2(X | y(N-1), \ldots, y(2N-2)) \big) (y(N-1)^{\lambda_2}) \Big) \\
&   & g_2(X|y(0),\ldots,y(N-1))\Psi^{0,1_{\bar0},2_{\bar0}}\,. 
\end{eqnarray*}
which is nothing but the composition law in the group $\on{GTM}_{\bar1}(N,\kk)$. 
%This concludes the proof, as the composite of moperad morphisms $G_2 \circ G_1$ is compatible with the composition of operad morphisms $F_2 \circ F_1$.
\end{proof}

\section{The moperad of $N$-chord diagrams, and cyclotomic associators}
\label{Assocyclo}

% 5.1

\subsection{Infinitesimal cyclotomic braids}

Let $\Gamma=\Z/N\Z$, $I$ a finite set, and let $\t_I^{\Gamma}(\kk)$ 
be the Lie $\kk$-algebra with generators 
$t_{0i}$, ($i \in I$), and $t^{\alpha}_{ij}$, 
($i \neq j\in I$, $\alpha \in \Z/N\Z$), and relations: 
\begin{flalign}
& t^{\alpha}_{ij} = t^{-\alpha}_{ji}\,, 				\tag{tS}\label{eqn:tS} \\
& [t_{0i},t^{\alpha}_{jk}]=0\textrm{ and }[t^{\alpha}_{ij},t^{\beta}_{kl}]=0\,,		\tag{tL}\label{eqn:tL} \\
& [t^{\alpha}_{ij},t^{\alpha+\beta}_{ik}+t^{\beta}_{jk}] =0\,, 					\tag{t4T}\label{eqn:t4T} \\
& [t_{0i},t_{0j}+\sum_{\alpha \in\Z/N\Z}t^{\alpha}_{ij}]=0\,,						\tag{t4$\mathbb{T}$}\label{eqn:t4T0} \\
& [t_{0i}+t_{0j}+\sum_{\beta \in \Z/N\Z}t^{\beta}_{ij},t^{\alpha}_{ij}]=0\,, & 		\tag{t6$\mathbb{T}$}\label{eqn:t6T0}
\end{flalign}
where $i,j,k,l\in I$ are pairwise distinct and $\alpha, \beta \in\Z/N\Z$. 
We will call it the $\kk$-Lie algebra of \textit{infinitesimal cyclotomic braids}. 
This definition is obviously functorial with respect to bijections, exhibiting 
$\t^{\Gamma}(\kk):=\{\t_I^{\Gamma}(\kk)\}_I$ as an $\mathfrak{S}$-module. 
It moreover also has the structure of a $\t(\kk)$-moperad, where 
partial compositions are defined as follows\footnote{We just re-package 
Enriquez's insertion-coproduct morphisms \cite[\S2.1.1]{En} in a moperadic manner. }: for $i\in I$, 
\[
\begin{array}{cccc}
\circ_i: & \t^{\Gamma}_I(\kk) \oplus \t_J(\kk)  & \longrightarrow &  
\t^{\Gamma}_{J\sqcup I-\{i\}}(\kk) \\
	 & (0,t_{p q}) 				& \longmapsto 	  &   t_{pq}^0 \\
	 & (t^{\alpha}_{jk},0) 			& \longmapsto 	  & 
\begin{cases}
  \begin{tabular}{ccc}
  $t^{\alpha}_{jk}$ 			  & if & $ i\notin\{j,k\} $ \\
  $\sum\limits_{r\in J} t^{\alpha}_{r k}$ & if & $j=i$ \\
  $\sum\limits_{r\in J} t^{\alpha}_{jr}$  & if & $k=i$ 
  \end{tabular}
  \end{cases}\\
	 & (t_{0j},0) 				& \longmapsto & 
\begin{cases}
  \begin{tabular}{ccc}
  $t_{0j}$ 				  & if & $ j\neq i $ \\
  $\sum\limits_{p\in J} t_{0p}+\frac12\sum\limits_{\gamma\in\Gamma}\sum\limits_{\substack{q,r\in J \\ q\neq r}}t_{qr}^\gamma$ 	  & if & $j=i$ 
  \end{tabular}
  \end{cases}
\end{array}
\]
and 
\[
\begin{array}{cccc}
  \circ_0: 	& \t^{\Gamma}_I(\kk) \oplus \t^{\Gamma}_J(\kk)  & 
  \longrightarrow &  \t^{\Gamma}_{J\sqcup I}(\kk) \\
    		& (0,t_{0p}) 					& \longmapsto 	  &   t_{0p} \\
		& (0,t^\alpha_{p q})				& \longmapsto 	  &   t^{\alpha}_{p  q }\\
		& (t^{\alpha}_{jk},0) 				& \longmapsto 	  &   t^{\alpha}_{jk} \\
		& (t_{0i},0) 					& \longmapsto 	  &   t_{0i}+\sum\limits_{\gamma\in\Gamma}\sum\limits_{j\in J}t_{ji}^\gamma \\
  \end{array}
\]
We call $\t^{\Gamma}({\kk})$ the moperad of infinitesimal cyclotomic braids. It is acted on by $\Gamma$: 
for $\gamma\in\Gamma$ and $1\leq i\leq n$, $\gamma_i\in\Gamma^n$ acts as 
\begin{eqnarray*}
& \gamma_i\cdot t_{0p}=t_{0p} 							& (p\in\{1,\dots,n\})\,, 						\\
& \gamma_i\cdot t^\alpha_{qr}=t^\alpha_{qr}				& (\alpha\in\Gamma\textrm{ and }q,r\neq i)\,, 	\\
& \gamma_i\cdot t^\alpha_{ir}=t^{\alpha+\gamma}_{ir}	& (\alpha\in\Gamma\textrm{ and }r\neq i)_,, 	\\
& \gamma_i\cdot t^\alpha_{qi}=t^{\alpha-\gamma}_{qi}	& (\alpha\in\Gamma\textrm{ and }q\neq i)\,.
\end{eqnarray*}

\begin{remark}
When $\Gamma=1$, our moperad $\t^{1}({\kk})$ and the ``shifted'' moperad $\t^+$ from \cite{DHLARS} are different. More precisely the $\circ_i$ operation sends $(t_{0i},0)$ to 
\begin{itemize}
\item $\sum\limits_{p\in J} t_{0p}$ for $\t^+$. 
\item $\sum\limits_{p\in J} t_{0p}+\frac12\sum\limits_{\substack{q,r\in J \\ q\neq r}}t_{qr}$ for $\t^{1}({\kk})$. 
\end{itemize}
This is again an indirect consequence of different choices of framing, as explained in Remarks \ref{Lien avec Marcy and co} and \ref{Lien avec Marcy and co 2}. Both moperads of Lie algebras can be obtained by considering holonomy Lie algebras for the compactified configuration spaces of $\mathbb{C}^\times$, where the moperad structure depends on the choice of framing. 
\end{remark}

% 5.2

\subsection{Horizontal $N$-chord diagrams}

We now consider the $\CD(\KK)$-moperad $\CD_0^\Gamma(\KK) := \hat{\mathcal{U}}(\t^\Gamma({\kk}))$ 
in $\mathbf{Cat(CoAlg_\KK)}$. 
Morphisms in $\CD^\Gamma_0(\KK)(n)$ can be represented as linear combinations of diagrams of chords on 
$n+1$ vertical strands, together with a labelling (by elements of $\Gamma$) of chords that are not 
connected to the leftmost strand (i.e.~the frozen one). 
On can equivalently represent them as certain horizontal $N$-diagrams according to the terminology of 
\cite[Definition 6.4]{Adrien} (where the relation to Vassiliev invariants has been explored): more precisely, 
those horizontal $N$-diagrams for which the sum of labels on each strand is $\bar0$. 
Using the representation from \cite{Adrien}, i.e.~the one with labels on (non frozen) strands rather than on chords, 
the diagram corresponding to $t_{0i}$ is 
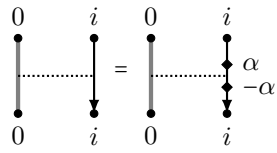
\begin{align*}
	\tik{\tzero[->]{0}{1}	\node[point, label=above:$0$] at (0,0) {};
										\node[point, label=above:$i$] at (1,0) {};
										\node[point, label=below:$0$] at (0,-1) {};
										\node[point, label=below:$i$] at (1,-1) {};}
&=\tik{\tzero[->]{0}{1}	\node[point, label=above:$0$] at (0,0) {};
										\node[point, label=above:$i$] at (1,0) {};
										\node[point, label=below:$0$] at (0,-1) {};
										\node[point, label=below:$i$] at (1,-1) {};
										\node[diam, label=right:$\alpha$] at (1,-0.35) {};
										\node[diam, label=right:$-\alpha$] at (1,-0.65) {};}
\end{align*}
while the one corresponding to $t_{ij}^\alpha=t_{ji}^{-\alpha}$ is 
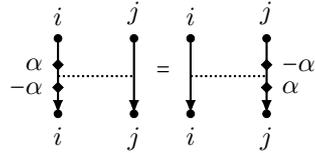
\begin{align*}
	\tik{\hori[->]{0}{0}{1}	\node[point, label=above:$i$] at (0,0) {};
											\node[point, label=above:$j$] at (1,0) {};
											\node[point, label=below:$i$] at (0,-1) {};
											\node[point, label=below:$j$] at (1,-1) {};
											\node[diam, label=left:$\alpha$] at (0,-0.35) {};
											\node[diam, label=left:$-\alpha$] at (0,-0.65) {};}
&=\tik{\hori[->]{0}{0}{1} \node[point, label=above:$i$] at (0,0) {};
											\node[point, label=above:$j$] at (1,0) {};
											\node[point, label=below:$i$] at (0,-1) {};
											\node[point, label=below:$j$] at (1,-1) {};
											\node[diam, label=right:$\alpha$] at (1,-0.65) {};
											\node[diam, label=right:$-\alpha$] at (1,-0.35) {};}
\end{align*}
and relations can be depicted as follows: 
\begin{align}\tag{\ref{eqn:tL}}
\tik{ 
\hori{0}{0}{1}\straight[->]{0}{1}; \hori[->]{2}{1}{1}\straight[->]{1}{1}\straight{2}{0}\straight{3}{0}
	\node[point, label=above:$j$] at (1,0) {}; 
	\node[point, label=above:$k$] at (2,0) {};
	\node[point, label=above:$i$] at (0,0) {}; 
	\node[point, label=above:$l$] at (3,0) {};
	\node[point, label=below:$i$] at (0,-2) {}; 
	\node[point, label=below:$l$] at (3,-2) {};
	\node[point, label=below:$j$] at (1,-2) {}; 
	\node[point, label=below:$k$] at (2,-2) {};
	\node[diam, label=left:$\alpha$] at (0,-0.35) {};
	\node[diam, label=left:$-\alpha$] at (0,-0.65) {};
	\node[diam, label=left:$\beta$] at (2,-1.35) {};
	\node[diam, label=left:$-\beta$] at (2,-1.65) {};
}
= \tik{
\hori[->]{0}{1}{1}\straight{0}{0}; \hori{2}{0}{1}\straight{1}{0}\straight[->]{2}{1}\straight[->]{3}{1}\node[point, label=above:$j$] at (1,0) {}; 
\node[point, label=above:$k$] at (2,0) {};
\node[point, label=above:$i$] at (0,0) {}; 
\node[point, label=above:$l$] at (3,0) {};
\node[point, label=below:$i$] at (0,-2) {}; 
\node[point, label=below:$l$] at (3,-2) {};
\node[point, label=below:$j$] at (1,-2) {};
\node[point, label=below:$k$] at (2,-2) {};
	\node[diam, label=left:$\alpha$] at (0,-1.35) {};
	\node[diam, label=left:$-\alpha$] at (0,-1.65) {};
	\node[diam, label=left:$\beta$] at (2,-0.35) {};
	\node[diam, label=left:$-\beta$] at (2,-0.65) {};
}
\end{align}
\begin{align*}
\tik{ 
\tzero[-]{0}{1} \fixch{1}  \hori[->]{2}{1}{1}\straight[->]{1}{1}\straight{2}{0}\straight{3}{0}
	\node[point, label=above:$i$] at (1,0) {}; 
	\node[point, label=above:$j$] at (2,0) {};
	\node[point, label=above:$0$] at (0,0) {}; 
	\node[point, label=above:$k$] at (3,0) {};
	\node[point, label=below:$0$] at (0,-2) {}; 
	\node[point, label=below:$k$] at (3,-2) {};
	\node[point, label=below:$i$] at (1,-2) {}; 
	\node[point, label=below:$j$] at (2,-2) {};
	\node[diam, label=left:$\alpha$] at (2,-1.35) {};
	\node[diam, label=left:$-\alpha$] at (2,-1.65) {};
}
= \tik{
\fixch{0}    \tzero[->]{1}{1};\hori{2}{0}{1}\straight{1}{0}\straight[->]{2}{1}\straight[->]{3}{1}
	\node[point, label=above:$i$] at (1,0) {}; 
	\node[point, label=above:$j$] at (2,0) {};
	\node[point, label=above:$0$] at (0,0) {}; 
	\node[point, label=above:$k$] at (3,0) {};
	\node[point, label=below:$0$] at (0,-2) {}; 
	\node[point, label=below:$k$] at (3,-2) {};
	\node[point, label=below:$i$] at (1,-2) {}; 
	\node[point, label=below:$j$] at (2,-2) {};
	\node[diam, label=left:$\alpha$] at (2,-0.35) {};
	\node[diam, label=left:$-\alpha$] at (2,-0.65) {};
}
\end{align*}

\begin{align}\tag{\ref{eqn:t4T}}
		\tik{
		\hori{0}{0}{1}\straight{2}{0}\hori[->]{0}{1}{2}\straight[->]{1}{1}
			\node[point, label=above:$i$] at (0,0) {};
			\node[point, label=above:$j$] at (1,0) {};
			\node[point, label=above:$k$] at (2,0) {};
			\node[point, label=below:$i$] at (0,-2) {};
			\node[point, label=below:$j$] at (1,-2) {};
			\node[point, label=below:$k$] at (2,-2) {};
			\node[diam, label=left:$\alpha$] at (0,-0.35) {};
			\node[diam, label=left:$\beta$] at (0,-1) {};
			\node[diam, label=left:$-\alpha-\beta$] at (0,-1.65) {};
		}
	+	\tik{
	\hori{0}{0}{1} \straight{2}{0}\hori[->]{1}{1}{1}\straight[->]{0}{1}
		\node[point, label=above:$i$] at (0,0) {};
			\node[point, label=above:$j$] at (1,0) {};
			\node[point, label=above:$k$] at (2,0) {};
			\node[point, label=below:$i$] at (0,-2) {};
			\node[point, label=below:$j$] at (1,-2) {};
			\node[point, label=below:$k$] at (2,-2) {};
			\node[diam, label=left:$\alpha$] at (0,-0.35) {};
			\node[diam, label=left:$-\alpha$] at (0,-1) {};
			\node[diam, label=left:$-\beta$] at (1,-1) {};
			\node[diam, label=left:$-\beta$] at (1,-1.65) {};
	}
	=	\tik{
	\hori[->]{0}{1}{1}\straight[->]{2}{1}\hori{0}{0}{2}\straight{1}{0}
	\node[point, label=above:$i$] at (0,0) {};
			\node[point, label=above:$j$] at (1,0) {};
			\node[point, label=above:$k$] at (2,0) {};
			\node[point, label=below:$i$] at (0,-2) {};
			\node[point, label=below:$j$] at (1,-2) {};
			\node[point, label=below:$k$] at (2,-2) {};	
			\node[diam, label=left:$\alpha+\beta$] at (0,-0.35) {};
			\node[diam, label=left:$-\beta$] at (0,-1) {};
			\node[diam, label=left:$-\alpha$] at (0,-1.65) {};
	}
	+	\tik{
	\hori[->]{0}{1}{1} \straight[->]{2}{1}\hori{1}{0}{1}\straight{0}{0}
	\node[point, label=above:$i$] at (0,0) {};
			\node[point, label=above:$j$] at (1,0) {};
			\node[point, label=above:$k$] at (2,0) {};
			\node[point, label=below:$i$] at (0,-2) {};
			\node[point, label=below:$j$] at (1,-2) {};
			\node[point, label=below:$k$] at (2,-2) {};	
				\node[diam, label=left:$\beta$] at (1,-0.35) {};
			\node[diam, label=left:$-\beta$] at (1,-1) {};
			\node[diam, label=left:$\alpha$] at (0,-1) {};
			\node[diam, label=left:$-\alpha$] at (0,-1.65) {};
	}
\end{align}

\begin{align}\tag{\ref{eqn:t4T0}}
		\tik{
		\tzero{0}{1}\straight{2}{0}\tzero[->]{1}{2}\straight[->]{1}{1}
			\node[point, label=above:$0$] at (0,0) {};			
			\node[point, label=above:$i$] at (1,0) {};
			\node[point, label=above:$j$] at (2,0) {};
			\node[point, label=below:$0$] at (0,-2) {};
			\node[point, label=below:$i$] at (1,-2) {};
			\node[point, label=below:$j$] at (2,-2) {};	
		}
	+	\sum_\alpha \tik{
	\tzero{0}{1}\straight{2}{0}\hori[->]{1}{1}{1}\fixch{1}
				\node[point, label=above:$0$] at (0,0) {};			
			\node[point, label=above:$i$] at (1,0) {};
			\node[point, label=above:$j$] at (2,0) {};
			\node[point, label=below:$0$] at (0,-2) {};
			\node[point, label=below:$i$] at (1,-2) {};
			\node[point, label=below:$j$] at (2,-2) {};	
					\node[diam, label=left:$\alpha$] at (1,-1) {};
			\node[diam, label=left:$-\alpha$] at (1,-1.65) {};
	}
&	=	\tik{
\tzero[->]{1}{1}\straight[->]{2}{1}\tzero{0}{2}\straight{1}{0}
			\node[point, label=above:$0$] at (0,0) {};			
			\node[point, label=above:$i$] at (1,0) {};
			\node[point, label=above:$j$] at (2,0) {};
			\node[point, label=below:$0$] at (0,-2) {};
			\node[point, label=below:$i$] at (1,-2) {};
			\node[point, label=below:$j$] at (2,-2) {};	
}
	+	\sum_\alpha \tik{
	\tzero[->]{1}{1}\straight[->]{2}{1}\hori{1}{0}{1}\fixch{0}	
				\node[point, label=above:$0$] at (0,0) {};			
			\node[point, label=above:$i$] at (1,0) {};
			\node[point, label=above:$j$] at (2,0) {};
			\node[point, label=below:$0$] at (0,-2) {};
			\node[point, label=below:$i$] at (1,-2) {};
			\node[point, label=below:$j$] at (2,-2) {};	
				\node[diam, label=left:$\alpha$] at (1,-0.35) {};
			\node[diam, label=left:$-\alpha$] at (1,-1) {};
	}
\end{align}

\begin{align*}\tag{\ref{eqn:t6T0}}
		\tik{
		\hori{1}{0}{1}\fixch{0}\tzero[->]{1}{1}\straight[->]{1}{1}\straight[->]{2}{1}
					\node[point, label=above:$0$] at (0,0) {};			
			\node[point, label=above:$i$] at (1,0) {};
			\node[point, label=above:$j$] at (2,0) {};
			\node[point, label=below:$0$] at (0,-2) {};
			\node[point, label=below:$i$] at (1,-2) {};
			\node[point, label=below:$j$] at (2,-2) {};	
				\node[diam, label=left:$\alpha$] at (1,-0.35) {};
			\node[diam, label=left:$-\alpha$] at (1,-1) {};
		}
	+ \tik{
	\hori{1}{0}{1}\fixch{0}\tzero[->]{1}{2}\straight[->]{1}{1}
				\node[point, label=above:$0$] at (0,0) {};			
			\node[point, label=above:$i$] at (1,0) {};
			\node[point, label=above:$j$] at (2,0) {};
			\node[point, label=below:$0$] at (0,-2) {};
			\node[point, label=below:$i$] at (1,-2) {};
			\node[point, label=below:$j$] at (2,-2) {};	
							\node[diam, label=left:$\alpha$] at (1,-0.35) {};
			\node[diam, label=left:$-\alpha$] at (1,-1) {};
	}
	+	\sum_\beta \tik{
	\hori{1}{0}{1}\fixch{0}\hori[->]{1}{1}{1}\fixch{1}
				\node[point, label=above:$0$] at (0,0) {};			
			\node[point, label=above:$i$] at (1,0) {};
			\node[point, label=above:$j$] at (2,0) {};
			\node[point, label=below:$0$] at (0,-2) {};
			\node[point, label=below:$i$] at (1,-2) {};
			\node[point, label=below:$j$] at (2,-2) {};	
			\node[diam, label=left:$\alpha$] at (1,-0.35) {};
			\node[diam, label=left:$\beta-\alpha$] at (1,-1) {};
			\node[diam, label=left:$-\beta$] at (1,-1.65) {};
	} \\
	=	\tik{
\hori[->]{1}{1}{1}\fixch{1}\tzero{0}{1}\straight{1}{0}\straight{2}{0}
			\node[point, label=above:$0$] at (0,0) {};			
			\node[point, label=above:$i$] at (1,0) {};
			\node[point, label=above:$j$] at (2,0) {};
			\node[point, label=below:$0$] at (0,-2) {};
			\node[point, label=below:$i$] at (1,-2) {};
			\node[point, label=below:$j$] at (2,-2) {};	
			\node[diam, label=left:$\alpha$] at (1,-1) {};
			\node[diam, label=left:$-\alpha$] at (1,-1.65) {};
}
	+	\tik{
	\hori[->]{1}{1}{1}\fixch{1}\tzero{0}{2}\straight{1}{0}
				\node[point, label=above:$0$] at (0,0) {};			
			\node[point, label=above:$i$] at (1,0) {};
			\node[point, label=above:$j$] at (2,0) {};
			\node[point, label=below:$0$] at (0,-2) {};
			\node[point, label=below:$i$] at (1,-2) {};
			\node[point, label=below:$j$] at (2,-2) {};	
			\node[diam, label=left:$\alpha$] at (1,-1) {};
			\node[diam, label=left:$-\alpha$] at (1,-1.65) {};
	}
	+	\sum_\beta \tik{
	\hori{1}{0}{1}\fixch{0}\hori[->]{1}{1}{1}\fixch{1}
				\node[point, label=above:$0$] at (0,0) {};			
			\node[point, label=above:$i$] at (1,0) {};
			\node[point, label=above:$j$] at (2,0) {};
			\node[point, label=below:$0$] at (0,-2) {};
			\node[point, label=below:$i$] at (1,-2) {};
			\node[point, label=below:$j$] at (2,-2) {};	
			\node[diam, label=left:$\beta$] at (1,-0.35) {};
			\node[diam, label=left:$\alpha-\beta$] at (1,-1) {};
			\node[diam, label=left:$-\alpha$] at (1,-1.65) {};
	}
\end{align*}

Let us now introduce another $\CD(\KK)$-moperad $\CD^\Gamma(\KK)$, which will be 
made of all horizontal $N$-diagrams. In arity $n$, objects of 
$\CD^\Gamma(\KK)$ are just labellings $\{1,\dots,n\}\to\Gamma$, 
$\on{Ob}(\CD^\Gamma(\KK))(n)=\Gamma^n$, and the $*$-moperad structure 
is given as follows: 
\begin{itemize}
\item $\circ_0:\Gamma^n\times\Gamma^m\to\Gamma^{n+m}$ 
is the concatenation of sequences, 
\item for every $i\neq 0$, $\circ_i:\Gamma^n\to\Gamma^{n+m-1}$ 
is the partial diagonal 
$$
(\alpha_1,\dots,\alpha_n)\longmapsto (\alpha_1,\dots,\alpha_{i-1},
\underbrace{\alpha_i,\dots,\alpha_i}_{m~\textrm{times}},\alpha_{i+1},\dots,\alpha_n)\,.
$$
\end{itemize}
Given two labellings $\underline{\alpha}=(\alpha_1,\dots,\alpha_n)$ and 
$\underline{\beta}=(\beta_1,\dots,\beta_n)$, the $\kk$-vector space of morphisms 
from $\underline{\alpha}$ to $\underline{\beta}$ in $\CD^\Gamma(\KK)$ is 
the vector space of horizontal $N$-chord diagrams such that, on the $i$-th strand, 
the sum of labels equals $\beta_i-\alpha_i$. 
The $\CD(\kk)$-moperad structure on morphisms is the exact same as the one for $\CD_0^\Gamma(\kk)$. 

We call $\CD^\Gamma(\kk)$ the $\CD(\kk)$-moperad of \textit{$N$-chord diagrams}. It carries an obvious 
action of $\Gamma$, by translation on the labelling of objects. 

\begin{example}[Notable arrows of $\CD^{\Gamma}(\kk)(1)$]
We have the following arrows in $\CD^{\Gamma}(\kk)(1)$: 

\begin{center}
$K^{0,1_{\alpha}}:=t_{01}\cdot$
\begin{tikzpicture}[baseline=(current bounding box.center)]
\tikzstyle point=[circle, fill=black, inner sep=0.05cm]
 \node[point, label=above:$1_{\alpha}$] at (1,1) {};
 \node[point, label=below:$1_{\alpha}$] at (1,-0.25) {};
 \draw[->,thick] (1,1) to (1,-0.20); 
 \draw[zero] (0,1) -- (0,-0.25);
  \node[point, label=above:$0$] at (0,1) {};
 \node[point, label=below:$0$] at (0,-0.25) {};
\end{tikzpicture}
$=$
\begin{tikzpicture}[baseline=(current bounding box.center)]
\tikzstyle point=[circle, fill=black, inner sep=0.05cm]
 \node[point, label=above:$1_{\alpha}$] at (1,1) {};
 \node[point, label=below:$1_{\alpha}$] at (1,-0.25) {};
 \draw[->,thick] (1,1) to (1,-0.20); 
 \draw[zero] (0,1) -- (0,-0.25);
 \draw[densely dotted, thick] (0,0.38) to (1,0.38); 
  \node[point, label=above:$0$] at (0,1) {};
 \node[point, label=below:$0$] at (0,-0.25) {};
\end{tikzpicture}
\qquad
$L^{0,1_{\alpha}}:=$
\begin{tikzpicture}[baseline=(current bounding box.center)]
\tikzstyle point=[circle, fill=black, inner sep=0.05cm]
 \node[point, label=above:$1_{\alpha}$] at (1,1) {};
 \node[point, label=below:$1_{\alpha+\bar1}$] at (1,-0.25) {};
 \draw[->,thick] (1,1) to (1,-0.20); 
 \draw[zero] (0,1) -- (0,-0.25);
 \node[diam, label=right:$\bar{1}$] at (1,0.5) {};
  \node[point, label=above:$0$] at (0,1) {};
 \node[point, label=below:$0$] at (0,-0.25) {};
\end{tikzpicture}
\end{center}
\end{example}
Note that, by definition,  
\[
t_{12}^{\bar{a}}=\on{Ad}_{(L^{0,1_{\bar0}})^{(a)}}(t_{12}^0)\,
\]
for every $a\in\mathbb{N}$. 

There is an obvious $\Gamma$-equivariant morphism of moperads $\omega:\Pa_0^\Gamma\to\on{Ob}(\CD^{\Gamma}(\kk))$, 
over the terminal operad morphism $\Pa\to*=\on{Ob}(\CD(\kk))$, that forgets the underlying parenthsized 
permutation and just remembers the labelling.  Hence we can consider the fake pull-back $\PaCD(\KK)$-moperad 
$$
\text{\gls{PaCDg}}:=\omega^\star \CD^\Gamma(\kk)
$$
of \textit{parenthesized $N$-chord diagrams}, which is still acted on by $\Gamma$.  

\begin{example}[Notable arrow in $\mathbf{PaCD}^{\Gamma}(\kk)(2)$]
We also have the following arrow in $\mathbf{PaCD}^{\Gamma}(\kk)(2)$: 
\begin{center}
$b^{0,1_{\bar0},2_{\bar0}}:=1\cdot$
\begin{tikzpicture}[baseline=(current bounding box.center)]
\tikzstyle point=[circle, fill=black, inner sep=0.05cm]
 \node[point, label=above:$1_{\bar0})$] at (1.5,1) {};
 \node[point, label=below:$(1_{\bar0}$] at (3.5,-0.25) {};
 \node[point, label=above:$2_{\bar0}$] at (4,1) {};
 \node[point, label=below:$2_{\bar0})$] at (4,-0.25) {};
 \draw[zero] (1,1) -- (1,-0.25); 
 \draw[->,thick] (1.5,1) .. controls (1.5,1) and (3.5,-0.20).. (3.5,-0.20);
 \draw[->,thick] (4,1) .. controls (4,0) and (4,0).. (4,-0.20);
  \node[point, label=above:$(0$] at (1,1) {};
 \node[point, label=below:$0$] at (1,-0.25) {};
 \end{tikzpicture} 
\end{center}
\end{example}

\begin{remark}\label{PaCD:cyc:rel}
As explained in subsection \ref{sec-pointings}, there is a map of $\mathfrak{S}$-modules 
$\PaCD \longrightarrow \PaCD^\Gamma$. Following the same convention as before,  we denote by 
$X^{1_{\bar0},2_{\bar0}}$, $H^{1_{\bar0},2_{\bar0}}$ and $a^{1_{\bar0},2_{\bar0},3_{\bar0}}$ 
the images in $\PaCD^\Gamma$ of the corresponding arrows in $\PaCD$.  
The elements $K^{0,1_{\bar0}}$, $L^{0,1_{\bar0}}$ and $b^{0,1_{\bar0},2_{\bar0}}$ are generators 
of the ${\PaCD}(\kk)$-moperad with $\Gamma$-action ${\PaCD}^{\Gamma}(\kk)$. They satisfy the following relations: 
\begin{flalign}
& (L^{0,1_{\bar0}})^{(N)} := L^{0,1_{\bar0}}\cdot L^{0,1_{\bar1}}\cdots \cdot L^{0,1_{\overline{N-1}}}=Id_{01_{\bar{0}}}\,, 
\label{array2}\\
& L^{0,1_{\bar0}}K^{0,1_{\bar1}} = K^{0,1_{\bar0}}L^{0,1_{\bar1}}\,, 
\label{array1}\\
& b^{01_{\bar0},2_{\bar0},3_{\bar0}} b^{0,1_{\bar0},2_{\bar0}3_{\bar0}} = 
b^{0,1_{\bar0},2_{\bar0}} b^{0,1_{\bar0}2,3_{\bar0}} a^{1_{\bar0},2_{\bar0},3_{\bar0}}\,, 
\label{array3}\\
& b^{0,1_{\bar0},2_{\bar0}}L^{0,1_{\bar0}2_{\bar0}}(b^{0,1_{\bar1},2_{\bar1}})^{-1} = L^{0,1_{\bar0}}L^{01_{\bar1},2_{\bar0}} \,, 
\label{array1ter}\\
& L^{01_{\bar0},2_{\bar0}} = b^{0,1_{\bar0},2_{\bar0}}X^{1_{\bar0},2_{\bar0}}(b^{0,2_{\bar0},1_{\bar0}})^{-1}L^{0,2_{\bar0}}b^{0,2_{\bar1},1_{\bar0}}X^{2_{\bar1},1_{\bar0}}(b^{0,1_{\bar0},2_{\bar1}})^{-1} \,, 
\label{array1bis}\\
& K^{0,1_{\bar0}2_{\bar0}} = (b^{0,1_{\bar0},2_{\bar0}})^{-1}K^{0,1_{\bar0}}b^{0,1_{\bar0},2_{\bar0}}
+X^{1_{\bar0},2_{\bar0}}(b^{0,2_{\bar0},1_{\bar0}})^{-1}K^{0,2_{\bar0}}b^{0,2_{\bar0},1_{\bar0}}X^{2_{\bar0},1_{\bar0}} 
\label{array5}\\
& \qquad\qquad +\sum_{k=0}^{N-1}(L^{0,1_{\bar0}})^{(k)}H^{1_{\bar{k}},2_{\bar0}}(L^{0,1_{\bar0}})^{(-k)}\,, \nonumber \\
& K^{01_{\bar0},2_{\bar0}} = b^{0,1_{\bar0},2_{\bar0}}\big(X^{1_{\bar0},2_{\bar0}}(b^{0,2_{\bar0},1_{\bar0}})^{-1}
K^{0,2_{\bar0}}b^{0,2_{\bar0},1_{\bar0}}X^{2_{\bar0},1_{\bar0}}\big)(b^{0,1_{\bar0},2_{\bar0}})^{-1} 
\label{array4}\\
& \qquad\qquad +\sum_{k=0}^{N-1}\big((L^{0,1_{\bar0}})^{(k)}b^{0,1_{\bar{k}},2_{\bar0}}\big)H^{1_{\bar{k}},2_{\bar0}}
\big((L^{0,1_{\bar0}})^{(k)}b^{0,1_{\bar{k}},2_{\bar0}}\big)^{-1}\,. \nonumber 
\end{flalign}
\end{remark}

% 5.3

\subsection{Cyclotomic associators}

We follow the same notation as in \S\ref{sec:cycloGT}. 

\begin{definition}
A cyclotomic associator is a couple $(F,G)$ where $F$ is in 
$\Assoc(\kk)$ and $G$ is a $\Gamma$-equivariant 
isomorphism of $\widehat{\PaB}(\KK)$-moperads from 
$\widehat{\PaB}^{\Gamma}(\KK)$ to $G \PaCD^{\Gamma}(\KK)$ which 
is the identity on objects (the $\widehat{\PaB}(\KK)$-moperad 
structure on $G \PaCD^{\Gamma}(\KK)$ is given by $F$). We denote by
\[
\text{\gls{Assg}}:=\on{Iso}^+_{\on{Mop}}
\Big(\big(\widehat{\PaB}(\KK),\widehat{\PaB}^{\Gamma}(\KK)\big)\,,\,\big(G\PaCD(\KK),G\PaCD^{\Gamma}(\KK)\big)\Big)^{\Gamma}
\]
the set of cyclotomic associators. 
\end{definition}

\noindent\textbf{Notation.}
The Lie algebra $\t_{2}^\Gamma(\kk)$ is the direct sum of its center, that is one dimensional 
and generated by $c:=t_{01} + t_{02} + \sum_{\alpha\in\Gamma} t^\alpha_{12}$, with 
the free Lie algebra $\f_{N+1}(\kk)=\f(\kk)(t_{01},t_{12}^0,...,t_{12}^{N-1})$ 
generated by $t_{01}$ and the $t^\alpha_{12}$'s ($\alpha\in\Gamma$). 
The quotient of $\t_{2}^\Gamma(\kk)$ by its center will be 
denoted $\bar{\t}_2^\Gamma(\kk)$, and is thus isomorphic to $\f_{N+1}(\kk)$. 
For every $\psi\in \bar{\t}_2^\Gamma(\kk)$, we write $\psi^{0,1,2}:=\psi(t_{01},t_{12}^0,...,t_{12}^{N-1})$. 

\medskip

We then have the following theorem:
\begin{theorem}\label{Ass:cyc:iso}
There is a one-to-one correspondence between elements of 
$\Assoc^{\Gamma}({\KK})$ and those of the 
set $\text{\gls{bAssg}}$ consisting of triples
$(\mu,\varphi,\psi)\in  \kk^{\times} \times 
\on{exp}(\hat{\bar{\t}}_3(\kk)) \times \on{exp}(\hat{\bar{\t}}_2^\Gamma(\kk))$,
such that $(\mu,\varphi)\in \on{Ass}(\kk)$ and $\psi$ satisfies 
\begin{flalign}
& \psi^{01,2,3}\psi^{0,1,23}   =   \psi^{0,1,2}\psi^{0,12,3} \varphi^{1,2,3}\,, 
\quad \text{in }   \on{exp}(\hat{\bar{\t}}_3^\Gamma(\kk)) \label{eqn:pseudotwist} \\
& e^{\frac{\mu}{N}t_{01}}\psi^{0,1,2}e^{\frac{\mu}{2}t_{12}^0} (\psi^{0,2,1})^{-1}  
e^{\frac{\mu}{N}t_{02}} \alpha \cdot \left( \psi^{0,2,1} 
e^{\frac{\mu}{2}t^0_{12}}(\psi^{0,1,2})^{-1} \right) =  1\, \quad \text{in }    
\on{exp}(\hat{\bar{\t}}_2^\Gamma(\kk)), & \label{eqn:octogon}
\end{flalign}
where $\alpha=(\bar{0},\bar{1})\in\Gamma^2$.
\end{theorem}
\begin{proof}
Let $\tilde{F}$ be a $\kk$-associator $\widehat{\PaB}(\KK) \longrightarrow G \PaCD(\KK)$, and let $\tilde{G}$ 
be a $\Gamma$-equivariant isomorphism 
\[
\widehat{\PaB}^{\Gamma}(\KK) \longrightarrow G \PaCD^{\Gamma}(\KK)
\]
of $\widehat{\PaB}(\KK)$-moperads, which is the identity on objects. It corresponds to a unique 
$\Gamma$-equivariant morphism $G:\PaB^{\Gamma}\longrightarrow G \PaCD^{\Gamma}(\KK)$. 
From the presentation of $\PaB^{\Gamma}$, we know that $G$ is uniquely determined by the images 
of the morphisms $E^{0,1_{\bar0}} \in \on{Hom}_{\PaB^{\Gamma}(\KK)(1)}(01_0,01_1)$ and 
$\Psi^{0,1_{\bar0},2_{\bar0}} \in \on{Hom}_{\PaB^{\Gamma}(\KK)(2)}((01_0)2_0,0(1_02_0))$. 
Thus, there are elements $u\in \exp(\hat\t_{1}^\Gamma(\kk))$ and $v\in \exp(\hat\t_{2}^\Gamma(\kk))$ such that 
\begin{itemize}
\item $G(E^{0,1_{\bar0}})=u L^{0,1_{\bar0}}$, with $u=e^{\mu_1 t_{01}}$ for some $\mu_1\in\kk$, necessarily; 
\item $G(\Psi^{0,1_{\bar0},2_{\bar0}})=v b^{0,1_{\bar0},2_{\bar0}}$. 
\end{itemize}
These elements must satisfy the following relations, that are images of \eqref{eqn:tRP}, \eqref{eqn:tMP} 
and \eqref{eqn:tO}, respectively: 
\begin{flalign}
& v^{0,1,2}u^{0,12}(\bar1,\bar1)\cdot(v^{0,1,2})^{-1}=u^{0,1}(\bar1,\bar0)\cdot u^{01,2} \quad \big(\text{in }\on{exp}(\hat{\t}_2^\Gamma(\kk))\big),
\label{eqn:otherpentagon} \\
& v^{01,2,3}v^{0,1,23} = v^{0,1,2}v^{0,12,3}\varphi^{1,2,3} \quad \big(\text{in } \on{exp}(\hat{\t}_3^\Gamma(\kk)) \big),
\label{eqn:pseudotwistbis} \\
& u^{01,2}=v^{0,1,2}e^{\frac{\mu}{2}t_{12}^0}(v^{0,2,1})^{-1}u^{0,2}
\alpha\cdot\big(v^{0,2,1}e^{\frac{\mu}{2}t_{12}^0}(v^{0,1,2})^{-1}\big) \quad\big(\text{in }\on{exp}(\hat{\t}_2^\Gamma(\kk))\big). & 
\label{eqn:octogonbis}
\end{flalign}
\begin{lemma}\label{RP:PACD}
Equation \eqref{eqn:otherpentagon} is satisfied for arbitrary $u$ and $v$. 
\end{lemma}
\begin{proof}
First of all, observe that the diagonal action of $\Gamma$ on $\t_n^\Gamma(\kk)$ is trivial. 
Hence $(\bar1,\bar1)\cdot v^{0,1,2}=v^{0,1,2}$. 
Second of all, recall the formul\ae~for the moperadic structure on $\t^\Gamma(\kk)$: 
\begin{equation}\label{eqn:centralelement}
(t_{01})^{0,12}=t_{01}+t_{02}+\sum_{\gamma\in\Gamma}t_{12}^\gamma=t_{01}+(t_{01})^{01,2}\,.
\end{equation}
Therefore, $(\bar1,\bar0)\cdot u^{01,2}=u^{01,2}$, and $u^{0,12}=u^{0,1}u^{01,2}$. 
Finally, the above element \eqref{eqn:centralelement} is central in $\t_2^\Gamma(\kk)$, thus 
so is $u^{0,12}=u^{0,1}u^{01,2}$, and equation \eqref{eqn:otherpentagon} is satisfied. 
\end{proof}
Considering that we have a Lie algebra isomorphism 
\[
\t_{2}^\Gamma(\kk) \simeq \kk c \oplus \f_{N+1}(\kk)\,,
\] 
where $c=t_{01}+t_{02} + \sum_{a\in\Gamma} t_{12}^a$, then $v$ is of the form 
$e^{\mu_2 c}\psi(t_{01},t_{12}^0,...,t_{12}^{N-1})$ for some $\mu_2\in\kk$. 
\begin{lemma}\label{lambdamu}
We have $\mu_1=\frac{\mu}{N}$. In particular, $u=e^{\frac{\mu}{N}t_{01}}$. 
\end{lemma}
\begin{proof}[Proof of the Lemma]
Denote by $\varepsilon:\mathfrak{t}_{n}^\Gamma \longrightarrow \mathfrak{t}_{n - 1}^0$ 
the Lie algebra morphism sending $t_{0i}$ to $0$, and $t_{ij}^\alpha$ to $\frac{1}{N}t_{i-1,j-1}$ if $0<i<j$. 
We have $\varepsilon (u^{0, 1, 2}) = e^{s t_{01}}\in\exp(\hat{\mathfrak{t}}_{1}^1)$, 
for some $s \in \kk$. 
Now, the image of the mixed pentagon relation \eqref{eqn:pseudotwistbis} by $\varepsilon$ yields:
\begin{equation}\label{reltr1}
  u^{0, 1, 2} e^{s (t_{01} + t_{02})} 
	= e^{s t_{01}} e^{s (t_{02} +t_{12})} \varphi^{0, 1, 2}\,.
\end{equation}
Moreover, as $[t_{01}, t_{02} + t_{12}] = 0$, we further have 
$e^{s t_{01}} e^{s(t_{02} + t_{12})} = e^{s (t_{01} + t_{02} + t_{12})}$. 
Thus equation \eqref{reltr1} is equivalent to
\begin{equation}
  \psi^{0, 1, 2} 
	= e^{s (t_{01} + t_{02} + t_{12})} \varphi^{0, 1, 2} e^{-s(t_{01} + t_{02})}
	= \varphi^{0, 1, 2} e^{st_{12}}\,,
\end{equation}
where we have used that $t_{01} + t_{02} + t_{12}$ is central in $\t_2^0(\kk)$. 
Now, we consider the image of equation \eqref{eqn:octogonbis} in 
$\bar{\t}_2^0=\t_2^0/(t_{01} + t_{02} + t_{12})$. 
Using that $\psi^{0, 1, 2}=\varphi^{0, 1, 2} e^{st_{12}}$ we have 
\begin{eqnarray*}
  &  & e^{-\mu_1t_{01}} = \varphi^{0, 1, 2} e^{s t_{12}} e^{\frac{\mu}{2}
  \frac{t_{12}}{N}} (\varphi^{0, 2, 1} e^{s t_{12}})^{- 1} e^{\mu_1 t_{02}} \varphi^{0, 2,
  1} e^{s t_{12}} e^{\frac{\mu}{2}\frac{t_{12}}{N}} (\varphi^{0, 1, 2} e^{s
  t_{12}})^{- 1} \\
  & \Leftrightarrow & 1 = e^{\mu_1 t_{01}} \varphi^{0, 1, 2}
  e^{\frac{\mu}{2N} t_{12}} (\varphi^{0, 2, 1})^{- 1} e^{\mu_1 t_{02}}
  \varphi^{0, 2, 1} e^{\frac{\mu}{2N} t_{12}} (\varphi^{0, 1, 2})^{- 1} \\
	& \Leftrightarrow & 1 = e^{\frac{\mu_1}{2} t_{01}} \varphi^{0, 1, 2} e^{\frac{\mu}{2N} t_{12}}
   (\varphi^{0, 2, 1})^{- 1} e^{\frac{\mu_1}{2} t_{02}} e^{\frac{\mu_1}{2} t_{02}}
   \varphi^{0, 2, 1} e^{\frac{\mu}{2N} t_{12}} (\varphi^{0, 1, 2})^{- 1}
   e^{\frac{\mu_1}{2} t_{01}} \\
	& \Leftrightarrow & 1 = e^{\frac{N\mu_1-\mu}{2N} t_{01}}\varphi^{1,0,2}e^{\frac{N\mu_1-\mu}{2N} t_{02}}
	e^{\frac{N\mu_1-\mu}{2N} t_{02}}\psi^{2,0,1} e^{\frac{N\mu_1-\mu}{2N} t_{01}} 
\end{eqnarray*}
where for the last equivalence we have used the hexagon relation for the couple $(\mu, \varphi)$ twice. 
This is now equivalent to 
\[
1 = \varphi^{1, 0, 2} e^{(\mu_1-\frac{\mu}{N}) t_{02}} \varphi^{2, 0, 1} e^{(\mu_1-\frac{\mu}{N}) t_{01}}
\]
Sending the generator $t_{02}$ to $0$, we get that $e^{(\mu_1-\frac{\mu}{N})t_{01}}=1$, and thus $\mu_1 = \frac{\mu}{N}$.
\end{proof}
\noindent\textit{End of the proof of the Theorem.} 
Finally, relation \eqref{eqn:tU} is equivalent to 
$$
e^{\mu_2\partial_1(c)} \partial_1\psi(t_{01},t_{12}^0,\ldots,t_{12}^{N-1}) = 1\,,
$$
where we consider the restriction operator $\partial_1 : \t_2^\Gamma(\kk) \to \t_1^\Gamma(\kk)$. 
We have $\partial_1(t_{02}) =t_{01}$ and $\partial_1(t_{01}) = \partial_1(t^\alpha_{12}) = 0$. 
Therefore this equation reduces to $e^{\mu_2\partial_1(c)}=1$ which implies $\mu_2=0$. 
We conclude that $v=\psi(t_{01},t_{12}^0,\ldots,t_{12}^{N-1})$.
This automatically shows that relations \eqref{eqn:pseudotwistbis} and \eqref{eqn:pseudotwist} are equivalent. 
Now, in order to show that relation
\eqref{eqn:octogonbis} is equivalent to relation \eqref{eqn:octogon}, we notice that $[c,-]=0$ and that 
$\psi(t_{01},t_{12}^0,...,t_{12}^{N-1})$ has no component in weight 1 so one can collect the factors 
$e^{\frac{\mu}{2}c}$ in equation \eqref{eqn:octogonbis}. 
The element $u^{01,2}$ is of the form $e^{\frac{\mu}{N}(t_{02}+ \sum_{a\in\Gamma} t_{12}^a)}$. 
Then, by noticing that $-t_{01}=t_{02}+ \sum_{a\in\Gamma} t_{12}^a$ in $\bar{\t}_2^\Gamma(\kk)$, 
we obtain the equivalence between \eqref{eqn:octogonbis} and \eqref{eqn:octogon}.
\end{proof}
\begin{remark}
The set $\on{Ass}^\Gamma(\kk)$ is denoted $\mathbf{Pseudo}_{\bar{1}}(N,\kk)$ in \cite{En}. 
One has more generally a set $\mathbf{Pseudo}_{\gamma}(N,\kk)$ for every $\gamma\in(\mathbb{Z}/N\mathbb{Z})^\times$: 
one just has to replace $\alpha=(\bar0,\bar1)$ with $(\bar0,\gamma)$ in the definition appearing in the statement of 
Theorem \ref{Ass:cyc:iso}. 
A variation on the proof of Theorem \ref{Ass:cyc:iso} shows that $\mathbf{Pseudo}_{\gamma}(N,\kk)$ can be identified with 
$\Gamma$-equivariant isomorphisms between the $\widehat{\PaB}(\KK)$-moperad $\widehat{\PaB}^{\Gamma}(\KK)$ and the 
$G \PaCD(\KK)$-moperad $G\PaCD^{\Gamma}(\KK)$ which, on objects, is the global relabelling given by the automorphism of 
$\Gamma$ sending $\bar1$ to $\gamma$. 
\end{remark}
\begin{example}[Cyclotomic KZ Associator]
Consider the differential equation 
\begin{equation} \label{eq:H2}
 { {\on{d}} \over {\on{d}z} } H(z) 
= \left( {{t_{01}}\over z}
+ \sum_{\alpha\in\Z/N\Z}\frac{t_{12}^\alpha}{z-\zeta^\alpha}\right) H(z), 
\end{equation}
where $\zeta$ is a primitive $N$th root of unity, and let $H_{0^{+}},H_{1^{-}}$ be the solutions such that $H_{0^{+}}(z)
\sim z^{t_{01}}$ when $z\to 0^{+}$ and $H_{1^{-}}(z)
\sim z^{t^1_{12}}$ when $z\to 1^{-}$. Then, in our convention,  the renormalized holonomy $\text{\gls{KZAssg}} = H_{0^{+}}H_{1^{-}}^{-1} \in \on{exp}(\hat{\bar{\t}}_2^\Gamma(\kk)$ 
from $0$ to $1$ of the above differential equation is the cyclotomic KZ associator constructed by Enriquez in \cite{En}. 
More precisely, Enriquez showed that the triple $(2 i \pi,\Phi_{\on{KZ}}, \Psi_{\on{KZ}})$ is 
in $\mathbf{Pseudo}_{-\bar1}(N,{\C})$.
\end{example}

% 5.4

\subsection{Graded cyclotomic Grothendieck-Teichm\"uller groups}

\begin{definition}
The \textit{graded cyclotomic Grothendieck--Teichm\"uller group} is the group 
\[
\text{\gls{GRTg}}:=\on{Aut}_{\on{Mop}}^+\big(\PaCD(\kk),\PaCD^{\Gamma}(\KK)\big)^{\Gamma}
\]
of $\Gamma$-equivariant automorphisms of $\big(\PaCD(\KK),\PaCD^{\Gamma}(\KK)\big)$ which 
are the identity on objects. 
\end{definition}
In the rest of this subsection, we again compare our operadic definition 
with the ones given by Enriquez in \cite{En}.  
\begin{definition}
Define $\on{GRT}_{(\bar 1,1)}^\Gamma(\kk)$ as the set of pairs $(g,h)$, 
with $g\in \on{GRT}_1(\kk)$ and 
$h\in \exp(\hat{\bar{\t}}_{2}^{\Gamma}(\kk))$, such that
\begin{equation}\label{def:grt:cyc:1} 
h^{0,1,2}(h^{0,2,1})^{-1}h(t_{02}| t^{1}_{12},t^{0}_{12},\ldots,
t^{2-N}_{12})h(t_{01}|t^{1}_{12},\ldots,t^{N}_{12})^{-1}=1 
\quad\textrm{in}~\exp\big(\hat{\bar{\t}}_{2}^{\Gamma}(\kk)\big) \,,
\end{equation}
\begin{equation}\label{def:grt:cyc:2} 
t_{01} + \sum_{a=0}^{N-1} \Ad(h(t_{01}|t^a_{12},\ldots,
t^{a+N-1}_{12}))(t^a_{12})+ \Ad\left( h^{0,1,2}(h^{0,2,1})^{-1} \right)(t_{02}) = 0
\quad\textrm{in}~\hat{\bar{\t}}_{2}^{\Gamma}(\kk)\,, 
\end{equation}
\begin{equation}\label{def:grt:cyc:3} 
h^{01,2,3} h^{0,1,23} =h^{0,1,2}  h^{0,12,3}  g^{1,2,3}
\quad\textrm{in}~\exp(\hat \t_{3}^{\Gamma}(\kk))\,.
\end{equation}
One can show that $\on{GRT}^\Gamma_{(\bar 1,1)}(\kk)$ is a group when equipped with the product 
$$
(g_1,h_1) * (g_2,h_2) = (g,h)\,,
$$
where
\begin{itemize}
\item $g(t_{12},t_{23}) =  g_1(t_{12},\Ad g_2(t_{12},t_{23})(t_{23}))g_2(t_{12},t_{23}), $
\item $h^{0,1,2} =  h_1\left( t_{01} | \Ad((h_2^{0,1,2}))(t_{12}^0),\ldots,
\Ad(h_2(t_{01}| t_{12}^{N-1},\ldots,t_{12}^{2N-2}))(t_{12}^{N-1}) \right) h_2^{0,1,2}.$
\end{itemize}
The action of $(\Z/N\Z)^\times \times \kk^\times$ by automorphisms 
of $\t_{3}^{\Gamma}$ (resp. $\t_3$) given by $(c,\gamma)\cdot t_{0i}=\gamma t_{0i},
(c,\gamma)\cdot t^{\alpha}_{ij}=\gamma t^{c\alpha}_{ij}$ 
(resp. $(c,\gamma)\cdot t_{ij}=\gamma t_{ij}$) induces its action by automorphisms of 
$\on{GRT}_{(\bar 1,1)}^\Gamma(\kk)$. We denote by $\on{GRT}^\Gamma(\kk)$ 
the corresponding semidirect product, and 
$\text{\gls{bGRTg}}=\on{GRT}_{(\bar 1,1)}^\Gamma(\kk) \rtimes \kk^\times$.
\end{definition}

\begin{remark}
It is very likely that there should be an analog of \cite[Proposition 5.7]{DrGal}: we believe that 
equations \eqref{def:grt:cyc:1} and \eqref{def:grt:cyc:3} should imply equation \eqref{def:grt:cyc:2}. 
\end{remark}

\begin{proposition}\label{GRT:cyc:cor}
There is an injective group morphism $\GRT^\Gamma(\KK) \to \on{GRT}_{\bar 1}^\Gamma(\kk)$.
\end{proposition}
\begin{proof}
Let $(G,H)$ be an element in $\GRT^{\Gamma}(\KK)$. We have
\begin{itemize}
\item $G(X^{1,2})=X^{1,2}$,
\item $G(H^{1,2})=\lambda H^{1,2}$,
\item $G(a^{1,2,3})=g(t_{12},t_{23}) a^{1,2,3}$,
\item $H(b^{0,1_{\bar0},2_{\bar0}})= w b^{0,1_{\bar0},2_{\bar0}}$,
\item $H(K^{0,1_{\bar0}})=\mu K^{0,1_{\bar0}}$,
\item $H(L^{0,1_{\bar0}})= L^{0,1_{\bar0}}$,
\end{itemize}
where $(\lambda,g)\in \on{GRT}(\kk)$ and $(\mu,w) \in \kk^{\times} \times \on{exp}\big(\hat{\t}^\Gamma_2(\kk)\big)$. 
Borrowing the notation from the previous subsection, let us write 
$w=e^{\nu c}h(t_{01},t_{12}^0,\dots,t_{12}^{N-1})$, 
with $h\in\hat{\on{F}}_{N+1}(\kk)\simeq \on{exp}\big(\hat{\bar{\t}}^\Gamma_2(\kk)\big)$. 
First of all, observe that one can show, along the same lines as in the previous subsection\footnote{See the proof 
of Lemma \ref{lambdamu}, and the end of the proof of Theorem \ref{Ass:cyc:iso}. }, that $\mu=\lambda$ and $\nu=0$. 
Second of all, \eqref{def:grt:cyc:3} follows directly from \eqref{array3}, and \eqref{def:grt:cyc:1} follows directly from \eqref{array1bis}. 
Finally, \eqref{array4} implies that
\[
t_{02}+ \sum_{a=0}^{N-1} t^a_{12} =  \Ad\left( h^{0,1,2}(h^{0,2,1})^{-1} \right)(t_{02})  
	+ \sum_{a=0}^{N-1} \Ad\left(h(t_{01}|t^a_{12},\ldots,t^{a+N-1}_{12})\right)(t^a_{12})\,,
\]
and since $t_{02}+ \sum_{a=0}^{N-1} t^a_{12}=-t_{01}$ in $\hat{\bar{\t}}_{2}^{\Gamma}(\kk)$, we get \eqref{def:grt:cyc:2}.

\medskip

As a consequence of the above discussion, the assignment 
\begin{equation}\label{eqn:GRT-to-GRT}
(G,H)\longmapsto \big(\lambda, g(t_{12},t_{23}), h(t_{01}|t^{0}_{12},\ldots,t^{N-1}_{12})\big)
\end{equation}
defines a map $\GRT^\Gamma(\KK) \to \on{GRT}_{\bar 1}^\Gamma(\kk)$, that is obviously injective. 
It remains to prove that the composition of automorphisms in $\GRT^{\Gamma}(\KK)$
corresponds to the composition law of the group $\on{GRT}_{\bar 1}^\Gamma(\kk)$. 
We already know that the composition of automorphisms $G_1$ and $G_2$ in $\GRT(\KK)$
corresponds to the composition law in $\on{GRT}(\kk)$, that is, the associated 
couples $(\lambda_1,g_1)$ and $(\lambda_2,g_2)$ in $\kk^\times \times \on{exp}(\hat{{\t}}_{3}(\kk))$ satisfy 
\begin{itemize}
\item $(G_1 \circ G_2)(H^{1,2})=\lambda_1\lambda_2 H^{1,2}$,
\item $(G_1 \circ G_2)(a^{1,2,3})=  g_2 \left(\lambda_1 t_{12}, g_1(t_{12},t_{23}) 
(\lambda_1 t_{23})g_1(t_{12},t_{23})^{-1}\right) g_1(t_{12},t_{23}) a^{1,2,3}$.
\end{itemize}
We also already showed that any two $H_1$ and $H_2$ such that $(G_1,H_1)$ and $(G_2,H_2)$ 
lie in $\GRT^{\Gamma}(\KK)$ are determined by elements
$h_1(t_{01}|t^{0}_{12},\ldots,t^{N-1}_{12})$ and $h_2(t_{01}|t^{0}_{12},\ldots,t^{N-1}_{12})$
which represent automorphisms of the parenthesized word $(01_{\bar0})2_{\bar0}$ 
in the groupoid $G\PaCD^\Gamma(\kk)(2)$. 
Note that the group $\on{Aut}_{G\PaCD^\Gamma(\kk)(3)}((01_{\bar0})2_{\bar0})$ is canonically identified with 
$\exp\big(\hat{\t}_2^\Gamma\big)$. Within this identification, $t_{01}=K^{0,1_{\bar0}}$ for instance, but 
$t_{12}^0=b^{0,1_{\bar0},2_{\bar0}}H^{1_{\bar0},2_{\bar0}}(b^{0,1_{\bar0},2_{\bar0}})^{-1}$. 
Therefore 
\[
H_i(t_{01})=\lambda_i t_{01}
\qquad\textrm{and}\qquad
H_i(t_{12}^0)=\on{Ad}\big(h_i(t_{01}|t^{0}_{12},\ldots,t^{N-1}_{12})\big)(\lambda_i t_{12}^0)\,.
\]
More generally, 
$t_{12}^a=(L^{0,1_{\bar0}})^{(a)}b^{0,1_{\bar{a}},2_{\bar0}}H^{1_{\bar{a}},2_{\bar{a}}}
(b^{0,1_{\bar{a}},2_{\bar0}})^{-1}(L^{0,1_{\bar0}})^{(-a)}$, 
and thus 
\[
H_i(t_{12}^a)=\on{Ad}\big(h_i(t_{01},t^a_{12},\ldots,t^{\alpha+N-1}_{12})\big)(\lambda_i t_{12}^a)\,.
\]
Therefore, we compute 
\begin{eqnarray*} 
(H_1 \circ H_2)(b^{0,1_{\bar0},2_{\bar0}})&=& H_1\big(h_2(t_{01},t^{0}_{12},\ldots,t^{N-1}_{12})b^{0,1_{\bar0},2_{\bar0}}\big) \\
& = &	h_2\big(H_1(t_{01}),H_1(t^{0}_{12}),\ldots,H_1(t^{N-1}_{12})\big)
		h_1(t_{01},t^{0}_{12},\ldots,t^{N-1}_{12}) b^{0,1_{\bar0},2_{\bar0}} \\
& = &	h_2\Big(\lambda_1 t_{01},\on{Ad}\big(h_1(t_{01},t^0_{12},\ldots,t^{N-1}_{12})\big)(\lambda_1 t_{12}^0),\ldots\\
&   &	\ldots,\on{Ad}\big(h_1(t_{01},t^{N-1}_{12},\ldots,t^{2N-2}_{12})\big)(\lambda_1 t_{12}^{N-1})\Big)
		h_1(t_{01}|t^{0}_{12},\ldots,t^{N-1}_{12}) b^{0,1_{\bar0},2_{\bar0}}\,,
\end{eqnarray*}
which is nothing but the composition law in the group 
$\on{GRT}_{\bar 1}^\Gamma(\kk)$. This concludes the proof, 
as the composite of moperad morphisms $H_1 \circ H_2$ 
is compatible with the composition of operad morphisms $G_1 \circ G_2$.
\end{proof}
In the next section we will show, among other things, that this injective morphism is actually an isomorphism. 
We could prove surjectivity by proving that the relations from Remark \ref{PaCD:cyc:rel} completely determine 
$\PaCD^{\Gamma}(\KK)$ (which we believe is true), but we use instead the torsor structure.

\subsection{Bitorsors}
Our main goal in this final section is to promote the one-to-one correspondence from Theorem \ref{Ass:cyc:iso} to a bitorsor isomorphism. 

\medskip

Recall from \cite{En} that the group $\on{GTM}_{\bar 1}(N,\kk)$ acts freely and transitively on $\on{Ass}^\Gamma(\KK)$ 
from the left, in the following manner:  
\[
(\lambda,f,g) * (\lambda',\varphi',\psi') = (\lambda\lambda',\varphi'',\psi'')\,,  
\]
where $\varphi''(t_{12},t_{23}) := \varphi'(t_{12},t_{23}) f(e^{\lambda' t_{12}}, 
\Ad(\varphi'(t_{12},t_{23}))(e^{\lambda't_{23}}))$, and 
\begin{eqnarray*}
 \psi''(t^0_{12} | t^{0}_{23},\ldots,t^{N-1}_{23}) & := &
\psi'(t^0_{12} | t^{0}_{23},\ldots,t^{N-1}_{23})
\\ & &
g \Big(e^{ \lambda't^0_{12}} | 
\on{Ad}\big( \psi'(t^0_{12}  | t^{0}_{23},\ldots,t^{N-1}_{23}) \big) 
(e^{ \lambda' t^{0}_{23}}), \\ 
& & \Ad\big(e^{ (\lambda'/N)t^0_{12}} \psi'(t^0_{12} | 
t^{1}_{23}, \ldots, t^{N}_{23})\big)
(e^{ \lambda't^{1}_{23}}), \ldots \\ 
&  &
\ldots,\Ad\big((e^{ ((N-1)\lambda'/N)t^0_{12}}\psi'(t^0_{12} | 
t^{N-1}_{23}, \ldots, t^{2N-2}_{23})\big)
(e^{ \lambda' t^{N-1}_{23}}) \Big) 
\end{eqnarray*}

Enriquez also constructed \cite{En} a free and transitive right action of $\on{GRT}_{\bar 1}^\Gamma(\kk)$ 
on $\on{Ass}^{\Gamma}(\kk)$, which commutes with the above left action of $\on{GTM}_{\bar 1}(N,\kk)$, thus turning the triple 
\[
\big(\on{GTM}_{\bar 1}(N,\kk),\on{Ass}^\Gamma(\kk),\on{GRT}_{\bar 1}^\Gamma(\kk)\big)
\]
into a bitorsor.

For the sake of completeness, let us recall how this right action is defined. 
For every $\mu\in \kk^\times$, the group $\on{GRT}_{(\bar 1,1)}^\Gamma(\kk)$ acts on
\[
\on{Ass}_{\mu}^{\Gamma}(\kk):=\{ (\varphi,\psi )\in\on{exp}(\hat\t_3^0(\kk)) 
\times \on{exp}(\hat\bar{\t}_2^\Gamma(\kk))  ; (\mu,\varphi,\psi )\in  \on{Ass}^{\Gamma}(\kk)\}\,.
\]
from the right by $(\varphi,\psi) * (g,h) = (\varphi',\psi')$, where
\[
\varphi'(t_{12},t_{23}) = \varphi(t_{12},\Ad(g(t_{12},t_{23}))(t_{23}))g(t_{12},t_{23})
\]
and
\begin{eqnarray*} 
 \psi'(t_{01}|t^{0}_{12},\ldots,t^{N-1}_{12}) &=& \psi\Big(t_{01} | \Ad\big(h(t_{01}|t^{0}_{12},\ldots,t^{N-1}_{12})\big)(t^{0}_{12}), \ldots \\ 
 & & \ldots,
\Ad\big(h(t_{01}|t^{N-1}_{12},\ldots,t^{2N-2}_{12})\big)(t^{N-1}_{12}) \Big) h(t_{01}|t^{0}_{12},\ldots,t^{N-1}_{12})\,. 
\end{eqnarray*}
This naturally extends to an action of $\on{GRT}_{\bar 1}^\Gamma(\kk)$ on $\on{Ass}^{\Gamma}(\kk)$, which is compatible 
with the scaling action of $\kk^{\times}$ on $\kk$. 

We already know that, by definition, 
$$
\big(\widehat{\GT}^\Gamma(\kk),\Assoc^\Gamma(\kk),\GRT^\Gamma(\kk)\big)
$$
has a natural bitorsor structure. 
\begin{theorem}\label{thm-last-torsor}
There is a bitorsor isomorphism
\[
\big(\widehat{\GT}^{\Gamma}(\kk),\Assoc^{\Gamma}(\kk),\GRT^{\Gamma}(\kk)\big) \longrightarrow 
\big(\on{GTM}_{\bar1}(N,\kk),\on{Ass}^{\Gamma}(\kk),\on{GRT}_{\bar 1}^\Gamma(\kk)\big)\,.
\]
\end{theorem}

\begin{proof}
This is a summary of most of the above results. Indeed, we proved that 
\begin{itemize}
\item There is an injective group morphism $\widehat{\GT}^\Gamma(\kk)\to\on{GTM}_{\bar1}(N,\kk)$ (Proposition \ref{Cyc:GT});
\item There is an injective group morphism $\GRT^\Gamma(\kk)\to\on{GRT}_{\bar 1}^\Gamma(\kk)$ (Proposition \ref{GRT:cyc:cor}); 
\item There is a bijection $\Assoc^{\Gamma}(\kk)\to \on{Ass}^{\Gamma}(\kk)$ (Theorem \ref{Ass:cyc:iso}). 
\end{itemize}
Hence, in order to conclude it is sufficient to prove that the three above maps take the respective actions 
$\widehat{\GT}^\Gamma(\kk)$ and $\GRT^\Gamma(\kk)$ on $\Assoc^\Gamma(\kk)$, to the ones 
of $\on{GTM}_{\bar1}(N,\kk)$ and $\on{GRT}_{\bar 1}^\Gamma(\kk)$ on $\on{Ass}^\Gamma(\kk)$. 
The proof is similar to the proofs that $\widehat{\GT}^\Gamma(\kk)\to\on{GTM}_{\bar1}(N,\kk)$ and 
$\GRT^\Gamma(\kk)\to\on{GRT}_{\bar 1}^\Gamma(\kk)$ are group morphisms (see the proofs of Propositions \ref{Cyc:GT} 
and \ref{GRT:cyc:cor}), and is left to the reader. 
\end{proof}

\clearpage
\section*{List of notation}
\printnoidxglossaries


\clearpage

\begin{thebibliography}{EnGh}

\bibitem{AS}
S.~Axelrod \& I.~Singer, 
Chern--Simons perturbation theory II, 
Jour. Diff. Geom. \textbf{39} (1994), no. 1, 173--213. 

\bibitem{BN}
D.~Bar-Natan, 
On Associators and the Grothendieck-Teichm\"uller Group I,
Selecta Mathematica New Series \textbf{4} (1998) 183--212. 

\bibitem{BZBJ}
D.~Ben-Zvi, A.~Brochier \& D.~Jordan, 
Quantum character varieties and braided module categories, 
Selecta Mathematica \textbf{24} (2018), 4711--4748. 

\bibitem{Adrien}
A.~Brochier, 
Cyclotomic associators and finite type invariants for tangles in the solid torus, 
Algebraic \& Geometric Topology, \textbf{13} (2013), 3365--3409.

\bibitem{CaGo2}
D.~Calaque \& M.~Gonzalez, 
Ellipsitomic associators, 
M\'emoires de la SMF \textbf{179} (2023), 96 pages. 

\bibitem{DHLARS}
Z.~Dancso, I.~Halacheva, G.~Laplante-Anfossi, M.~Robertson \& C.~Singh, 
Genus zero Kashiwara--Vergne solutions from braids, 
preprint \href{https://arxiv.org/abs/2507.16243}{arXiv:2507.16243} (2025). 
 
\bibitem{DrGal}
V.~Drinfeld, 
On quasitriangular quasi-Hopf algebras and a group closely connected 
with $\on{Gal}(\bar\Q/\Q)$, 
Leningrad Math. J. \textbf{2} (1991), 829--860.

\bibitem{En}
B.~Enriquez, 
Quasi-reflection algebras and cyclotomic associators, 
Selecta Mathematica (NS) \textbf{13} (2008), no. 3, 391--463. 

\bibitem{En2}
B.~Enriquez, 
Elliptic associators, 
Selecta Mathematica (NS) \textbf{20} (2014), no. 2, 491--584. 

\bibitem{FreMod}
B.~Fresse, 
\textit{Modules over operads and functors}, 
Lecture Notes in Mathematics, \textbf{1967} (2009), Springer, 314pp. 

\bibitem{Fresse}
B.~Fresse, 
\textit{Homotopy of Operads and Grothendieck–Teichm\"uller Groups: 
Part 1: The Algebraic Theory and its Topological Background}, 
Mathematical Surveys and Monographs \textbf{217} (2017), AMS, 532 pp. 

\bibitem{FMP}
W.~Fulton \& R.~MacPherson, 
A compactification of configuration spaces, 
Annals of Mathematics \textbf{139} (1994), 183--225. 

\bibitem{G1}
M.~Gonzalez, 
\textit{Contributions to the theory of associators}, 
PhD Thesis, \href{https://theses.hal.science/tel-02865424/}{tel-02865424} (2018). 

\bibitem{Esquisse}
A.~Grothendieck, 
Esquisse d'un programme, 
in \textit{Geometric Galois Actions} (L.~Schneps L \& P.~Lochak eds.), 
London Mathematical Society Lecture Note Series, Cambridge University Press,  1997, 7--48. 

\bibitem{HainMalcev}
R.~Hain, 
The Hodge de Rham theory of relative Malcev completion, 
Annales scientifiques de l'\'Ecole Normale Sup\'erieure 
(S\'erie 4) \textbf{31} (1998), no. 1 , 47--92. 

\bibitem{LNS-new}
P.~Lochak, H.~Nakamura \& L.~Schneps, 
On a new version of the Grothendieck-Teichm\"uller group, 
CRAS - Series I - Mathematics \textbf{325} (1997), no. 1, 11--16. 

\bibitem{Sos}
A.B.~Sossinsky, 
Preparation theorems for isotopy invariants of links in 3-manifolds, 
in \textit{Quantum Groups} (P.~Kulish, ed.), Lecture Notes in Math. 
\textbf{1510} (1992), Springer-Verlag, 354--362.

\bibitem{Tam}
D.~Tamarkin, 
Formality of chain operad of little disks, 
Letters in Mathematical Physics \textbf{66} (2003), 65--72. 

\bibitem{Ver}
V.~Vershinin, 
On braid groups in handlebodies, 
Siberian Math. J. \textbf{39} (1998), no. 4, 645--654.

\bibitem{Will16}
T.~Willwacher, 
The homotopy braces formality morphism, 
Duke Math. J. 165, \textbf{10} (2016), 1815--1964. 

\end{thebibliography}
\end{document}